\documentclass[11pt]{amsart}

\usepackage{
    amsmath,
    amsfonts,
    amssymb,
    amsthm,
    amscd,
    comment,
    enumitem,
    etoolbox,
    textcomp,   
    gensymb,    
    mathtools,
    mathdots,
    stmaryrd    
}
\usepackage[dvipsnames]{xcolor}
\usepackage{tipa}

\usepackage[letterpaper,margin=1in]{geometry}
\usepackage[T1]{fontenc}
\usepackage{lmodern}            
\usepackage{dsfont}             
\usepackage[
    unicode=true,
    pdfencoding=auto,
    psdextra,
    colorlinks=true,
    linkcolor=blue,
    citecolor=blue,
    urlcolor=blue,
    breaklinks=true,
    pdfdisplaydoctitle=true
]{hyperref}

\DeclareFontFamily{OT1}{pzc}{}
\DeclareFontShape{OT1}{pzc}{m}{it}{<-> s * [1.10] pzcmi7t}{}
\DeclareMathAlphabet{\mathpzc}{OT1}{pzc}{m}{it}

\usepackage{zref-clever}

\zcsetup{
  cap,                      
  nameinlink=false,         
  lastsep = {, and }        
}

\zcRefTypeSetup{assumption}{
    Name-sg = Assumption ,
    name-sg = assumption ,
    Name-pl = Assumptions ,
    name-pl = assumptions ,
}
\zcRefTypeSetup{cor}{
    Name-sg = Corollary ,
    name-sg = corollary ,
    Name-pl = Corollaries ,
    name-pl = corollaries ,
}
\zcRefTypeSetup{defin}{
    Name-sg = Definition ,
    name-sg = definition ,
    Name-pl = Definitions ,
    name-pl = definitions ,
}
\zcRefTypeSetup{enumi}{
    Name-sg = {} ,
    name-sg = {} ,
    Name-pl = {} ,
    name-pl = {} ,
}
\zcRefTypeSetup{equation}{
    Name-sg = {} ,
    name-sg = {} ,
    Name-pl = {} ,
    name-pl = {} ,
}
\zcRefTypeSetup{eg}{
    Name-sg = Example ,
    name-sg = example ,
    Name-pl = Examples ,
    name-pl = examples ,
}
\zcRefTypeSetup{lem}{
    Name-sg = Lemma ,
    name-sg = lemma ,
    Name-pl = Lemmas ,
    name-pl = lemmas ,
}
\zcRefTypeSetup{prop}{
    Name-sg = Proposition ,
    name-sg = proposition ,
    Name-pl = Propositions ,
    name-pl = propositions ,
}
\zcRefTypeSetup{rem}{
    Name-sg = Remark ,
    name-sg = remark ,
    Name-pl = Remarks ,
    name-pl = remarks ,
}
\zcRefTypeSetup{theo}{
    Name-sg = Theorem ,
    name-sg = theorem ,
    Name-pl = Theorems ,
    name-pl = theorems ,
}

\zcsetup{
  countertype = {
    enumi=item, enumii=item, enumiii=item, enumiv=item
  },
  counterresetby = {
    enumii=enumi, enumiii=enumii, enumiv=enumiii
  }
}
\zcRefTypeSetup{item}{
    Name-sg = {} ,
    name-sg = {} ,
    Name-pl = {} ,
    name-pl = {} ,
}

\AddToHook{env/assumption/begin}{\zcsetup{countertype={theo=assumption}}}
\AddToHook{env/cor/begin}{\zcsetup{countertype={theo=cor}}}
\AddToHook{env/defin/begin}{\zcsetup{countertype={theo=defin}}}
\AddToHook{env/eg/begin}{\zcsetup{countertype={theo=eg}}}
\AddToHook{env/lem/begin}{\zcsetup{countertype={theo=lem}}}
\AddToHook{env/prop/begin}{\zcsetup{countertype={theo=prop}}}
\AddToHook{env/rem/begin}{\zcsetup{countertype={theo=rem}}}

\newcommand{\ceqref}[1]{
    \mathrel{\mathop{=}\limits^{\text{\cref{#1}}}}
}
\newcommand{\ceqsref}[2]{
    \mathrel{\mathop{=}\limits^{\text{\cref{#1}}}_{\text{\cref{#2}}}}
}

\makeatletter

\newlength{\eqalignref@lab}   
\newlength{\eqalignref@tmp}

\newcommand{\eqalignref@measure}[1]{%
  \settowidth{\eqalignref@tmp}{\ensuremath{\scriptstyle\text{\cref{#1}}}}%
  \ifdim\eqalignref@tmp>\eqalignref@lab \global\eqalignref@lab=\eqalignref@tmp \fi
}

\newcommand{\eqalignref@finish}{%
  \settowidth{\eqalignref@tmp}{\ensuremath{=}}%
  \global\advance\eqalignref@lab -\eqalignref@tmp
  \ifdim\eqalignref@lab<\z@ \global\eqalignref@lab=\z@ \fi
  \global\eqalignref@lab=.5\eqalignref@lab
}

\newcommand{\eqalignref@pad}[1]{\mspace{#1}\hspace{\eqalignref@lab}}

\newcommand{\eqalignref}[2][3mu]{%
  \global\eqalignref@lab=\z@ \eqalignref@measure{#2}\eqalignref@finish
  \eqalignref@pad{#1}
  &\mathrel{\mathop{=}\limits^{\mathclap{\text{\cref{#2}}}}}%
  \eqalignref@pad{#1}
}

\newcommand{\eqalignsref}[3][3mu]{%
  \global\eqalignref@lab=\z@ \eqalignref@measure{#2}\eqalignref@measure{#3}\eqalignref@finish
  \eqalignref@pad{#1}%
  &\mathrel{\mathop{=}\limits^{\mathclap{\text{\cref{#2}}}}
                              _{\mathclap{\text{\cref{#3}}}}}%
  \eqalignref@pad{#1}%
}
\makeatother

\newcommand{\cref}[1]{\zcref{#1}}
\newcommand{\Cref}[1]{\zcref[S]{#1}}

\makeatletter
\def\namedlabel#1#2{\begingroup
    (#2)%
    \def\@currentlabel{(#2)}%
    \phantomsection\label{#1}\endgroup
}
\makeatother

\newcommand\kk{\Bbbk}
\newcommand\one{\mathds{1}}

\newcommand\bA{\mathbf{A}}

\newcommand\cC{\mathcal{C}}
\newcommand\cD{\mathcal{D}}

\newcommand\cM{\mathcal{M}}
\newcommand\cN{\mathcal{N}}

\newcommand\rev{\textup{rev}}

\newcommand\AOScat{\mathpzc{AOS}}       
\newcommand\AKcat{\mathpzc{AK}}         
\newcommand\cEnd{\mathpzc{End}}         
\newcommand\Dcat{{\color{cyan} \mathcal{D}}}
\newcommand\DScat{\mathpzc{DS}}         
\newcommand\Mcat{{\color{cyan} \mathcal{M}}}    
\newcommand\OScat{\mathpzc{OS}}         
\newcommand\Kcat{\mathpzc{K}}           

\newcommand{\upobj}[1][black]{{\color{#1} \mathord{\uparrow}}}
\newcommand{\downobj}[1][black]{{\color{#1} \mathord{\downarrow}}}
\newcommand{\Kobj}{\mathsf{K}}

\newcommand\bF{\mathbf{F}}
\newcommand\bG{\mathbf{G}}
\newcommand\bH{\mathbf{H}}

\DeclareMathOperator{\Aff}{Aff}

\DeclareMathOperator{\id}{id}
\DeclareMathOperator{\Mor}{Mor}                 

\DeclareMathOperator{\Ob}{Ob}                   

\usepackage{tikz}
\usetikzlibrary{arrows.meta,patterns}
\usetikzlibrary{decorations.markings,decorations.pathreplacing}
\usepackage{tikz-cd}

\tikzset{multi/.style={very thick}}
\tikzset{MC/.style={cyan}}
\tikzset{present/.style={cyan}}
\tikzset{
    centerzero/.style={
        >=To,
        baseline={([yshift=-0.5ex]#1)},
        },
    centerzero/.default={0,0},
    anchorbase/.style={centerzero=current bounding box.center},
}
\tikzset{   
    wipe/.style={
        white,
        line cap = butt,    
        line width=4pt
    }
} 
\tikzset{
    over/.style={
        preaction={
            draw=white,
            line cap = butt,    
            line width=4pt,
            -{},                
        }
    },
}
\tikzset{->-/.style={decoration={
    markings,
    mark=at position #1 with {\arrow{>}}},postaction={decorate}}
}
\tikzset{-<-/.style={decoration={
    markings,
    mark=at position #1 with {\arrow{<}}},postaction={decorate}}
}

\newcommand{\strandlabel}[1]{$\scriptstyle{#1}$}
\newcommand{\botlabel}[1]{node[anchor=north,color=black] {\strandlabel{#1}}}
\newcommand{\botlabelshift}[2]{node[anchor=north,color=black,xshift=#1] {\strandlabel{#2}}}
\newcommand{\toplabel}[1]{node[anchor=south,color=black,xshift=0pt] {\strandlabel{#1}}}
\newcommand{\toplabelshift}[2]{node[anchor=south,color=black,xshift=#1] {\strandlabel{#2}}}
\newcommand{\braidup}{to[out=up,in=down]}
\newcommand{\braiddown}{to[out=down,in=up]}
\newcommand{\coupon}[2]{ 
    \draw (#1) node[minimum height=9pt,minimum width=4pt,inner sep=2pt,draw=black,rectangle,rounded corners,fill=black!20!white] (coup) {\color{black} $\scriptstyle{#2}$}
}
\newcommand{\genbox}[3]{
    \filldraw[fill=black!20!white,rounded corners] (#1) rectangle (#2) node[midway] {\strandlabel{#3}}
}

\newcommand{\posdot}[1]{
    \filldraw[fill=white] (#1)+(0.09,0) arc(0:360:0.09);
    \draw (#1)++(-0.09,0) to ++(0.18,0);
    \draw (#1)++(0,-0.09) -- ++(0,0.18)
}
\newcommand{\negdot}[1]{
    \filldraw[fill=white] (#1)+(0.09,0) arc(0:360:0.09);
    \draw (#1)++(-0.09,0) to ++(0.18,0)
}
\newcommand{\opendot}[1]{
    \node[white] at (#1) {$\scriptstyle{\bullet}$};
    \node at (#1) {$\scriptstyle{\circ}$}
}
\newcommand{\closeddot}[1]{
    \node at (#1) {$\scriptstyle{\bullet}$}
}
\newcommand{\bubright}[1]{
    \draw[->] (#1)+(0.2,0) arc(360:0:0.2)
}

\newcommand{\bubun}[1]{
    \draw (#1)+(0.2,0) arc(0:360:0.2)
}

\newcommand{\poscross}{
    \begin{tikzpicture}[centerzero]
        \draw (0.2,-0.2) -- (-0.2,0.2);
        \draw[over] (-0.2,-0.2) -- (0.2,0.2);
    \end{tikzpicture}
}
\newcommand{\negcross}{
    \begin{tikzpicture}[centerzero]
        \draw (-0.2,-0.2) -- (0.2,0.2);
        \draw[over] (0.2,-0.2) -- (-0.2,0.2);
    \end{tikzpicture}
}
\newcommand{\posupcross}{
    \begin{tikzpicture}[centerzero]
        \draw[->] (0.2,-0.2) -- (-0.2,0.2);
        \draw[over,->] (-0.2,-0.2) -- (0.2,0.2);
    \end{tikzpicture}
}
\newcommand{\negupcross}{
    \begin{tikzpicture}[centerzero]
        \draw[->] (-0.2,-0.2) -- (0.2,0.2);
        \draw[over,->] (0.2,-0.2) -- (-0.2,0.2);
    \end{tikzpicture}
}
\newcommand{\posrightcross}{
    \begin{tikzpicture}[centerzero]
        \draw[->] (-0.2,-0.2) -- (0.2,0.2);
        \draw[over,<-] (0.2,-0.2) -- (-0.2,0.2);
    \end{tikzpicture}
}
\newcommand{\negrightcross}{
    \begin{tikzpicture}[centerzero]
        \draw[<-] (0.2,-0.2) -- (-0.2,0.2);
        \draw[over,->] (-0.2,-0.2) -- (0.2,0.2);
    \end{tikzpicture}
}
\newcommand{\posdowncross}{
    \begin{tikzpicture}[centerzero]
        \draw[<-] (0.2,-0.2) -- (-0.2,0.2);
        \draw[over,<-] (-0.2,-0.2) -- (0.2,0.2);
    \end{tikzpicture}
}
\newcommand{\negdowncross}{
    \begin{tikzpicture}[centerzero]
        \draw[<-] (-0.2,-0.2) -- (0.2,0.2);
        \draw[over,<-] (0.2,-0.2) -- (-0.2,0.2);
    \end{tikzpicture}
}
\newcommand{\negleftcross}{
    \begin{tikzpicture}[centerzero]
        \draw[->] (0.2,-0.2) -- (-0.2,0.2);
        \draw[over,<-] (-0.2,-0.2) -- (0.2,0.2);
    \end{tikzpicture}
}
\newcommand{\posleftcross}{
    \begin{tikzpicture}[centerzero]
        \draw[<-] (-0.2,-0.2) -- (0.2,0.2);
        \draw[over,->] (0.2,-0.2) -- (-0.2,0.2);
    \end{tikzpicture}
}

\newcommand{\upstrand}{
    \begin{tikzpicture}[centerzero]
        \draw[->] (0,-0.25) -- (0,0.25);
    \end{tikzpicture}
}
\newcommand{\downstrand}[1][]{
    \begin{tikzpicture}[centerzero,#1]
        \draw[<-] (0,-0.25) -- (0,0.25);
    \end{tikzpicture}
}

\newcommand{\upopendot}[1][]{
    \begin{tikzpicture}[centerzero,#1]
        \draw[->] (0,-0.25) -- (0,0.25);
        \opendot{0,0};
    \end{tikzpicture}
}
\newcommand{\upcloseddot}[1][]{
    \begin{tikzpicture}[centerzero,#1]
        \draw[->] (0,-0.25) -- (0,0.25);
        \closeddot{0,0};
    \end{tikzpicture}
}
\newcommand{\upposdot}[1][]{
    \begin{tikzpicture}[centerzero,#1]
        \draw[->] (0,-0.25) -- (0,0.25);
        \posdot{0,0};
    \end{tikzpicture}
}
\newcommand{\upnegdot}[1][]{
    \begin{tikzpicture}[centerzero,#1]
        \draw[->] (0,-0.25) -- (0,0.25);
        \negdot{0,0};
    \end{tikzpicture}
}

\newcommand{\downposdot}[1][]{
    \begin{tikzpicture}[centerzero,#1]
        \draw[<-] (0,-0.25) -- (0,0.25);
        \posdot{0,0};
    \end{tikzpicture}
}
\newcommand{\downnegdot}[1][]{
    \begin{tikzpicture}[centerzero,#1]
        \draw[<-] (0,-0.25) -- (0,0.25);
        \negdot{0,0};
    \end{tikzpicture}
}
\newcommand{\Kopendot}{
    \begin{tikzpicture}[centerzero]
        \draw (0,-0.25) -- (0,0.25);
        \opendot{0,0};
    \end{tikzpicture}
}
\newcommand{\Kcloseddot}{
    \begin{tikzpicture}[centerzero]
        \draw (0,-0.25) -- (0,0.25);
        \closeddot{0,0};
    \end{tikzpicture}
}
\newcommand{\togupdown}[1][]{
    \begin{tikzpicture}[centerzero,#1]
        \draw[<->] (0,-0.2) -- (0,0.2);
        \opendot{0,0};
    \end{tikzpicture}
}
\newcommand{\togdownup}[1][]{
    \begin{tikzpicture}[centerzero,#1]
        \draw[->] (0,0.2) -- (0,0);
        \draw[->] (0,-.2) -- (0,0);
        \opendot{0,0};
    \end{tikzpicture}
}

\newcommand{\rightcup}[1][]{
    \begin{tikzpicture}[anchorbase,#1]
        \draw[->] (-0.15,0.15) -- (-0.15,0) arc(180:360:0.15) -- (0.15,0.15);
    \end{tikzpicture}
}
\newcommand{\leftcup}[1][]{
    \begin{tikzpicture}[anchorbase,#1]
        \draw[<-] (-0.15,0.15) -- (-0.15,0) arc(180:360:0.15) -- (0.15,0.15);
    \end{tikzpicture}
}
\newcommand{\rightcap}[1][]{
    \begin{tikzpicture}[anchorbase,#1]
        \draw[->] (-0.15,-0.15) -- (-0.15,0) arc(180:0:0.15) -- (0.15,-0.15);
    \end{tikzpicture}
}
\newcommand{\leftcap}[1][]{
    \begin{tikzpicture}[anchorbase,#1]
        \draw[<-] (-0.15,-0.15) -- (-0.15,0) arc(180:0:0.15) -- (0.15,-0.15);
    \end{tikzpicture}
}
\newcommand{\Kcup}{
    \begin{tikzpicture}[anchorbase]
        \draw (-0.15,0.15) -- (-0.15,0) arc(180:360:0.15) -- (0.15,0.15);
    \end{tikzpicture}
}
\newcommand{\Kcap}{
    \begin{tikzpicture}[anchorbase]
        \draw (-0.15,-0.15) -- (-0.15,0) arc(180:0:0.15) -- (0.15,-0.15);
    \end{tikzpicture}
}

\newcommand{\rightbub}{
    \begin{tikzpicture}[centerzero]
        \bubright{0,0};
    \end{tikzpicture}
}

\newcommand{\unbub}{
    \begin{tikzpicture}[centerzero]
        \bubun{0,0};
    \end{tikzpicture}
}

\newcommand{\Kbub}[1]{
    \begin{tikzpicture}[centerzero]
        \bubun{0,0};
    \end{tikzpicture}
}

\newtheorem{theo}{Theorem}[section]

\newtheorem{prop}[theo]{Proposition}

\newtheorem{cor}[theo]{Corollary}

\theoremstyle{definition}

\newtheorem{defin}[theo]{Definition}
\newtheorem{rem}[theo]{Remark}
\newtheorem{eg}[theo]{Example}

\numberwithin{equation}{section}
\allowdisplaybreaks

\setenumerate[1]{label=(\alph*)}          

\newtoggle{comments}
\newtoggle{details}
\newtoggle{detailsnote}

\iftoggle{comments}{%
    \usepackage[notref,notcite]{showkeys}   
    \newcommand{\acomments}[1]{
        \ \\
        {\color{red}
            \textbf{AS:} #1
        }
        \ \\
    }
    \newcommand{\hcomments}[1]{
        \ \\
        {\color{cyan}
            \textbf{HS:} #1
        }
        \ \\
    }

    \newcommand{\ycomments}[1]{
        \ \\
        {\color{blue}
            \textbf{YS:} #1
        }
        \ \\
    }
    \newcommand{\question}[1]{
        \ \\
        {\color{blue}
            \textbf{Question:} #1
        }
        \ \\
    }
    }{%
        \newcommand{\acomments}[1]{}
        \newcommand{\hcomments}[1]{}
        \newcommand{\ycomments}[1]{}
        \newcommand{\question}[1]{}
    }

\iftoggle{details}{%
    \newcommand{\details}[1]{
        \ \\
        {\color{OliveGreen}
            \textbf{Details:} #1
        }
        \\
    }
}{%
    \newcommand{\details}[1]{\ignorespaces}
}

\begin{document}

\title{Affinizations and braided module categories}

\author{Hadi Salmasian}
\address[H.S.]{
    Department of Mathematics and Statistics \\
    University of Ottawa \\
    Ottawa, ON, K1N 6N5, Canada
}
\urladdr{\href{https://mysite.science.uottawa.ca/hsalmasi/}{mysite.science.uottawa.ca/hsalmasi/}, \textrm{\textit{ORCiD}:} \href{https://orcid.org/0000-0002-1073-7183}{orcid.org/0000-0002-1073-7183}}
\email{hadi.salmasian@uottawa.ca}

\author{Alistair Savage}
\address[A.S.]{
    Department of Mathematics and Statistics \\
    University of Ottawa \\
    Ottawa, ON, K1N 6N5, Canada
}
\urladdr{\href{https://alistairsavage.ca}{alistairsavage.ca}, \textrm{\textit{ORCiD}:} \href{https://orcid.org/0000-0002-2859-0239}{orcid.org/0000-0002-2859-0239}}
\email{alistair.savage@uottawa.ca}

\author{Yaolong Shen}
\address[Y.S.]{
   School of Mathematical Sciences, Key Laboratory of MEA (Ministry of Education)
   \& Shanghai Key Laboratory of PMMP, East China Normal University, Shanghai 200241,
   China
}
\urladdr{\href{https://yaolong-shen-mathematics.yaolongshen3555.chatgpt.site/}{yaolong-shen-mathematics.yaolongshen3555.chatgpt.site/},\newline \indent \textrm{\textit{ORCiD}:} \href{https://orcid.org/0000-0002-8840-3394}{orcid.org/0000-0002-8840-3394}}
\email{ylshen@math.ecnu.edu.cn}

\begin{abstract}
    We relate affinizations of braided monoidal categories to braided module categories, allowing a twist by a braided monoidal automorphism.  Our main theorem establishes a bijection between twisted braided structures on a right module category and extensions of its action to the corresponding twisted affinization.  We also show that the twisted affinization is the free twisted braided module category on one object.  We give explicit presentations of twisted affinizations, both as monoidal categories and as module categories, using invertible dots and reflection relations.  We also describe the relation to cylinder twists and tensor-pair cylinder twists, including the naturality condition needed to obtain a braided module structure.  As applications, we obtain module-category presentations of the affine Kauffman and affine oriented skein categories and an orientation-reversing twisted version of the latter.  We realize the disoriented skein category as a quotient of this twisted affinization by a single additional relation.
\end{abstract}



\ifboolexpr{togl{comments} or togl{details}}{%
  {\color{magenta}DETAILS OR COMMENTS ON}
}{%
}

\hypersetup{
    pdftitle={Affinizations and braided module categories},
    pdfauthor={Hadi Salmasian; Alistair Savage; Yaolong Shen},
    pdfsubject={Twisted affinizations of braided monoidal categories, their universal properties as braided module categories, and applications to skein categories},
    pdfkeywords={Affinization; braided monoidal category; braided module category; cylinder twist; reflection equation; skein category; quantum symmetric pair; 2020 MSC: Primary 18M15, Secondary 17B37, 18M30},
    pdfcreator={LaTeX with hyperref},
}

\maketitle
\thispagestyle{empty}


\section{Introduction}

Braided module categories occur both in the topology of tangles and in the representation theory of quantum symmetric pairs.  A module category carries an action of a monoidal category, and a module braiding describes how a strand from the acting category winds around a module object, compatibly with crossings and tensor products.  Tangles in a solid torus give a geometric example of this structure \cite{Bro13}.  On the representation-theoretic side, representations of the coideal subalgebra in a quantum symmetric pair can be tensored on the right with representations of the ambient quantum group.  Universal $K$-matrices provide compatible module braidings, with a possible twist by an automorphism of the acting category \cite{BK19,Kol20,DNTY19}.  

The affine oriented skein category and the affine Kauffman skein category extend their non-affine counterparts by allowing strands to wind around a cylinder.  In presentations by generators and relations, this winding is represented by invertible dots on strands.  The interpretation of affine generators as winding operations also relates affine Hecke algebras to the topology of braids on a cylinder.  In \cite{MS21}, these examples were formalized through a general procedure for \emph{affinizing} a monoidal category: one adjoins invertible morphisms representing winding, with relations specifying their compatibility with the original morphisms and tensor products.  The resulting category is called the \emph{affinization}.  When the original category is braided, its affinization is again monoidal and admits an equivalent presentation by dot generators.  This construction gives a uniform treatment of affine braid, tangle, and skein categories, recovering in particular the affine oriented and Kauffman skein categories from the corresponding skein categories.

In the present paper, we relate the above two concepts by describing twisted braided module structures in terms of actions of twisted affinizations.  We also give explicit presentations of these affinizations by generators and relations, both as monoidal categories and as module categories.  Presentations as module categories turn out to be particularly relevant in the diagrammatic study of quantum symmetric pairs \cite{SSSsub26,SSS26,SSS-AIII,SSS-BCD}.

Let $\cC$ be a braided strict monoidal category with unit object $\one$, and let $\sigma\colon\cC\to\cC$ be a strict braided monoidal automorphism.  We use $\otimes$ for both the tensor product in $\cC$ and the action on a right $\cC$-module category.  A \emph{$\sigma$-braided} structure on a strict right $\cC$-module category $\cM$ is a family of isomorphisms, natural in objects $M$ of $\cM$ and $X$ of $\cC$,
\[
    \eta_{M,X}\colon M\otimes X\xrightarrow{\sim}M\otimes\sigma(X),
\]
satisfying the octagon and multiplicativity axioms of \cref{braidedmodcat}.  These axioms specify the behavior of the family under the module action and under tensor products in $\cC$, respectively.  The case $\sigma=\id_\cC$ recovers the untwisted notion of a braided module category.

We describe such structures using the \emph{$\sigma$-twisted affinization} $\Aff_\sigma(\cC)$ of \cref{affdef}.  This construction extends the affinization of \cite{MS21}, which describes diagrams on a cylinder by adjoining morphisms representing winding.  The automorphism $\sigma$ allows winding to change the labels of strands.  The twisted affinization has the same objects as $\cC$, and the braiding on $\cC$ gives the twisted affinization a strict monoidal structure for which the canonical functor $\imath_\cC\colon\cC\to\Aff_\sigma(\cC)$ is strict monoidal; see \cref{afftensor}.

Our main structural result, \cref{affactionextension}, states that \emph{$\sigma$-braided structures on a strict right $\cC$-module category are in bijection with strict right $\Aff_\sigma(\cC)$-actions extending the given $\cC$-action}.  The correspondence is explicit, allowing one to easily move between the two points of view. The twisted affinization itself has a canonical $\sigma$-braided right $\cC$-module structure, and it is \emph{free on one object} in this setting; see \cref{affbraiding,afffree}.   This universal property extends the description of tangles in the solid torus in \cite[Thms.~5.2--5.3]{Bro13}.  In the untwisted case, the action-extension theorem is also an algebraic counterpart of the description using the annulus category in \cite[Thm.~3.11]{BBJ18}.  Our results allow a nontrivial automorphism $\sigma$ and require neither rigidity nor a balancing on $\cC$.

Explicit presentations of the twisted affinization \emph{as module categories} (as opposed to as monoidal categories) are central to the above results.  In \cref{jungle,jungleneg}, we present $\Aff_\sigma(\cC)$ as a right $\cC$-module category by invertible dot morphisms, a relation determining the dot on a tensor product, the twisted reflection equation, and naturality relations.  When the objects of $\cC$ are freely generated under tensor product by a chosen set, these presentations reduce to dots on the generating objects and relations involving generating morphisms; see \cref{swamp+,swamp-}.  The module-category presentations make it possible to construct functors into a right $\cC$-module category by checking explicit relations among the images of the dots.  They also provide a setting for quotients by relations preserved by the right $\cC$-action.  In the linear setting, cyclotomic quotients impose polynomial relations on dot endomorphisms, and the module-category viewpoint records the compatibility of such relations with the acting category.

While our main interest is in the connection between affinizations and braided module categories, some of the literature uses the related concept of a \emph{cylinder twist}.  In \cref{S:braid} we give the precise connection between braided module categories and two types of cylinder twists.

Our main results yield explicit presentations of several skein categories as module categories.  The affine Kauffman skein category and the affine oriented skein category admit the module-category presentations in \cref{AKpresent,AOSpresent}.  In the oriented case, a single invertible dot and a reflection relation suffice.  Taking $\sigma$ to reverse all strand orientations gives a twisted affine oriented skein category generated, as a module category, by a single invertible morphism between the two strand orientations, subject to a twisted reflection relation; see \cref{tAOSmodule}.  Finally, \cref{DStheorem} realizes the disoriented skein category of \cite[Def.~2.1]{SSS26} as a quotient of this twisted affinization by one additional relation.  This also shows that one of the relations in its original presentation is redundant.  The disoriented skein category provides a diagrammatic description of representations of quantum symmetric pairs associated with orthosymplectic Lie superalgebras \cite{SSS26}, so the quotient description connects the general affinization construction to these representation categories.

The paper is organized as follows.  We define twisted affinization and give its monoidal presentations in \cref{S:affinize}.  We develop presentations as module categories in \cref{S:affmodule}.  We compare braided module structures, cylinder twists, and tensor-pair cylinder twists in \cref{S:braid}, then prove the action-extension theorem and universal property in \cref{S:affbraided}.  The skein applications occupy \cref{S:applications}.

\subsection*{Connection to the literature}

The reflection equation appears in algebra presentations obtained by adjoining a single affine generator.  For example, in the normalization $(T_i-q)(T_i+q^{-1})=0$, the extended affine Hecke algebra of type $A_{n-1}$, $n \ge 2$, is obtained from the finite Hecke algebra by adjoining an invertible $X_1$ commuting with $T_j$ for $j>1$ and satisfying
\[
    X_1T_1X_1T_1=T_1X_1T_1X_1.
\]
If we recursively define $X_{i+1}:=T_iX_iT_i$, $1 \le i < n$, then the $X_i$ commute; the displayed reflection equation is precisely $X_1X_2=X_2X_1$.  Cyclotomic quotients impose a polynomial relation on $X_1$; see \cite[\S2]{OP12}.  Affine and cyclotomic BMW algebras have analogous presentations, with additional relations involving cups and caps \cite[Def.~2.3 and \S3]{GHM09}.  

The connection with these algebraic examples is most explicit in the module-category presentations of \cref{jungle,jungleneg}.  Here the reflection equation appears alongside naturality relations and a relation determining the dot on a tensor product from the dots on its factors.

Isaev--Ogievetsky connect these presentations to $K$-matrices by realizing affine Hecke generators using reflection equation algebras and constructing spectral-parameter solutions from affine generators \cite[\S1]{IO07}.  Jordan--Ma use quantum symmetric pairs and reflection equations to construct representations of affine and double affine Hecke algebras of type $C^\vee C_n$ \cite{JM11}.  Buciumas--Ko explicitly relate the first type-$B$ Hecke generator to a $K$-matrix, the remaining Jucys--Murphy elements to its braid translates, and these constructions to cylinder-braided action pairs and braided module categories \cite[\S2.1, Thm.~4.7, Rem.~4.9]{BuK23}.

There are substantial categorical precedents as well.  Cylinder braidings were studied by tom Dieck--H\"aring-Oldenburg \cite{tDH98}; Balagovi\'c--Kolb formulate universal $K$-matrices using tensor-pair cylinder twists, including twisted versions \cite[\S4]{BK19}, and Kolb constructs braided module categories from quantum symmetric pairs, using an equivariantization when required by the twist \cite{Kol20}.  Freeness results for tangles in the solid torus appear in \cite[Thms.~5.2--5.3]{Bro13}, and a related universal property for cylindrical braid groupoids in the setting of boundary bimodule categories appears in \cite[Prop.~C.6.2]{AV25}.

\subsection*{Acknowledgments}

The research of A.S.\ and H.S.\ was supported by Discovery Grants RGPIN-2023-03842 and RGPIN-2024-04030 from the Natural Sciences and Engineering Research Council of Canada.  Y.S.\ was supported by these grants, the Fields Institute for Research in Mathematical Sciences, and the Science and Technology Commission of Shanghai Municipality (No. 22DZ2229014). 

\subsection*{Use of AI tools}

The authors used OpenAI's GPT-6 Astra and Anthropic's Claude Fable 5.1 to conduct literature searches and to proofread the manuscript, including checking the validity of proofs.  In some cases, AI caught small logical gaps in proofs and suggested fixes.  Some of these fixes were adopted after being edited by the authors.  For other gaps, the authors developed their own fix.  The authors take full responsibility for the content of this paper.

\section{Twisted affinization of braided monoidal categories\label{S:affinize}}

In this section, we generalize the affinization construction of \cite{MS21} by allowing a monoidal automorphism to act on strands winding around the cylinder.  When this automorphism is the identity, our definitions and results recover those of that paper.  The generalization is fairly straightforward but the added flexibility will allow us to treat some more cases of interest; see \cref{SS:tAOScat}.

\subsection{Definition}

We assume all categories are essentially small.  For a category $\cC$, we let $\Ob(\cC)$ denote its class of objects, and $\Mor(\cC)$ its class of morphisms.  For $X,Y \in \Ob(\cC)$, we denote the class of morphisms from $X$ to $Y$ by $\cC(X,Y)$.  We let $1_X$ denote the identity morphism of an object $X$.  We use $\kk$ to denote a commutative ground ring.  We let $\one$ denote the unit object of a monoidal category.  Throughout this section, $\cC$ denotes a strict monoidal category and $\sigma \colon \cC \to \cC$ is a strict monoidal automorphism.  When $\cC$ is $\kk$-linear, we also assume that $\sigma$ is $\kk$-linear.  We write $X^\sigma := \sigma(X)$ and $f^\sigma := \sigma(f)$ for objects and morphisms.  In particular, $(X \otimes Y)^\sigma = X^\sigma \otimes Y^\sigma$ and $\one^\sigma=\one$.

\begin{defin}[Twisted affinization of a strict monoidal category] \label{affdef}
    The \emph{$\sigma$-twisted affinization} of $\cC$ is the category $\Aff_\sigma(\cC)$ obtained from $\cC$ by adjoining invertible morphisms $\xi_{X,Y} \colon X \otimes Y^\sigma \to Y \otimes X$ for each pair of objects $X,Y \in \Ob(\cC)$, subject to the relations
    \begin{gather} \label{coilrel1}
        \xi_{X,Y \otimes Z} = \xi_{Z \otimes X,Y} \circ \xi_{X \otimes Y^\sigma,Z},
        \\ \label{coilrel2}
        \xi_{X_2,Y_2} \circ (g \otimes f^\sigma) = (f \otimes g) \circ \xi_{X_1,Y_1},
    \end{gather}
    for all $X,X_1,X_2,Y,Y_1,Y_2,Z \in \Ob(\cC)$, $f \in \cC(Y_1,Y_2)$, and $g \in \cC(X_1,X_2)$.  We call the $\xi_{X,Y}$ \emph{coils}.  We define
    \[
        \xi_X := \xi_{\one,X} \colon X^\sigma \to X.
    \]
    When $\cC$ is $\kk$-linear, we take this presentation in $\kk$-linear categories, so that $\Aff_\sigma(\cC)$ is $\kk$-linear.  We write $\Aff(\cC) := \Aff_{\id_\cC}(\cC)$ for the untwisted affinization.
\end{defin}

It follows from \cref{coilrel1} that $\xi_{X,\one} = \xi_{X,\one \otimes \one} = \xi_{X,\one} \circ \xi_{X,\one}$ and so, since $\xi_{X,\one}$ is invertible, we have
\begin{equation} \label{cable}
    \xi_{X,\one} = 1_X
    \quad \text{for all } X \in \Ob(\cC).
\end{equation}
It follows from \cref{coilrel2} that
\begin{equation}
    \xi_Y \circ f^\sigma = f \circ \xi_X, \quad
    f \in \cC(X,Y).
\end{equation}
In other words, $(\xi_X)_{X \in \Ob(\cC)}$ is a natural isomorphism $\imath_\cC \circ \sigma \Rightarrow \imath_\cC$, where
\begin{equation}\label{imath}
    \imath_\cC \colon \cC \to \Aff_\sigma(\cC)
\end{equation}
is the canonical functor, which is the identity on objects.

Recall that a strict monoidal category $\cC$ is \emph{braided} if it is equipped with a natural family of isomorphisms
\[
    \beta_{X,Y}
    =
    \begin{tikzpicture}[centerzero]
        \draw (0.3,-0.3) \botlabel{Y} -- (-0.3,0.3);
        \draw[over] (-0.3,-0.3) \botlabel{X} -- (0.3,0.3);
    \end{tikzpicture}
    \colon X \otimes Y \to Y \otimes X
    \quad \text{with inverses} \quad
    \beta_{X,Y}^{-1}
    =
    \begin{tikzpicture}[centerzero]
        \draw (-0.3,-0.3) \botlabel{Y} -- (0.3,0.3);
        \draw[over] (0.3,-0.3) \botlabel{X} -- (-0.3,0.3);
    \end{tikzpicture}
\]
for $X,Y \in \Ob(\cC)$, satisfying
\begin{equation} \label{hexagon}
    \begin{tikzpicture}[centerzero]
        \draw (0.5,-0.5) \botlabel{Y \otimes Z} -- (-0.5,0.5);
        \draw[over] (-0.5,-0.5) \botlabel{X} -- (0.5,0.5);
    \end{tikzpicture}
    =
    \begin{tikzpicture}[centerzero]
        \draw (0,-0.5) \botlabel{Y} -- (-0.5,0.5);
        \draw (0.5,-0.5) \botlabel{Z} -- (0,0.5);
        \draw[over] (-0.5,-0.5) \botlabel{X} -- (0.5,0.5);
    \end{tikzpicture}
    \ ,\quad
    \begin{tikzpicture}[centerzero]
        \draw (0.5,-0.5) \botlabel{Z} -- (-0.5,0.5);
        \draw[over] (-0.5,-0.5) \botlabel{X \otimes Y} -- (0.5,0.5);
    \end{tikzpicture}
    =
    \begin{tikzpicture}[centerzero]
        \draw (0.5,-0.5) \botlabel{Z} -- (-0.5,0.5);
        \draw[over] (0,-0.5) \botlabel{Y} -- (0.5,0.5);
        \draw[over] (-0.5,-0.5) \botlabel{X} -- (0,0.5);
    \end{tikzpicture}
    \ .
\end{equation}

In what follows, we will use unlabeled strands to indicate that a relation holds for any labeling of the strands.  In particular,
\[
    \begin{tikzpicture}[centerzero]
        \draw (0,-0.4) -- (0,0.4);
        \coupon{0,0}{f};
    \end{tikzpicture}
    \quad \text{represents an arbitrary morphism} \quad
    \begin{tikzpicture}[centerzero]
        \draw (0,-0.4) \botlabel{X} -- (0,0.4) \toplabel{Y};
        \coupon{0,0}{f};
    \end{tikzpicture}
    \quad \text{in } \cC,
\]
and we have
\begin{equation} \label{Reid2}
    \begin{tikzpicture}[centerzero]
        \draw (0.25,-0.5) to[out=135,in=down] (-0.2,0) to[out=up,in=225] (0.25,0.5);
        \draw[over] (-0.25,-0.5) to[out=45,in=down] (0.2,0) to[out=up,in=-45] (-0.25,0.5);
    \end{tikzpicture}
    \ =\
    \begin{tikzpicture}[centerzero]
        \draw (-0.2,-0.5) -- (-0.2,0.5);
        \draw (0.2,-0.5) -- (0.2,0.5);
    \end{tikzpicture}
    \ =\
    \begin{tikzpicture}[centerzero]
        \draw (-0.25,-0.5) to[out=45,in=down] (0.2,0) to[out=up,in=-45] (-0.25,0.5);
        \draw[over] (0.25,-0.5) to[out=135,in=down] (-0.2,0) to[out=up,in=225] (0.25,0.5);
    \end{tikzpicture}
    \ .
\end{equation}
The naturality of the braiding means that
\begin{gather} \label{Reid3}
    \begin{tikzpicture}[centerzero]
        \draw (0.5,-0.5) -- (-0.5,0.5);
        \draw[over] (0,-0.5) to[out=135,in=down] (-0.4,0) to[out=up,in=225] (0,0.5);
        \draw[over] (-0.5,-0.5) -- (0.5,0.5);
    \end{tikzpicture}
    \ =\
    \begin{tikzpicture}[centerzero]
        \draw (0.5,-0.5) -- (-0.5,0.5);
        \draw[over] (0,-0.5) to[out=45,in=down] (0.4,0) to[out=up,in=-45] (0,0.5);
        \draw[over] (-0.5,-0.5) -- (0.5,0.5);
    \end{tikzpicture}
    \ ,
    \\ \label{lynx}
    \begin{tikzpicture}[centerzero]
        \draw (-0.5,-0.5) -- (0.5,0.5);
        \draw[over] (0.5,-0.5) -- (-0.5,0.5);
        \coupon{-0.28,-0.28}{f};
    \end{tikzpicture}
    =
    \begin{tikzpicture}[centerzero]
        \draw (-0.5,-0.5) -- (0.5,0.5);
        \draw[over] (0.5,-0.5) -- (-0.5,0.5);
        \coupon{0.28,0.28}{f};
    \end{tikzpicture}
    \ ,\quad
    \begin{tikzpicture}[centerzero]
        \draw (0.5,-0.5) -- (-0.5,0.5);
        \draw[over] (-0.5,-0.5) -- (0.5,0.5);
        \coupon{-0.25,-0.25}{f};
    \end{tikzpicture}
    =
    \begin{tikzpicture}[centerzero]
        \draw (0.5,-0.5) -- (-0.5,0.5);
        \draw[over] (-0.5,-0.5) -- (0.5,0.5);
        \coupon{0.25,0.25}{f};
    \end{tikzpicture}
    \ ,\quad
    \begin{tikzpicture}[centerzero]
        \draw (-0.5,-0.5) -- (0.5,0.5);
        \draw[over] (0.5,-0.5) -- (-0.5,0.5);
        \coupon{0.25,-0.25}{f};
    \end{tikzpicture}
    =
    \begin{tikzpicture}[centerzero]
        \draw (-0.5,-0.5) -- (0.5,0.5);
        \draw[over] (0.5,-0.5) -- (-0.5,0.5);
        \coupon{-0.25,0.25}{f};
    \end{tikzpicture}
    \ ,\quad
    \begin{tikzpicture}[centerzero]
        \draw (0.5,-0.5) -- (-0.5,0.5);
        \draw[over] (-0.5,-0.5) -- (0.5,0.5);
        \coupon{0.28,-0.28}{f};
    \end{tikzpicture}
    =
    \begin{tikzpicture}[centerzero]
        \draw (0.5,-0.5) -- (-0.5,0.5);
        \draw[over] (-0.5,-0.5) -- (0.5,0.5);
        \coupon{-0.28,0.28}{f};
    \end{tikzpicture}
    \ .
\end{gather}
In other words, the braid relations are satisfied, and coupons slide through crossings.

\subsection{Monoidal structure}

For the remainder of this section, assume that $\cC$ is braided and that $\sigma$ is a braided monoidal automorphism, so that
\[
    \sigma(\beta_{X,Y}) = \beta_{X^\sigma,Y^\sigma},\qquad X,Y \in \Ob(\cC).
\]

\begin{prop} \label{afftensor}
    There is a unique way to extend the tensor product of $\cC$ to $\Aff_\sigma(\cC)$ such that $\Aff_\sigma(\cC)$ is a strict monoidal category and
    \begin{equation} \label{stack}
        \xi_{X,Y} \otimes 1_Z = \xi_{X \otimes Z, Y} \circ (1_X \otimes \beta_{Y^\sigma,Z}), \quad
        1_X \otimes \xi_{Y,Z} = (\beta_{X,Z}^{-1} \otimes 1_Y) \circ \xi_{X \otimes Y,Z},
    \end{equation}
    for all $X,Y,Z \in \Ob(\cC)$.
\end{prop}

\begin{proof}
    The proof of \cite[Prop.~2.5]{MS21}, which treats $\sigma=\id_\cC$, applies with the twisted coil relations \cref{coilrel1,coilrel2}.  The formulas \cref{stack} determine the tensor product uniquely, since morphisms of $\cC$, coils, and their inverses generate the category.  Naturality and the identities \cref{hexagon}, together with $\sigma(\beta_{X,Y})=\beta_{X^\sigma,Y^\sigma}$, give the same verification that these formulas preserve the defining relations and satisfy associativity, the unit axioms, and interchange.
\end{proof}

With the tensor product described in \cref{afftensor}, the canonical functor $\imath_\cC$ of \cref{imath} is strict monoidal.

We use a \emph{positive dot} on a strand to denote the morphism $\xi_X$ and a \emph{negative dot} to denote $\xi_X^{-1}$:
\[
    \begin{tikzpicture}[centerzero]
        \draw (0,-0.3) \botlabel{X^\sigma} -- (0,0.3) \toplabel{X};
        \posdot{0,0};
    \end{tikzpicture}
    := \xi_X,
    \qquad
    \begin{tikzpicture}[centerzero]
        \draw (0,-0.3) \botlabel{X} -- (0,0.3) \toplabel{X^\sigma};
        \negdot{0,0};
    \end{tikzpicture}
    := \xi_X^{-1}.
\]
Thus a positive dot changes a label $X^\sigma$ to $X$, while a negative dot changes $X$ to $X^\sigma$.  Replacing $X,Y,Z$ by $\one,Y,X$, respectively, in the left-hand equation in \cref{stack}, then composing on the right with $\beta_{Y^\sigma,X}^{-1}$ gives
\begin{equation} \label{stealth}
    \xi_{X,Y} = (\xi_Y \otimes 1_X) \circ \beta_{Y^\sigma,X}^{-1}
    =
    \begin{tikzpicture}[centerzero]
        \draw (-0.2,-0.35) \botlabel{X} \braidup (0.2,0.05) -- (0.2,0.35) \toplabel{X};
        \draw[over] (0.2,-0.35) \botlabel{Y^\sigma} \braidup (-0.2,0.05) -- (-0.2,0.35) \toplabel{Y};
        \posdot{-0.2,0.1};
    \end{tikzpicture}
    \ .
\end{equation}
Using \cref{coilrel1,coilrel2,stack}, we obtain
\begin{gather} \label{jaguar+}
    \begin{tikzpicture}[centerzero]
        \draw (-0.3,-0.3) \botlabel{X^\sigma} -- (0.3,0.3) \toplabel{X};
        \draw[over] (0.3,-0.3) \botlabel{Y} -- (-0.3,0.3) \toplabel{Y};
        \posdot{-0.18,-0.18};
    \end{tikzpicture}
    =
    \begin{tikzpicture}[centerzero]
        \draw (0.3,-0.3) \botlabel{Y} -- (-0.3,0.3) \toplabel{Y};
        \draw[over] (-0.3,-0.3) \botlabel{X^\sigma} -- (0.3,0.3) \toplabel{X};
        \posdot{0.15,0.15};
    \end{tikzpicture}
    \ ,\quad
    \begin{tikzpicture}[centerzero]
        \draw (0.3,-0.3) \botlabel{Y^\sigma} -- (-0.3,0.3) \toplabel{Y};
        \draw[over] (-0.3,-0.3) \botlabel{X} -- (0.3,0.3) \toplabel{X};
        \posdot{0.18,-0.18};
    \end{tikzpicture}
    =
    \begin{tikzpicture}[centerzero]
        \draw (-0.3,-0.3) \botlabel{X} -- (0.3,0.3) \toplabel{X};
        \draw[over] (0.3,-0.3) \botlabel{Y^\sigma} -- (-0.3,0.3) \toplabel{Y};
        \posdot{-0.15,0.15};
    \end{tikzpicture}
    \ ,\quad
    \begin{tikzpicture}[centerzero]
        \draw (0,-0.4) \botlabel{(X \otimes Y)^\sigma} -- (0,0.4) \toplabel{X \otimes Y};
        \posdot{0,0};
    \end{tikzpicture}
    =
    \begin{tikzpicture}[centerzero]
        \draw (0.2,-0.4) \botlabel{Y^\sigma} -- (0.2,0.4) \toplabel{Y};
        \draw (-0.2,-0.4) \botlabel{X^\sigma} -- (-0.2,0.4) \toplabel{X};
        \posdot{-0.2,0};
        \posdot{0.2,0};
    \end{tikzpicture}
    \ ,\quad
    \begin{tikzpicture}[centerzero]
        \draw (0,-0.5) \botlabel{X^\sigma} -- (0,0.5) \toplabel{Y};
        \posdot{0,0.25};
        \coupon{0,-0.2}{f^\sigma};
    \end{tikzpicture}
    =
    \begin{tikzpicture}[centerzero]
        \draw (0,-0.5) \botlabel{X^\sigma} -- (0,0.5) \toplabel{Y};
        \posdot{0,-0.25};
        \coupon{0,0.2}{f};
    \end{tikzpicture}
    \ ,
    \\ \label{jaguar-}
    \begin{tikzpicture}[centerzero]
        \draw (0.3,-0.3) \botlabel{Y} -- (-0.3,0.3) \toplabel{Y};
        \draw[over] (-0.3,-0.3) \botlabel{X} -- (0.3,0.3) \toplabel{X^\sigma};
        \negdot{-0.15,-0.15};
    \end{tikzpicture}
    =
    \begin{tikzpicture}[centerzero]
        \draw (-0.3,-0.3) \botlabel{X} -- (0.3,0.3) \toplabel{X^\sigma};
        \draw[over] (0.3,-0.3) \botlabel{Y} -- (-0.3,0.3) \toplabel{Y};
        \negdot{0.18,0.18};
    \end{tikzpicture}
    \ ,\quad
    \begin{tikzpicture}[centerzero]
        \draw (-0.3,-0.3) \botlabel{X} -- (0.3,0.3) \toplabel{X};
        \draw[over] (0.3,-0.3) \botlabel{Y} -- (-0.3,0.3) \toplabel{Y^\sigma};
        \negdot{0.15,-0.15};
    \end{tikzpicture}
    =
    \begin{tikzpicture}[centerzero]
        \draw (0.3,-0.3) \botlabel{Y} -- (-0.3,0.3) \toplabel{Y^\sigma};
        \draw[over] (-0.3,-0.3) \botlabel{X} -- (0.3,0.3) \toplabel{X};
        \negdot{-0.18,0.18};
    \end{tikzpicture}
    \ ,\quad
    \begin{tikzpicture}[centerzero]
        \draw (0,-0.4) \botlabel{X \otimes Y} -- (0,0.4) \toplabel{(X \otimes Y)^\sigma};
        \negdot{0,0};
    \end{tikzpicture}
    =
    \begin{tikzpicture}[centerzero]
        \draw (0.2,-0.4) \botlabel{Y} -- (0.2,0.4) \toplabel{Y^\sigma};
        \draw (-0.2,-0.4) \botlabel{X} -- (-0.2,0.4) \toplabel{X^\sigma};
        \negdot{-0.2,0};
        \negdot{0.2,0};
    \end{tikzpicture}
    \ ,\quad
    \begin{tikzpicture}[centerzero]
        \draw (0,-0.5) \botlabel{X} -- (0,0.5) \toplabel{Y^\sigma};
        \negdot{0,0.25};
        \coupon{0,-0.2}{f};
    \end{tikzpicture}
    =
    \begin{tikzpicture}[centerzero]
        \draw (0,-0.5) \botlabel{X} -- (0,0.5) \toplabel{Y^\sigma};
        \negdot{0,-0.25};
        \coupon{0,0.2}{f^\sigma};
    \end{tikzpicture}
    \ ,
    \\
    \begin{tikzpicture}[centerzero]
        \draw (0,-0.4) \botlabel{X} -- (0,0.4) \toplabel{X};
        \posdot{0,0.17};
        \negdot{0,-0.17};
    \end{tikzpicture}
    =
    \begin{tikzpicture}[centerzero]
        \draw (0,-0.4) \botlabel{X} -- (0,0.4) \toplabel{X};
    \end{tikzpicture}
    \ ,\qquad
    \begin{tikzpicture}[centerzero]
        \draw (0,-0.4) \botlabel{X^\sigma} -- (0,0.4) \toplabel{X^\sigma};
        \negdot{0,0.17};
        \posdot{0,-0.17};
    \end{tikzpicture}
    =
    \begin{tikzpicture}[centerzero]
        \draw (0,-0.4) \botlabel{X^\sigma} -- (0,0.4) \toplabel{X^\sigma};
    \end{tikzpicture}
    \ ,
\end{gather}
for all $X,Y \in \Ob(\cC)$ and $f \in \cC(X,Y)$.  The relations in \cref{jaguar-} follow from those in \cref{jaguar+} by inverting dots and relabeling strands.  
It follows from the third relation in \cref{jaguar+} that $\xi_\one = 1_\one$.  This, together with the third and fourth relations in \cref{jaguar+},  says that $(\xi_X)_{X \in \Ob(\cC)}$ is a monoidal natural isomorphism $\imath_\cC \circ \sigma \Rightarrow \imath_\cC$.

\subsection{Monoidal dot presentations}

When we have a presentation of the braided monoidal category $\cC$, we can give a more efficient presentation of its affinization.  Assume that the objects of $\cC$ are freely generated by a set $S_0$, and that the morphisms of $\cC$ are generated by a set $S_1$.  Let $S_1'$ be obtained from $S_1$ by removing any braiding morphisms and their inverses.

\begin{prop}[Monoidal positive-dot presentation] \label{lion+}
    The twisted affinization $\Aff_\sigma(\cC)$ is isomorphic to the strict monoidal category obtained from $\cC$ by adjoining invertible morphisms
    \begin{equation} \label{lion+gen}
        \begin{tikzpicture}[centerzero]
            \draw (0,-0.3) \botlabel{X^\sigma} -- (0,0.3) \toplabel{X};
            \posdot{0,0};
        \end{tikzpicture}
        \colon X^\sigma \to X,\qquad
        X \in S_0,
    \end{equation}
    subject to the relations
    \begin{equation} \label{lion+rels}
        \begin{tikzpicture}[centerzero]
            \draw (-0.3,-0.3) \botlabel{X^\sigma} -- (0.3,0.3) \toplabel{X};
            \draw[over] (0.3,-0.3) \botlabel{Y} -- (-0.3,0.3) \toplabel{Y};
            \posdot{-0.18,-0.18};
        \end{tikzpicture}
        =
        \begin{tikzpicture}[centerzero]
            \draw (0.3,-0.3) \botlabel{Y} -- (-0.3,0.3) \toplabel{Y};
            \draw[over] (-0.3,-0.3) \botlabel{X^\sigma} -- (0.3,0.3) \toplabel{X};
            \posdot{0.15,0.15};
        \end{tikzpicture}
        \ ,\qquad
        \begin{tikzpicture}[centerzero]
            \draw (0,-0.6) -- (0,0.6);
            \posdot{0,0.25};
            \coupon{0,-0.2}{f^\sigma};
        \end{tikzpicture}
        =
        \begin{tikzpicture}[centerzero]
            \draw (0,-0.6) -- (0,0.6);
            \posdot{0,-0.25};
            \coupon{0,0.2}{f};
        \end{tikzpicture}
        \ ,\qquad
        X,Y \in S_0,\ f \in S_1'.
    \end{equation}
    In this presentation, the positive dots on arbitrary objects are determined by the third relation in \cref{jaguar+}.
\end{prop}

\begin{proof}
    A straightforward generalization of the argument of \cite[Thm.~2.6]{MS21}, which treats the $\sigma=\id_\cC$ case, shows that $\Aff_\sigma(\cC)$ is isomorphic to the strict monoidal category obtained from $\cC$ by adjoining invertible morphisms \cref{lion+gen} for \emph{all} objects $X \in \Ob(\cC)$, subject to the first, third, and fourth relations in \cref{jaguar+}.  Using the third relation in \cref{jaguar+} to \emph{define} positive dots on arbitrary objects, it suffices to consider the smaller generating set \cref{lion+gen}.  It remains to prove that the relations \cref{lion+rels} for $X,Y \in S_0$, $f \in S_1'$ imply that these relations hold for \emph{all} objects and morphisms.

    It is straightforward to see that the first relation in \cref{jaguar+} holds for arbitrary objects.  Composing this relation on top with $\poscross$ and on the bottom with $\negcross$ then gives the second relation in \cref{jaguar+} for arbitrary objects.
    
    It is straightforward to check that the class of morphisms $f$ satisfying the fourth relation in \cref{jaguar+} is a monoidal subcategory of $\cC$ and closed under inverses of invertible morphisms.  Thus, it suffices to impose this relation for $f \in S_1$.  Finally, if $f \colon X \otimes Y \to Y \otimes X$ is a braiding morphism, we have, using only the first, second, and third relations in \cref{jaguar+},
    \[
        \begin{tikzpicture}[centerzero]
            \draw (0,-0.7) \botlabel{X^\sigma \otimes Y^\sigma} -- (0,0.7) \toplabel{Y \otimes X};
            \posdot{0,0.3};
            \coupon{0,-0.3}{f^\sigma};
        \end{tikzpicture}
        =
        \begin{tikzpicture}[centerzero]
            \draw (0.25,-0.7) \botlabel{Y^\sigma} -- (0.25,-0.3) \braidup (-0.25,0.3) -- (-0.25,0.7) \toplabel{Y};
            \draw[over] (-0.25,-0.7) \botlabel{X^\sigma} -- (-0.25,-0.3) \braidup (0.25,0.3) -- (0.25,0.7) \toplabel{X};
            \posdot{-0.25,0.4};
            \posdot{0.25,0.4};
        \end{tikzpicture}
        =
        \begin{tikzpicture}[centerzero]
            \draw (-0.25,-0.7) \botlabel{X^\sigma} -- (-0.25,-0.3) \braidup (0.25,0.3) -- (0.25,0.7) \toplabel{X};
            \draw[over] (0.25,-0.7) \botlabel{Y^\sigma} -- (0.25,-0.3) \braidup (-0.25,0.3) -- (-0.25,0.7) \toplabel{Y};
            \posdot{-0.25,0.4};
            \posdot{-0.25,-0.4};
        \end{tikzpicture}
        =
        \begin{tikzpicture}[centerzero]
            \draw (0.25,-0.7) \botlabel{Y^\sigma} -- (0.25,-0.3) \braidup (-0.25,0.3) -- (-0.25,0.7) \toplabel{Y};
            \draw[over] (-0.25,-0.7) \botlabel{X^\sigma} -- (-0.25,-0.3) \braidup (0.25,0.3) -- (0.25,0.7) \toplabel{X};
            \posdot{-0.25,-0.4};
            \posdot{0.25,-0.4};
        \end{tikzpicture}
        =
        \begin{tikzpicture}[centerzero]
            \draw (0,-0.7) \botlabel{X^\sigma \otimes Y^\sigma} -- (0,0.7) \toplabel{Y \otimes X};
            \posdot{0,-0.3};
            \coupon{0,0.3}{f};
        \end{tikzpicture}
        \ .
    \]
    Thus, we do not need to impose the last relation in \cref{jaguar+} when $f$ is a braiding morphism or its inverse.
\end{proof}

\begin{cor}[Monoidal negative-dot presentation] \label{lion-}
    The twisted affinization $\Aff_\sigma(\cC)$ is isomorphic to the strict monoidal category obtained from $\cC$ by adjoining invertible morphisms
    \begin{equation} \label{lion-gen}
        \begin{tikzpicture}[centerzero]
            \draw (0,-0.3) \botlabel{X} -- (0,0.3) \toplabel{X^\sigma};
            \negdot{0,0};
        \end{tikzpicture}
        \colon X \to X^\sigma,\qquad
        X \in S_0,
    \end{equation}
    subject to the relations
    \begin{equation} \label{lion-rels}
        \begin{tikzpicture}[centerzero]
            \draw (0.3,-0.3) \botlabel{Y} -- (-0.3,0.3) \toplabel{Y};
            \draw[over] (-0.3,-0.3) \botlabel{X} -- (0.3,0.3) \toplabel{X^\sigma};
            \negdot{-0.15,-0.15};
        \end{tikzpicture}
        =
        \begin{tikzpicture}[centerzero]
            \draw (-0.3,-0.3) \botlabel{X} -- (0.3,0.3) \toplabel{X^\sigma};
            \draw[over] (0.3,-0.3) \botlabel{Y} -- (-0.3,0.3) \toplabel{Y};
            \negdot{0.18,0.18};
        \end{tikzpicture}
        \ ,\qquad
        \begin{tikzpicture}[centerzero]
            \draw (0,-0.6) -- (0,0.6);
            \negdot{0,0.25};
            \coupon{0,-0.2}{f};
        \end{tikzpicture}
        =
        \begin{tikzpicture}[centerzero]
            \draw (0,-0.6) -- (0,0.6);
            \negdot{0,-0.25};
            \coupon{0,0.2}{f^\sigma};
        \end{tikzpicture}
        \ ,\qquad
        X,Y \in S_0,\ f \in S_1'.
    \end{equation}
    In this presentation, the negative dots on arbitrary objects are determined by the third relation in \cref{jaguar-}.
\end{cor}

\begin{rem}\label{monpresrem}
    \begin{enumerate}
        \item When $\cC$ is $\kk$-linear, it suffices to take $S_1$ to be a class of morphisms that generate $\cC$ as a $\kk$-linear monoidal category since the class of morphisms $f$ satisfying the fourth relation in \cref{jaguar+} is a $\kk$-linear monoidal subcategory of $\cC$.
    
        \item \label{monpresrem1} Even in the untwisted case, where $\sigma = \id_\cC$, \cref{lion+} gives a reduced presentation compared to \cite[Thm.~2.6]{MS21}.  Without the assumption that the objects of $\cC$ are freely generated by some subset of objects, one has the obvious twisted analogue of \cite[Thm.~2.6]{MS21} including all the positive dots as generators.
            
        \item \label{monpresrem2} In some of our examples of interest, the presentations in \cref{lion+,lion-} can be simplified even further by omitting the dot generator on one object in a dual pair.  See \cref{SS:AOScat,tAOSmonoidal}.
    \end{enumerate}
\end{rem}

\section{Presentation as a module category\label{S:affmodule}}

Our next goal is to give presentations of affinizations as \emph{module categories}, as opposed to as monoidal categories.

\subsection{Module categories}

Let $(\cC,\otimes, \one)$ be a strict monoidal category, and let $\cM$ be a category.  We denote by $\cEnd(\cM)$ the strict monoidal category of endofunctors of $\cM$ and natural transformations, with tensor product given by composition.   A \emph{right action} of $\cC$ on $\cM$ is a monoidal functor $\bA \colon \cC \to \cEnd(\cM)^\rev$, where $\cD^\rev$ denotes the \emph{reverse} of a monoidal category $\cD$, where we reverse the tensor product.  Given such a right action, we also say that $\cM$ is a \emph{right module category} over $\cC$, or \emph{$\cC$-module} for short; see \cite[\S 7.1]{EGNO15}.  For  $C \in \Ob(\cC)$ and $M \in \Ob(\cM)$, we abbreviate 
\begin{equation} \label{beaver}
    M \otimes C := \bA(C)(M),
\end{equation}
and similarly for morphisms.  We say that the $\cC$-module $\cM$ is \emph{strict} if $\bA$ is a strict monoidal functor.  Equivalently, the associativity and unit constraints of the action are identities; in particular,
\[
    M \otimes (C_1 \otimes C_2) = (M \otimes C_1) \otimes C_2,
    \qquad M \otimes \one = M,
\]
and the corresponding equalities hold for morphisms.

Let $\cM$, $\cN$ be strict $\cC$-modules.  A \emph{morphism of $\cC$-modules} from $\cM$ to $\cN$ is a pair $(\bH,\omega)$, where $\bH$ is a functor from $\cM$ to $\cN$ and $\omega$ is a natural isomorphism with components
\[
    \omega_{M,C} \colon \bH(M \otimes C) \xrightarrow{\cong} \bH(M) \otimes C,\qquad M \in \cM,\ C \in \cC,
\]
such that the diagram
\begin{equation} \label{buffalo}
    \begin{tikzcd}
        & \bH (M \otimes C \otimes D) \arrow[dl,"\omega_{M \otimes C, D}"'] \ar[dr,"\omega_{M, C \otimes D}"] &
        \\
        \bH(M \otimes C) \otimes D \arrow[rr,"\omega_{M,C} \otimes 1_D"] & & \bH(M) \otimes C \otimes D
    \end{tikzcd}
\end{equation}
commutes for all $M \in \cM$ and $C,D \in \mathcal{C}$.  (See \cite[Def.~7.2.1]{EGNO15} for a more general definition, where the module categories are not required to be strict.)  We say that a morphism of $\cC$-modules is \emph{strict} if $\omega_{M,C}$ is the identity morphism for all $M \in \cM$ and $C \in \cC$.  An \emph{equivalence of $\cC$-modules} is a morphism of $\cC$-modules that is also an equivalence of categories.

The definitions and results in \cref{S:affmodule,S:braid,S:affbraided} also apply in the $\kk$-linear setting: categories and functors are $\kk$-linear, and tensor products and module actions are $\kk$-bilinear on morphisms.  In that setting, $\cEnd(\cM)$ consists of $\kk$-linear endofunctors and natural transformations.

\subsection{Presentations of module categories}

We now describe the notion of a \emph{presentation} of a $\cC$-module by generators and relations.  We restrict ourselves to the case where there are only generating morphisms (i.e., no nontrivial generating objects).  Let $\mathtt{M}$ be the set of generating morphisms.  These are formal morphisms whose domains and codomains are objects of $\cC$.  For objects $\lambda$ and $\mu$, let $\mathtt{M}(\lambda,\mu)$ denote the subset of $\mathtt{M}$ consisting of the morphisms with domain $\lambda$ and codomain $\mu$.

Let $\cM$ be the category with the same objects as $\cC$, generated by the morphisms of $\cC$ and formal morphisms
\[
    f \otimes g \colon \lambda \otimes \lambda' \to \mu \otimes \mu',\qquad
    f \in \mathtt{M}(\lambda,\mu),\quad g \in \cC(\lambda',\mu'),
\]
for $\lambda,\lambda',\mu,\mu' \in \Ob(\cC)$.  We retain all composition and identity relations of $\cC$, write $f$ for $f \otimes 1_\one$, and impose the relations
\begin{equation} \label{deer}
    (1_Y \otimes g) \circ (f \otimes h)
    = f \otimes (g \circ h)
    = (f \otimes g) \circ (1_X \otimes h),
\end{equation}
for all $f \colon X \to Y$ in $\mathtt{M}$ and all composable morphisms $g,h$ in $\cC$.  The category $\cM$ has a natural structure of a strict $\cC$-module extending the tensor product of $\cC$, with the action on the formal generators given by
\[
    (f \otimes g) \otimes h := f \otimes (g \otimes h),\quad f \in \mathtt{M},\ g,h \in \Mor(\cC).
\]

Now, let $\mathtt{R}$ be a set of pairs of parallel morphisms in $\cM$, representing the equalities to be imposed.  Let $\hat{\mathtt{R}}$ be the smallest family of equivalence relations on the hom-sets of $\cM$ containing $\mathtt{R}$ and compatible with composition on either side and the right $\cC$-action.  We define the \emph{$\cC$-module generated by the morphisms $\mathtt{M}$ subject to the relations $\mathtt{R}$} to be the quotient category $\cM/\hat{\mathtt{R}}$.  The right $\cC$-action descends to this quotient.

In the $\kk$-linear setting, allow finite $\kk$-linear combinations of morphisms and impose bilinearity of composition and the right action, while preserving the given linear structure of $\cC$.  Relations can then be expressed as linear combinations of parallel morphisms set equal to zero; take their closure under $\kk$-linear combinations, composition on either side, and the right action.

Using the usual string diagram notation for module categories, the relations \cref{deer} become
\begin{equation} \label{snow}
    \begin{tikzpicture}[centerzero]
        \draw (0,-1) -- (0,-0.4) node[draw,fill=white,rounded corners] {$\scriptstyle{f}$} -- (0,1);
        \draw (0.7,-1) -- (0.7,-0.4) node[draw,fill=white,rounded corners] {$\scriptstyle{h}$} -- (0.7,0.4) node[draw,fill=white,rounded corners] {$\scriptstyle{g}$} -- (0.7,1);
    \end{tikzpicture}
    \ =\
    \begin{tikzpicture}[centerzero]
        \draw (0,-1) -- (0,0) node[draw,fill=white,rounded corners] {$\scriptstyle{f}$} -- (0,1);
        \draw (0.8,-1) -- (0.8,0) node[draw,fill=white,rounded corners] {$\scriptstyle{g \circ h}$} -- (0.8,1);
    \end{tikzpicture}
    \ =\
    \begin{tikzpicture}[centerzero]
        \draw (0,-1) -- (0,0.4) node[draw,fill=white,rounded corners] {$\scriptstyle{f}$} -- (0,1);
        \draw (0.7,-1) -- (0.7,-0.4) node[draw,fill=white,rounded corners] {$\scriptstyle{h}$} -- (0.7,0.4) node[draw,fill=white,rounded corners] {$\scriptstyle{g}$} -- (0.7,1);
    \end{tikzpicture}
    \ ,\qquad
    f \in \mathtt{M},\ g,h \in \Mor(\cC).
\end{equation}

\subsection{The affinization as a module category}

Suppose $\cC$ is a braided strict monoidal category and $\sigma \colon \cC \to \cC$ is a strict braided monoidal automorphism.  The twisted affinization $\Aff_\sigma(\cC)$ is naturally a right $\cC$-module with action given on objects by the tensor product and on morphisms by
\[
    f \otimes g := f \otimes \imath_\cC(g),\qquad
    f \in \Mor(\Aff_\sigma(\cC)),\ g \in \Mor(\cC).
\]

\begin{theo} \label{jungle}
    As a right $\cC$-module, the twisted affinization $\Aff_\sigma(\cC)$ is generated by invertible morphisms
    \begin{equation} \label{pole}
        \begin{tikzpicture}[centerzero]
            \draw (0,-0.3) \botlabel{X^\sigma} -- (0,0.3) \toplabel{X};
            \posdot{0,0};
        \end{tikzpicture}
        \colon X^\sigma \to X,\qquad
        X \in \Ob(\cC),
    \end{equation}
    subject to the relations
    \begin{equation} \label{terraform}
        \begin{tikzpicture}[centerzero]
            \draw (0,-0.8) \botlabel{(X \otimes Y)^\sigma} -- (0,0.8) \toplabel{X \otimes Y};
            \posdot{0,0};
        \end{tikzpicture}
        =
        \begin{tikzpicture}[centerzero]
            \draw (-0.25,0) \braidup (0.25,0.5) -- (0.25,0.8) \toplabel{Y};
            \draw[over] (-0.25,-0.8) \botlabel{X^\sigma} -- (-0.25,-0.5) \braidup (0.25,0) \braidup (-0.25,0.5) -- (-0.25,0.8) \toplabel{X};
            \draw[over] (0.25,-0.8) \botlabel{Y^\sigma} -- (0.25,-0.5) \braidup (-0.25,0);
            \posdot{-0.25,-0.5};
            \posdot{-0.25,0};
        \end{tikzpicture}
        =
        \begin{tikzpicture}[centerzero]
            \draw (-0.25,0) \braidup (0.25,0.5) -- (0.25,0.8) \toplabel{Y};
            \draw[over] (-0.25,-0.8) \botlabel{X^\sigma} -- (-0.25,-0.5) \braidup (0.25,0) \braidup (-0.25,0.5) -- (-0.25,0.8) \toplabel{X};
            \draw[over] (0.25,-0.8) \botlabel{Y^\sigma} -- (0.25,-0.5) \braidup (-0.25,0);
            \posdot{-0.25,0.5};
            \posdot{-0.25,0};
        \end{tikzpicture}
        \ ,\qquad
        \begin{tikzpicture}[centerzero]
            \draw (0,-0.8) \botlabel{X^\sigma} -- (0,0.8) \toplabel{Y};
            \posdot{0,-0.3};
            \coupon{0,0.3}{f};
        \end{tikzpicture}
        =
        \begin{tikzpicture}[centerzero]
            \draw (0,-0.8) \botlabel{X^\sigma} -- (0,0.8) \toplabel{Y};
            \posdot{0,0.3};
            \coupon{0,-0.3}{f^\sigma};
        \end{tikzpicture}
        \ ,
    \end{equation}
    for all $X,Y \in \Ob(\cC)$ and $f \in \cC(X,Y)$.
\end{theo}

The second equality in \cref{terraform} is called the \emph{twisted reflection equation}.

\begin{proof}
    Let $\Mcat$ be the $\cC$-module generated by the morphisms \cref{pole} subject to the relations \cref{terraform}.  We will draw morphisms in $\Mcat$ in cyan, to help distinguish them from morphisms in $\Aff_\sigma(\cC)$.  Define a strict morphism of $\cC$-modules $\bF \colon \Mcat \to \Aff_\sigma(\cC)$ by fixing objects, sending each morphism $f$ of $\cC$ to $\imath_\cC(f)$, and
    \begin{equation} \label{Fdef}
        \bF \colon
        \begin{tikzpicture}[MC,centerzero]
            \draw (0,-0.3) \botlabel{X^\sigma} -- (0,0.3) \toplabel{X};
            \posdot{0,0};
        \end{tikzpicture}
        \mapsto
        \begin{tikzpicture}[centerzero]
            \draw (0,-0.3) \botlabel{X^\sigma} -- (0,0.3) \toplabel{X};
            \posdot{0,0};
        \end{tikzpicture}
        \ ,\qquad X \in \Ob(\cC).
    \end{equation}
    We first show that this functor is well defined.  Since it clearly respects \cref{snow}, it remains to verify that it respects the relations \cref{terraform}.  In $\Aff_\sigma(\cC)$, we have, for all objects $X$ and $Y$,
    \begin{gather*}
        \begin{tikzpicture}[centerzero]
            \draw (-0.25,0) \braidup (0.25,0.5) -- (0.25,0.8) \toplabel{Y};
            \draw[over] (-0.25,-0.8) \botlabel{X^\sigma} -- (-0.25,-0.5) \braidup (0.25,0) \braidup (-0.25,0.5) -- (-0.25,0.8) \toplabel{X};
            \draw[over] (0.25,-0.8) \botlabel{Y^\sigma} -- (0.25,-0.5) \braidup (-0.25,0);
            \posdot{-0.25,-0.5};
            \posdot{-0.25,0};
        \end{tikzpicture}
        \ceqref{jaguar+}
        \begin{tikzpicture}[centerzero]
            \draw (-0.25,-0.8) \botlabel{X^\sigma} -- (-0.25,-0.5) \braidup (0.25,0) \braidup (-0.25,0.5) -- (-0.25,0.8) \toplabel{X};
            \draw[over] (0.25,-0.8) \botlabel{Y^\sigma} -- (0.25,-0.5) \braidup (-0.25,0) \braidup (0.25,0.5) -- (0.25,0.8) \toplabel{Y};
            \posdot{-0.25,-0.5};
            \posdot{0.25,0.5};
        \end{tikzpicture}
        \ceqref{Reid2}
        \begin{tikzpicture}[centerzero]
            \draw (-0.25,-0.8) \botlabel{X^\sigma} -- (-0.25,0.8) \toplabel{X};
            \draw[over] (0.25,-0.8) \botlabel{Y^\sigma} -- (0.25,0.8) \toplabel{Y};
            \posdot{-0.25,0};
            \posdot{0.25,0};
        \end{tikzpicture}
        \ceqref{jaguar+}
        \begin{tikzpicture}[centerzero]
            \draw (0,-0.8) \botlabel{(X \otimes Y)^\sigma} -- (0,0.8) \toplabel{X \otimes Y};
            \posdot{0,0};
        \end{tikzpicture}
        \ ,\\[1ex]
        \begin{tikzpicture}[centerzero]
            \draw (-0.25,0) \braidup (0.25,0.5) -- (0.25,0.8) \toplabel{Y};
            \draw[over] (-0.25,-0.8) \botlabel{X^\sigma} -- (-0.25,-0.5) \braidup (0.25,0) \braidup (-0.25,0.5) -- (-0.25,0.8) \toplabel{X};
            \draw[over] (0.25,-0.8) \botlabel{Y^\sigma} -- (0.25,-0.5) \braidup (-0.25,0);
            \posdot{-0.25,0.5};
            \posdot{-0.25,0};
        \end{tikzpicture}
        \ceqref{jaguar+}
        \begin{tikzpicture}[centerzero]
            \draw (0.25,-0.8) \botlabel{Y^\sigma} -- (0.25,-0.5) \braidup (-0.25,0) \braidup (0.25,0.5) -- (0.25,0.8) \toplabel{Y};
            \draw[over] (-0.25,-0.8) \botlabel{X^\sigma} -- (-0.25,-0.5) \braidup (0.25,0) \braidup (-0.25,0.5) -- (-0.25,0.8) \toplabel{X};
            \posdot{-0.25,0.5};
            \posdot{0.25,-0.5};
        \end{tikzpicture}        
        \ceqref{Reid2}
        \begin{tikzpicture}[centerzero]
            \draw (-0.25,-0.8) \botlabel{X^\sigma} -- (-0.25,0.8) \toplabel{X};
            \draw[over] (0.25,-0.8) \botlabel{Y^\sigma} -- (0.25,0.8) \toplabel{Y};
            \posdot{-0.25,0};
            \posdot{0.25,0};
        \end{tikzpicture}
        \ceqref{jaguar+}
        \begin{tikzpicture}[centerzero]
            \draw (0,-0.8) \botlabel{(X \otimes Y)^\sigma} -- (0,0.8) \toplabel{X \otimes Y};
            \posdot{0,0};
        \end{tikzpicture}
        \ .
    \end{gather*}
    Thus, $\bF$ respects the first and second relations in \cref{terraform}.   It also respects the last relation in \cref{terraform} by the last relation in \cref{jaguar+}.
    
    We now construct an inverse.  A priori, positive and negative dots can only appear on the left side of diagrams in $\Mcat$.  For $X,Y \in \Ob(\cC)$, we define the following morphism from $X^\sigma \otimes Y^\sigma$ to $X^\sigma \otimes Y$:
    \begin{equation} \label{mouse1}
        \begin{tikzpicture}[MC,centerzero]
            \draw (-0.2,-0.45) \botlabel{X^\sigma} -- (-0.2,0.45) \toplabel{X^\sigma};
            \draw (0.2,-0.45) \botlabel{Y^\sigma} -- (0.2,0.45) \toplabel{Y};
            \posdot{0.2,0};
        \end{tikzpicture}
        :=
        \begin{tikzpicture}[MC,centerzero]
            \draw (-0.2,-0.45) \botlabel{X^\sigma} \braidup (0.2,0);
            \draw[over] (0.2,-0.45) \botlabel{Y^\sigma} \braidup (-0.2,0) \braidup (0.2,0.45) \toplabel{Y};
            \draw[over] (0.2,0) \braidup (-0.2,0.45) \toplabel{X^\sigma};
            \posdot{-0.2,0};
        \end{tikzpicture}
        \ .
    \end{equation}
    Since every object of $\cC$ is of the form $X^\sigma$, this defines a positive dot to the right of any object.  Negative dots in that position denote the inverse morphisms.  It follows that, in $\Mcat$,
    \begin{equation} \label{mouse2}
        \begin{tikzpicture}[MC,centerzero]
            \draw (0.2,-0.4) \botlabel{Y^\sigma} -- (0.2,0.4) \toplabel{Y};
            \draw (-0.2,-0.4) \botlabel{X^\sigma} -- (-0.2,0.4) \toplabel{X};
            \posdot{-0.2,0.1};
            \posdot{0.2,-0.1};
        \end{tikzpicture}
        \ceqref{mouse1}
        \begin{tikzpicture}[MC,centerzero]
            \draw (-0.2,0) \braidup (0.2,0.5) -- (0.2,0.8) \toplabel{Y};
            \draw[over] (-0.2,-0.8) \botlabel{X^\sigma} -- (-0.2,-0.5) \braidup (0.2,0) \braidup (-0.2,0.5) -- (-0.2,0.8) \toplabel{X};
            \draw[over] (0.2,-0.8) \botlabel{Y^\sigma} -- (0.2,-0.5) \braidup (-0.2,0);
            \posdot{-0.2,0.5};
            \posdot{-0.2,0};
        \end{tikzpicture}
        \ceqref{terraform}
        \begin{tikzpicture}[MC,centerzero]
            \draw (0,-0.4) \botlabel{(X \otimes Y)^\sigma} -- (0,0.4) \toplabel{X \otimes Y};
            \posdot{0,0};
        \end{tikzpicture}
        \ceqref{terraform}
        \begin{tikzpicture}[MC,centerzero]
            \draw (-0.2,0) \braidup (0.2,0.5) -- (0.2,0.8) \toplabel{Y};
            \draw[over] (-0.2,-0.8) \botlabel{X^\sigma} -- (-0.2,-0.5) \braidup (0.2,0) \braidup (-0.2,0.5) -- (-0.2,0.8) \toplabel{X};
            \draw[over] (0.2,-0.8) \botlabel{Y^\sigma} -- (0.2,-0.5) \braidup (-0.2,0);
            \posdot{-0.2,-0.5};
            \posdot{-0.2,0};
        \end{tikzpicture}
        \ceqref{mouse1}
        \begin{tikzpicture}[MC,centerzero]
            \draw (0.2,-0.4) \botlabel{Y^\sigma} -- (0.2,0.4) \toplabel{Y};
            \draw (-0.2,-0.4) \botlabel{X^\sigma} -- (-0.2,0.4) \toplabel{X};
            \posdot{-0.2,-0.1};
            \posdot{0.2,0.1};
        \end{tikzpicture}
        \ .
    \end{equation}
    
    Motivated by \cref{stealth}, we define a functor
    \[
        \bG \colon \Aff_\sigma(\cC) \to \Mcat
    \]
    by sending objects and morphisms of $\cC$ to themselves, and by
    \begin{equation} \label{Gdef}
        \bG(\xi_{X,Y})
        :=
        \begin{tikzpicture}[MC,centerzero]
            \draw (-0.2,-0.35) \botlabel{X} \braidup (0.2,0.05) -- (0.2,0.35) \toplabel{X};
            \draw[over] (0.2,-0.35) \botlabel{Y^\sigma} \braidup (-0.2,0.05) -- (-0.2,0.35) \toplabel{Y};
            \posdot{-0.2,0.1};
        \end{tikzpicture}
        \ ,\qquad X,Y \in \Ob(\cC).
    \end{equation}
    Since the right-hand side above is invertible, this also determines the image of $\xi_{X,Y}^{-1}$.
    
    We check that $\bG$ respects \cref{coilrel1,coilrel2}.  For \cref{coilrel1},
    \[
        \bG( \xi_{X,Y \otimes Z} )
        \ceqref{Gdef}
        \begin{tikzpicture}[MC,centerzero]
            \draw (-0.3,-0.5) \botlabelshift{-3pt}{X} \braidup (0.3,0.1) -- (0.3,0.5) \toplabelshift{3pt}{X};
            \draw[over] (0.3,-0.5) \botlabelshift{6pt}{(Y \otimes Z)^\sigma} \braidup (-0.3,0.1) -- (-0.3,0.5) \toplabel{Y \otimes Z};
            \posdot{-0.3,0.2};
        \end{tikzpicture}
        \ceqref{mouse2}
        \begin{tikzpicture}[MC,centerzero]
            \draw (-0.4,-1) \botlabel{X} \braidup (0.4,0.2) -- (0.4,1) \toplabel{X};
            \draw[over] (0,-1) \botlabel{Y^\sigma} \braidup (-0.4,-0.4) \braidup (0,0.2);
            \draw[over] (0.4,-1) \botlabel{Z^\sigma} \braidup (-0.4,0.2) \braidup (0,0.8) -- (0,1) \toplabel{Z};
            \draw[over] (0,0.2) \braidup (-0.4,0.8) -- (-0.4,1) \toplabel{Y};
            \posdot{-0.4,0.2};
            \posdot{-0.4,0.8};
        \end{tikzpicture}
        \ceqref{Reid3}
        \begin{tikzpicture}[MC,centerzero]
            \draw (-0.4,-1) \botlabel{X} \braidup (0.4,1) \toplabel{X};
            \draw[over] (0,-1) \botlabel{Y^\sigma} \braidup (0.4,-0.1);
            \draw[over] (0.4,-1) \botlabel{Z^\sigma} \braidup (-0.4,0.1) \braidup (0,1) \toplabel{Z};
            \draw[over] (0.4,-0.1) \braidup (-0.4,0.8) -- (-0.4,1) \toplabel{Y};
            \posdot{-0.4,0.1};
            \posdot{-0.4,0.8};
        \end{tikzpicture}
        \ceqref{Gdef}
        \bG(\xi_{Z \otimes X,Y}) \circ \bG(\xi_{X \otimes Y^\sigma,Z}).
    \]
    For \cref{coilrel2}, let $g \in \cC(X_1,X_2)$ and $f \in \cC(Y_1,Y_2)$.  Then
    \[
        \bG(\xi_{X_2,Y_2}) \circ (g \otimes f^\sigma)
        \ceqref{Gdef}
        \begin{tikzpicture}[MC,centerzero]
            \draw (-0.3,-0.6) \botlabel{X_1} -- (-0.3,-0.1) \braidup (0.3,0.4) -- (0.3,0.6) \toplabel{X_2};
            \draw[over] (0.3,-0.6) \botlabel{Y_1^\sigma} -- (0.3,-0.1) \braidup (-0.3,0.4) -- (-0.3,0.6) \toplabel{Y_2};
            \posdot{-0.3,0.4};
            \coupon{0.3,-0.3}{f^\sigma};
            \coupon{-0.3,-0.3}{g};
        \end{tikzpicture}
        \ceqref{lynx}
        \begin{tikzpicture}[MC,centerzero]
            \draw (-0.3,-0.6) \botlabel{X_1} \braidup (0.3,-0.1) -- (0.3,0.6) \toplabel{X_2};
            \draw[over] (0.3,-0.6) \botlabel{Y_1^\sigma} \braidup (-0.3,-0.1) -- (-0.3,0.6) \toplabel{Y_2};
            \posdot{-0.3,0.4};
            \coupon{0.3,0}{g};
            \coupon{-0.3,0}{f^\sigma};
        \end{tikzpicture}
        \ceqref{terraform}
        \begin{tikzpicture}[MC,centerzero]
            \draw (-0.3,-0.6) \botlabel{X_1} \braidup (0.3,-0.1) -- (0.3,0.6) \toplabel{X_2};
            \draw[over] (0.3,-0.6) \botlabel{Y_1^\sigma} \braidup (-0.3,-0.1) -- (-0.3,0.6) \toplabel{Y_2};
            \posdot{-0.3,-0.1};
            \coupon{0.3,0.3}{g};
            \coupon{-0.3,0.3}{f};
        \end{tikzpicture}
        \ceqref{Gdef} (f \otimes g) \circ \bG(\xi_{X_1,Y_1}).
    \]
    
    The functor $\bG$ also respects the right $\cC$-action.  Indeed, \cref{stack,Gdef,hexagon} give
    \[
        \bG(\xi_{X,Y}\otimes1_Z)
        =
        \begin{tikzpicture}[MC,centerzero]
            \draw (-0.4,-0.4) \botlabel{X} \braidup (0,0.05) -- (0,0.4) \toplabel{X};
            \draw[over] (0,-0.4) \botlabel{Y^\sigma} \braidup (-0.4,0.05) -- (-0.4,0.4) \toplabel{Y};
            \draw (0.4,-0.4) \botlabel{Z} -- (0.4,0.4) \toplabel{Z};
            \posdot{-0.4,0.12};
        \end{tikzpicture}
        =
        \bG(\xi_{X,Y})\otimes1_Z.
    \]
    The same holds for inverse coils and for morphisms of $\cC$, and hence for all morphisms.  Thus, $\bG$ is a strict morphism of $\cC$-modules.  It remains to show that $\bF$ and $\bG$ are inverse.  They are clearly inverse on objects and on morphisms of $\cC$.  For a coil generator, we have
    \[
        \bF\bG(\xi_{X,Y})
        \ceqref{Gdef} \bF
        \left(
            \begin{tikzpicture}[MC,centerzero]
                \draw (-0.2,-0.35) \botlabel{X} \braidup (0.2,0.05) -- (0.2,0.35) \toplabel{X};
                \draw[over] (0.2,-0.35) \botlabel{Y^\sigma} \braidup (-0.2,0.05) -- (-0.2,0.35) \toplabel{Y};
                \posdot{-0.2,0.1};
            \end{tikzpicture}
        \right)
        \ceqref{Fdef}
        \begin{tikzpicture}[centerzero]
            \draw (-0.2,-0.35) \botlabel{X} \braidup (0.2,0.05) -- (0.2,0.35) \toplabel{X};
            \draw[over] (0.2,-0.35) \botlabel{Y^\sigma} \braidup (-0.2,0.05) -- (-0.2,0.35) \toplabel{Y};
            \posdot{-0.2,0.1};
        \end{tikzpicture}
        \ceqref{stealth}
        \xi_{X,Y}. 
    \]
    Hence $\bF \bG= \id_{\Aff_\sigma(\cC)}$, since these functors agree on the generators in \cref{affdef}.
    
    Conversely, if $g \in \cC(Y,Z)$, then
    \[
        \bG \bF
        \left(
            \begin{tikzpicture}[MC,centerzero]
                \draw (-0.2,-0.5) \botlabel{X^\sigma} -- (-0.2,0.5) \toplabel{X};
                \draw (0.2,-0.5) \botlabel{Y} -- (0.2,0.5) \toplabel{Z};
                \posdot{-0.2,0};
                \coupon{0.2,0}{g};
            \end{tikzpicture}
        \right)
        = \bG
        \left(
            \begin{tikzpicture}[centerzero]
                \draw (-0.2,-0.5) \botlabel{X^\sigma} -- (-0.2,0.5) \toplabel{X};
                \draw (0.2,-0.5) \botlabel{Y} -- (0.2,0.5) \toplabel{Z};
                \posdot{-0.2,0};
                \coupon{0.2,0}{g};
            \end{tikzpicture}
        \right)
        \ceqref{stealth} (1_X \otimes g) \circ \bG(\xi_{Y,X}) \circ \beta_{X^\sigma,Y}
        \ceqref{Gdef}
        \begin{tikzpicture}[MC,centerzero]
            \draw (-0.2,-0.5) \botlabel{X^\sigma} -- (-0.2,0.5) \toplabel{X};
            \draw (0.2,-0.5) \botlabel{Y} -- (0.2,0.5) \toplabel{Z};
            \posdot{-0.2,0};
            \coupon{0.2,0}{g};
        \end{tikzpicture}
        \ .
    \]
    Thus, $\bG\bF = \id_{\Mcat}$, since these module functors agree on the generators of $\Mcat$.
\end{proof}

\begin{cor} \label{jungleneg}
    As a right $\cC$-module, the twisted affinization $\Aff_\sigma(\cC)$ is generated by invertible morphisms
    \begin{equation}
        \begin{tikzpicture}[centerzero]
            \draw (0,-0.3) \botlabel{X} -- (0,0.3) \toplabel{X^\sigma};
            \negdot{0,0};
        \end{tikzpicture}
        \colon X \to X^\sigma,\qquad
        X \in \Ob(\cC),
    \end{equation}
    subject to the relations
    \begin{equation}\label{griffin}
        \begin{tikzpicture}[centerzero]
            \draw (0,-0.8) \botlabel{X \otimes Y} -- (0,0.8) \toplabel{(X \otimes Y)^\sigma};
            \negdot{0,0};
        \end{tikzpicture}
        =
        \begin{tikzpicture}[centerzero]
            \draw (0.25,-0.8) \botlabel{Y} -- (0.25,-0.5) \braidup (-0.25,0);
            \draw[over] (-0.25,-0.8) \botlabel{X} -- (-0.25,-0.5) \braidup (0.25,0) \braidup (-0.25,0.5) -- (-0.25,0.8) \toplabel{X^\sigma};
            \draw[over] (-0.25,0) \braidup (0.25,0.5) -- (0.25,0.8) \toplabel{Y^\sigma};
            \negdot{-0.25,-0.5};
            \negdot{-0.25,0};
        \end{tikzpicture}
        =
        \begin{tikzpicture}[centerzero]
            \draw (0.25,-0.8) \botlabel{Y} -- (0.25,-0.5) \braidup (-0.25,0);
            \draw[over] (-0.25,-0.8) \botlabel{X} -- (-0.25,-0.5) \braidup (0.25,0) \braidup (-0.25,0.5) -- (-0.25,0.8) \toplabel{X^\sigma};
            \draw[over] (-0.25,0) \braidup (0.25,0.5) -- (0.25,0.8) \toplabel{Y^\sigma};
            \negdot{-0.25,0.5};
            \negdot{-0.25,0};
        \end{tikzpicture}
        \ ,\qquad
        \begin{tikzpicture}[centerzero]
            \draw (0,-0.8) \botlabel{X} -- (0,0.8) \toplabel{Y^\sigma};
            \negdot{0,-0.3};
            \coupon{0,0.3}{f^\sigma};
        \end{tikzpicture}
        =
        \begin{tikzpicture}[centerzero]
            \draw (0,-0.8) \botlabel{X} -- (0,0.8) \toplabel{Y^\sigma};
            \negdot{0,0.3};
            \coupon{0,-0.3}{f};
        \end{tikzpicture}
        \ ,
    \end{equation}
    for all $X,Y \in \Ob(\cC)$ and $f \in \cC(X,Y)$.
\end{cor}

\begin{proof}
    This follows from inverting the generators and identities in \cref{jungle}.
\end{proof}

\subsection{Simplified module-category presentation of the affinization}

We continue to suppose that $\cC$ is a braided strict monoidal category and that $\sigma \colon \cC \to \cC$ is a strict braided monoidal automorphism.  When $\cC$ is given by generators and relations, the presentation of $\Aff_\sigma(\cC)$ described by \cref{jungle,jungleneg} can be simplified.  Throughout this subsection, we assume that the objects of $\cC$ are freely generated under tensor product by a set $S_0$, and that $S_1$ is a set of morphisms generating $\cC$ as a monoidal category, or as a $\kk$-linear monoidal category in the linear setting.  Let $S_1'$ be obtained from $S_1$ by removing any braiding morphisms and their inverses.

\begin{prop}[Module-category positive-dot presentation]\label{swamp+}
    As a right $\cC$-module, the twisted affinization $\Aff_\sigma(\cC)$ is generated by invertible morphisms
    \[
        \begin{tikzpicture}[centerzero]
            \draw (0,-0.3) \botlabel{X^\sigma} -- (0,0.3) \toplabel{X};
            \posdot{0,0};
        \end{tikzpicture}
        \colon X^\sigma \to X,\qquad
        X \in S_0,
    \]
    subject to the relations
    \begin{equation} \label{life}
        \begin{tikzpicture}[centerzero]
            \draw (-0.25,0) \braidup (0.25,0.5) -- (0.25,0.8) \toplabel{Y};
            \draw[over] (-0.25,-0.8) \botlabel{X^\sigma} -- (-0.25,-0.5) \braidup (0.25,0) \braidup (-0.25,0.5) -- (-0.25,0.8) \toplabel{X};
            \draw[over] (0.25,-0.8) \botlabel{Y^\sigma} -- (0.25,-0.5) \braidup (-0.25,0);
            \posdot{-0.25,-0.5};
            \posdot{-0.25,0};
        \end{tikzpicture}
        =
        \begin{tikzpicture}[centerzero]
            \draw (-0.25,0) \braidup (0.25,0.5) -- (0.25,0.8) \toplabel{Y};
            \draw[over] (-0.25,-0.8) \botlabel{X^\sigma} -- (-0.25,-0.5) \braidup (0.25,0) \braidup (-0.25,0.5) -- (-0.25,0.8) \toplabel{X};
            \draw[over] (0.25,-0.8) \botlabel{Y^\sigma} -- (0.25,-0.5) \braidup (-0.25,0);
            \posdot{-0.25,0.5};
            \posdot{-0.25,0};
        \end{tikzpicture}
        \ ,\qquad
        \begin{tikzpicture}[centerzero]
            \draw (0,-0.8) \botlabel{W^\sigma} -- (0,0.8) \toplabel{Z};
            \posdot{0,-0.3};
            \coupon{0,0.3}{f};
        \end{tikzpicture}
        =
        \begin{tikzpicture}[centerzero]
            \draw (0,-0.8) \botlabel{W^\sigma} -- (0,0.8) \toplabel{Z};
            \posdot{0,0.3};
            \coupon{0,-0.3}{f^\sigma};
        \end{tikzpicture}
        \ ,
    \end{equation}
    for all $X,Y \in S_0$ and $f \colon W \to Z $ in $S_1'$.  The positive dots on arbitrary objects are defined recursively by the first equality in \cref{terraform} with $X \in S_0$, and the positive dot on $\one$ is $1_\one$.
\end{prop}

\begin{proof}
    We first show that the first two equalities in \cref{terraform} hold for all objects of $\cC$. Since the objects of $\cC$ are freely generated by $S_0$, the positive dot on an object $X_1 \otimes X_2 \otimes \dotsb \otimes X_n$, $X_i \in S_0$, is defined by the first equality in \cref{terraform} with $X = X_1$ and $Y = X_2 \otimes \dotsb \otimes X_n$. The first equality in \cref{terraform} then holds for an arbitrary decomposition $X \otimes Y$ of an object; this follows by induction on the length of $X$, using the identities \cref{hexagon} and the braid relation \cref{Reid3}.

    We prove the twisted reflection equation for a pair of objects $W,Z$ by induction on the total length of $W$ and $Z$. If $W = \one$ or $Z = \one$ there is nothing to prove, and if $W, Z \in S_0$ it is the first relation in \cref{life}. Suppose first that $W = X \otimes Y$ with $X, Y \ne \one$. By induction, the twisted reflection equation holds for $X$ and $Y$, for $X$ and $Z$, and for $Y$ and $Z$. Since it holds for $X$ and $Y$, the dot on $W$ is also equal to the right-hand diagram in \cref{terraform}. Then
    \[
        \begin{tikzpicture}[centerzero]
            \draw (-0.25,0) \braidup (0.25,0.5) -- (0.25,0.8) \toplabel{Z};
            \draw[over] (-0.25,-0.8) \botlabel{W^\sigma} -- (-0.25,-0.5) \braidup (0.25,0) \braidup (-0.25,0.5) -- (-0.25,0.8) \toplabel{W};
            \draw[over] (0.25,-0.8) \botlabel{Z^\sigma} -- (0.25,-0.5) \braidup (-0.25,0);
            \posdot{-0.25,-0.5};
            \posdot{-0.25,0};
        \end{tikzpicture}
        =
        \begin{tikzpicture}[anchorbase]
            \draw (-0.5,0) \botlabel{X^\sigma} \braidup (0,0.5);
            \draw[over] (0,0) \botlabel{Y^\sigma} \braidup (-0.5,0.5) \braidup (0.4,1.5);
            \draw[over] (0,0.5) \braidup (-0.5,1) \braidup (0,1.5);
            \draw[over] (0.5,0) \botlabel{Z^\sigma} -- (0.5,0.8) \braidup (-0.5,1.5) \braidup (0.5,2.2) -- (0.5,2.3) \toplabel{Z};
            \draw[over] (0,1.5) \braidup (-0.5,2) -- (-0.5,2.3) \toplabel{X};
            \draw[over] (0.4,1.5) \braidup (0,2.2) -- (0,2.3) \toplabel{Y};
            \posdot{-0.5,0.5};
            \posdot{-0.5,1};
            \posdot{-0.5,1.5};
        \end{tikzpicture}
        =
        \begin{tikzpicture}[anchorbase]
            \draw (-0.5,0) \botlabel{X^\sigma} \braidup (0,0.5);
            \draw[over] (0,0) \botlabel{Y^\sigma} \braidup (-0.5,0.5) \braidup (0.4,1.5);
            \draw[over] (0,0.5) \braidup (-0.5,1) \braidup (0,1.5);
            \draw[over] (0.5,0) \botlabel{Z^\sigma} -- (0.5,0.8) \braidup (-0.5,1.5) \braidup (0.5,2.2) -- (0.5,2.3) \toplabel{Z};
            \draw[over] (0,1.5) \braidup (-0.5,2) -- (-0.5,2.3) \toplabel{X};
            \draw[over] (0.4,1.5) \braidup (0,2.2) -- (0,2.3) \toplabel{Y};
            \posdot{-0.5,0.5};
            \posdot{-0.5,2};
            \posdot{-0.5,1.5};
        \end{tikzpicture}
        =
        \begin{tikzpicture}[anchorbase]
            \draw (-0.5,0) \botlabel{X^\sigma} \braidup (0.4,1);
            \draw[over] (0,0) \botlabel{Y^\sigma} \braidup (-0.5,0.5) \braidup (0,1);
            \draw[over] (0.5,0) \botlabel{Z^\sigma} -- (0.5,0.3) \braidup (-0.5,1) \braidup (0.5,2) -- (0.5,2.3) \toplabel{Z};
            \draw[over] (0,1) \braidup (-0.5,1.5) \braidup (0,2.2) -- (0,2.3) \toplabel{Y};
            \draw[over] (0.4,1) \braidup (-0.5,2) -- (-0.5,2.3) \toplabel{X};
            \posdot{-0.5,0.5};
            \posdot{-0.5,1};
            \posdot{-0.5,2};
        \end{tikzpicture}
        =
        \begin{tikzpicture}[anchorbase]
            \draw (-0.5,0) \botlabel{X^\sigma} \braidup (0.4,1);
            \draw[over] (0,0) \botlabel{Y^\sigma} \braidup (-0.5,0.5) \braidup (0,1);
            \draw[over] (0.5,0) \botlabel{Z^\sigma} -- (0.5,0.3) \braidup (-0.5,1) \braidup (0.5,2) -- (0.5,2.3) \toplabel{Z};
            \draw[over] (0,1) \braidup (-0.5,1.5) \braidup (0,2.2) -- (0,2.3) \toplabel{Y};
            \draw[over] (0.4,1) \braidup (-0.5,2) -- (-0.5,2.3) \toplabel{X};
            \posdot{-0.5,1.5};
            \posdot{-0.5,1};
            \posdot{-0.5,2};
        \end{tikzpicture}
        =
        \begin{tikzpicture}[anchorbase]
            \draw (-0.5,0.5) \braidup (0.5,1.5) -- (0.5,2.3) \toplabel{Z};
            \draw[over] (-0.5,0) \botlabel{X^\sigma} \braidup (0,0.5) \braidup (-0.5,1);
            \draw[over] (0,0) \botlabel{Y^\sigma} \braidup (0.5,0.5) \braidup (0,1);
            \draw[over] (0.5,0) \botlabel{Z^\sigma} \braidup (-0.5,0.5);
            \draw[over] (-0.5,1.5) \braidup (0,2.2) -- (0,2.3) \toplabel{Y};
            \draw[over] (-0.5,1) \braidup (0,1.5) \braidup (-0.5,2) -- (-0.5,2.3) \toplabel{X};
            \draw[over] (0,1) \braidup (-0.5,1.5);
            \posdot{-0.5,1.5};
            \posdot{-0.5,0.5};
            \posdot{-0.5,2};
        \end{tikzpicture}
        =
        \begin{tikzpicture}[centerzero]
            \draw (-0.25,0) \braidup (0.25,0.5) -- (0.25,0.8) \toplabel{Z};
            \draw[over] (-0.25,-0.8) \botlabel{W^\sigma} -- (-0.25,-0.5) \braidup (0.25,0) \braidup (-0.25,0.5) -- (-0.25,0.8) \toplabel{W};
            \draw[over] (0.25,-0.8) \botlabel{Z^\sigma} -- (0.25,-0.5) \braidup (-0.25,0);
            \posdot{-0.25,0.5};
            \posdot{-0.25,0};
        \end{tikzpicture}
    \]
    where the first and last equalities use the right-hand diagram in \cref{terraform} for the dot on $W$ together with the identities \cref{hexagon}, the second equality uses the twisted reflection equation for $X$ and $Z$, the third and fifth equalities use the braid relation \cref{Reid3}, and the fourth equality uses the twisted reflection equation for $Y$ and $Z$. If instead the first object of the pair lies in $S_0$, we write the pair as $(Z,W)$ with $Z \in S_0$ and $W = X \otimes Y$, $X,Y \ne \one$. A similar argument, with the dot on $W$ now in the middle of the diagram and using the twisted reflection equations for $X$ and $Y$, for $Z$ and $X$, and for $Z$ and $Y$, all of which hold by induction, shows that
    \[
        \begin{tikzpicture}[centerzero]
            \draw (-0.25,0) \braidup (0.25,0.5) -- (0.25,0.8) \toplabel{W};
            \draw[over] (-0.25,-0.8) \botlabel{Z^\sigma} -- (-0.25,-0.5) \braidup (0.25,0) \braidup (-0.25,0.5) -- (-0.25,0.8) \toplabel{Z};
            \draw[over] (0.25,-0.8) \botlabel{W^\sigma} -- (0.25,-0.5) \braidup (-0.25,0);
            \posdot{-0.25,-0.5};
            \posdot{-0.25,0};
        \end{tikzpicture}
        =
        \begin{tikzpicture}[centerzero]
            \draw (-0.25,0) \braidup (0.25,0.5) -- (0.25,0.8) \toplabel{W};
            \draw[over] (-0.25,-0.8) \botlabel{Z^\sigma} -- (-0.25,-0.5) \braidup (0.25,0) \braidup (-0.25,0.5) -- (-0.25,0.8) \toplabel{Z};
            \draw[over] (0.25,-0.8) \botlabel{W^\sigma} -- (0.25,-0.5) \braidup (-0.25,0);
            \posdot{-0.25,0.5};
            \posdot{-0.25,0};
        \end{tikzpicture}
        \ .
    \]
    This completes the induction, and so the twisted reflection equation holds for all objects of $\cC$.
    
    It remains to consider morphisms.  Using the twisted reflection equation, the argument of \cite[Lem.~2.4]{SSSsub26} shows that the set of morphisms $f$ satisfying the last equation in \cref{life} is closed under tensor product.  This set also contains the identity morphisms and is closed under composition and under inverses of morphisms.  In the $\kk$-linear setting, it is additionally closed under linear combinations.  Thus it suffices to impose this relation for $f$ in the set $S_1$ of generating morphisms.  Then the argument of \cite[Lem.~2.5]{SSSsub26} shows that the last equation in \cref{life} follows from the reflection equation when $f$ is a braiding morphism.  Thus, it suffices to impose this condition for $f$ in $S_1'$.
\end{proof}

\begin{cor}[Module-category negative-dot presentation]\label{swamp-}
    As a right $\cC$-module, the twisted affinization $\Aff_\sigma(\cC)$ is generated by invertible morphisms
    \[
        \begin{tikzpicture}[centerzero]
            \draw (0,-0.3) \botlabel{X} -- (0,0.3) \toplabel{X^\sigma};
            \negdot{0,0};
        \end{tikzpicture}
        \colon X \to X^\sigma,\qquad
        X \in S_0,
    \]
    subject to the relations
    \begin{equation}\label{lifeneg}
        \begin{tikzpicture}[centerzero]
            \draw (0.25,-0.8) \botlabel{Y} -- (0.25,-0.5) \braidup (-0.25,0);
            \draw[over] (-0.25,-0.8) \botlabel{X} -- (-0.25,-0.5) \braidup (0.25,0) \braidup (-0.25,0.5) -- (-0.25,0.8) \toplabel{X^\sigma};
            \draw[over] (-0.25,0) \braidup (0.25,0.5) -- (0.25,0.8) \toplabel{Y^\sigma};
            \negdot{-0.25,-0.5};
            \negdot{-0.25,0};
        \end{tikzpicture}
        =
        \begin{tikzpicture}[centerzero]
            \draw (0.25,-0.8) \botlabel{Y} -- (0.25,-0.5) \braidup (-0.25,0);
            \draw[over] (-0.25,-0.8) \botlabel{X} -- (-0.25,-0.5) \braidup (0.25,0) \braidup (-0.25,0.5) -- (-0.25,0.8) \toplabel{X^\sigma};
            \draw[over] (-0.25,0) \braidup (0.25,0.5) -- (0.25,0.8) \toplabel{Y^\sigma};
            \negdot{-0.25,0.5};
            \negdot{-0.25,0};
        \end{tikzpicture}
        \ ,\qquad
        \begin{tikzpicture}[centerzero]
            \draw (0,-0.8) \botlabel{W} -- (0,0.8) \toplabel{Z^\sigma};
            \negdot{0,-0.3};
            \coupon{0,0.3}{f^\sigma};
        \end{tikzpicture}
        =
        \begin{tikzpicture}[centerzero]
            \draw (0,-0.8) \botlabel{W} -- (0,0.8) \toplabel{Z^\sigma};
            \negdot{0,0.3};
            \coupon{0,-0.3}{f};
        \end{tikzpicture}
        \ ,
    \end{equation}
    for all $X,Y \in S_0$ and $f \colon W \to Z $ in $S_1'$.   The negative dots on arbitrary objects are defined recursively by the first equality in \cref{griffin} with $X \in S_0$, and the negative dot on $\one$ is $1_\one$.
\end{cor}

\begin{proof}
    This follows from inverting the generators and identities in \cref{swamp+}.
\end{proof}

\begin{rem}
    The presentations in \cref{swamp+,swamp-} can sometimes be simplified even further by omitting the dot generator on one object in a dual pair.  See \cref{AOSpresent,tAOSmodule} for an example of this phenomenon.
\end{rem}

\section{Braided module categories and cylinder twists} \label{S:braid}

In this section, we recall the definitions of braided module categories, cylinder twists, and tensor-pair cylinder twists.  We explain the connections between these structures and the different conventions used in the literature.  In \cref{S:affbraided} we will frame the concept of affinization in terms of these structures.

Throughout this section, we let $\cC$ be a braided strict monoidal category and let $\sigma \colon \cC \to \cC$ be a strict braided monoidal automorphism of $\cC$.

\subsection{Definitions}

\begin{defin}[Braided module category] \label{braidedmodcat}
    A strict $\cC$-module $\cM$ is \emph{$\sigma$-braided} if its module action $\otimes \colon \cM \times \cC \to \cM$ is equipped with a natural isomorphism $\eta \colon \otimes \Rightarrow \otimes \circ (\id_\cM \times \sigma)$ with components
    \begin{equation} \label{dubstep}
        \eta_{M,C}
        = 
        \begin{tikzpicture}[centerzero]
            \draw (0.25,-0.6) \botlabel{C} \braidup (-0.25,0);
            \draw[over] (-0.25,-0.6) \botlabel{M} \braidup (0.25,0) \braidup (-0.25,0.6) \toplabel{M};
            \draw[over] (-0.25,0) \braidup (0.25,0.6) \toplabel{C^\sigma};
        \end{tikzpicture}
        \colon M \otimes C \xrightarrow{\sim} M \otimes C^\sigma,\qquad
        M \in \Ob(\cM),\ C \in \Ob(\cC),
    \end{equation}
    called the \emph{$\sigma$-braid}, satisfying the \emph{$\sigma$-octagon axiom}
    \begin{equation} \label{octagon}
        \begin{tikzpicture}[centerzero]
            \draw (0.35,-0.8) \botlabel{D} \braidup (-0.35,0);
            \draw[over] (-0.35,-0.8) \botlabel{M \otimes C} \braidup (0.35,0) \braidup (-0.35,0.8) \toplabel{M \otimes C};
            \draw[over] (-0.35,0) \braidup (0.35,0.8) \toplabel{D^\sigma};
        \end{tikzpicture}
        =
        \begin{tikzpicture}[centerzero]
            \draw (0.5,-0.8) \botlabel{D} \braidup (-0.5,0);
            \draw[over] (0,-0.8) \botlabel{C} \braidup (0.5,0) \braidup (0,0.8) \toplabel{C};
            \draw[over] (-0.5,-0.8) \botlabel{M} \braidup (0,0) \braidup (-0.5,0.8) \toplabel{M};
            \draw[over] (-0.5,0) \braidup (0.5,0.8) \toplabel{D^\sigma};
        \end{tikzpicture}
        \ ,\qquad M \in \Ob(\cM),\quad C,D \in \Ob(\cC),
    \end{equation}
    and the \emph{$\sigma$-multiplicativity axiom}
    \begin{equation} \label{multaxiom}
        \begin{tikzpicture}[centerzero]
            \draw (0.35,-0.8) \botlabel{C \otimes D} \braidup (-0.35,0);
            \draw[over] (-0.35,-0.8) \botlabel{M} \braidup (0.35,0) \braidup (-0.35,0.8) \toplabelshift{-3pt}{M};
            \draw[over] (-0.35,0) \braidup (0.35,0.8) \toplabelshift{3pt}{(C \otimes D)^\sigma};
        \end{tikzpicture}
        =
        \begin{tikzpicture}[anchorbase]
            \draw (0,-0.8) \botlabel{C} \braidup (-0.5,-0.4);
            \draw[over] (0.5,-0.8) \botlabel{D} -- (0.5,0) \braidup (-0.5,0.6);
            \draw[over] (-0.5,-0.8) \botlabel{M} \braidup (0,-0.4) \braidup (-0.5,0);
            \draw[over] (-0.5,-0.4) \braidup (0,0) \braidup (0.5,0.6) \braidup (0,1.2) \toplabel{C^\sigma};
            \draw[over] (-0.5,0) \braidup (0,0.6) \braidup (-0.5,1.2) \toplabel{M};
            \draw[over] (-0.5,0.6) \braidup (0.5,1.2) \toplabel{D^\sigma};
        \end{tikzpicture}
        \left(
            =
            \begin{tikzpicture}[anchorbase]
                \draw (0.5,-0.8) \botlabel{D} \braidup (-0.5,-0.2);
                \draw[over] (0,-0.8) \botlabel{C} \braidup (0.5,-0.2) \braidup (0,0.4) \braidup (-0.5,0.8);
                \draw[over] (-0.5,-0.8) \botlabel{M} \braidup (0,-0.2) \braidup (-0.5,0.4);
                \draw[over] (-0.5,-0.2) \braidup (0.5,0.4) -- (0.5,1.2) \toplabel{D^\sigma};
                \draw[over] (-0.5,0.4) \braidup (0,0.8) \braidup (-0.5,1.2) \toplabel{M};
                \draw[over] (-0.5,0.8) \braidup (0,1.2) \toplabel{C^\sigma};
            \end{tikzpicture}
        \right)
        \ ,\qquad M \in \Ob(\cM),\quad C,D \in \Ob(\cC),
    \end{equation}
    where the equality in parentheses holds because of naturality of the $\sigma$-braid.
\end{defin}

\begin{rem}[Terminology and conventions]
    \begin{enumerate}
        \item The term \emph{octagon} in the name of relation \cref{octagon} refers to the eight arrows in the coherence diagram for a non-strict module category: four module associators (or their inverses), two braidings in $\cC$, and two components of $\eta$; see \cite[(1.1)]{DNTY19}.  In our strict setting, the associators are identities and are omitted, but the name is retained.

        \item The diagram for $\eta_{M,C}$ represents a single morphism.  A priori, it is \emph{not} a composition of two braid-like morphisms.  Indeed, such a composition would factor through $C \otimes M$ or $C^\sigma \otimes M$, which is not even defined in general, since $\cM$ is a \emph{right} $\cC$-module.

        \item When $\sigma = \id_\cC$, we omit $\sigma$ from the terminology; our definition then agrees with \cite[\S5.1]{Bro13} and \cite[Def.~2.1]{Kol20}.  Brochier's \emph{reflection relation} \cite[(5-4) and proof of Thm.~5.3]{Bro13} follows from the octagon axiom \cref{octagon} and naturality alone; the same argument gives its $\sigma$-twisted analogue.  The \emph{$\sigma$-multiplicativity axiom} \cref{multaxiom} specifies how $\eta_{M,C\otimes D}$ is determined by the components for $C$ and $D$.  For $\sigma = \id_\cC$, it is equivalent, using the octagon axiom, to Brochier's second defining condition \cite[(5-3)]{Bro13}.

        \item Writing $\tau$ for the autoequivalence in \cite[Def.~1.4 and Rem.~1.5(1)]{DNTY19}, their \emph{$\tau$-braid} is a natural isomorphism $M \otimes \tau(C) \to M \otimes C$ satisfying the $\tau$-octagon equation; it is a \emph{ribbon $\tau$-braid} if it also satisfies their \emph{ribbon $\tau$-twist equation}.  Our $\eta$ corresponds to a ribbon $\tau$-braid with $\tau = \sigma^{-1}$ and components $\widetilde\eta_{M,C} := \eta_{M,\sigma^{-1}(C)}$.  Under this identification, their ribbon twist equation is equivalent to our multiplicativity axiom \cref{multaxiom} using \cref{octagon}.

        \item For $\sigma = \id_\cC$, the definition of a \emph{braided module category} in \cite[Def.~3.4 and Rem.~3.6]{BBJ18} uses the annulus framing instead of the cylinder framing underlying our axioms.  If $\cC$ has a balancing $\theta$, their module braiding $E$ corresponds to ours by $\eta_{M,C} = E_{M,C}^{-1} \circ (1_M \otimes \theta_C^{-1})$.  Laugwitz--Walton--Yakimov use left module categories and a different multiplicativity axiom.  When $\cC$ is ribbon, their module braiding can be converted to the left-module version of our convention by first applying the ribbon twist to the $\cC$-factor and then applying their module braiding; see \cite[Rem.~3.15(a)]{LWY25}.
    \end{enumerate}
\end{rem}

\begin{defin}[Cylinder $\sigma$-twist] \label{cylindertwist}
    A \emph{cylinder $\sigma$-twist} on $\cC$ is a natural isomorphism $\theta \colon \id_\cC \Rightarrow \sigma$ with components
    \begin{equation}
        \theta_C
        =
        \begin{tikzpicture}[centerzero]
            \draw (0,-0.4) \botlabel{C} -- (0,0.4) \toplabel{C^\sigma};
            \coupon{0,0}{\theta};
        \end{tikzpicture}
        \colon C \to C^\sigma,\qquad C \in \Ob(\cC),
    \end{equation}
    satisfying the first equality in
    \begin{equation} \label{twist}
        \begin{tikzpicture}[centerzero,scale=1.5]
            \draw (0,-0.6) \botlabel{C \otimes D} -- (0,0.6) \toplabel{(C \otimes D)^\sigma};
            \coupon{0,0}{\theta};
        \end{tikzpicture}
        =
        \begin{tikzpicture}[centerzero,scale=1.5]
            \draw (0.2,-0.6) \botlabel{D} -- (0.2,-0.4) \braidup (-0.2,0);
            \draw[over] (-0.2,-0.6) \botlabel{C} -- (-0.2,-0.4) \braidup (0.2,0);
            \draw (0.2,0) \braidup (-0.2,0.4) -- (-0.2,0.6) \toplabel{C^\sigma};
            \draw[over] (-0.2,0) \braidup (0.2,0.4) -- (0.2,0.6) \toplabel{D^\sigma};
            \coupon{-0.2,-0.4}{\theta};
            \coupon{-0.2,0}{\theta};
        \end{tikzpicture}
        \left(
            =
            \begin{tikzpicture}[centerzero,scale=1.5]
                \draw (0.2,-0.6) \botlabel{D} -- (0.2,-0.4) \braidup (-0.2,0);
                \draw[over] (-0.2,-0.6) \botlabel{C} -- (-0.2,-0.4) \braidup (0.2,0);
                \draw (0.2,0) \braidup (-0.2,0.4) -- (-0.2,0.6) \toplabel{C^\sigma};
                \draw[over] (-0.2,0) \braidup (0.2,0.4) -- (0.2,0.6) \toplabel{D^\sigma};
                \coupon{-0.2,0.4}{\theta};
                \coupon{-0.2,0}{\theta};
            \end{tikzpicture}
        \right)
        \qquad \text{for all } C,D \in \Ob(\cC).
    \end{equation}
    Here naturality means that, for morphisms $f \colon C \to D$ in $\cC$,
    \begin{equation}\label{ferns}
        \begin{tikzpicture}[centerzero]
            \draw (0,-0.8) \botlabel{C} -- (0,0.8) \toplabel{D^\sigma};
            \coupon{0,-0.3}{f};
            \coupon{0,0.3}{\theta};
        \end{tikzpicture}
        =
        \begin{tikzpicture}[centerzero]
            \draw (0,-0.8) \botlabel{C} -- (0,0.8) \toplabel{D^\sigma};
            \coupon{0,-0.3}{\theta};
            \coupon{0,0.3}{f^\sigma};
        \end{tikzpicture}
        \ .
    \end{equation}
    The equality in parentheses in \cref{twist} follows by applying \cref{ferns} to the braiding and using the first equality with $C$ and $D$ interchanged.
\end{defin}

\begin{rem} \label{cylindertwistcorrespondence}
    A cylinder $\sigma$-twist on $\cC$ is equivalent to a $\sigma$-braided structure on $\cC$, regarded as a right $\cC$-module via its tensor product.  Given such a $\sigma$-braid $\eta$, define
    \[
        \theta_C := \eta_{\one,C},\qquad C \in \Ob(\cC).
    \]
    Then the $\sigma$-octagon condition \cref{octagon} with $M=\one$ determines $\eta$ from $\theta$:
    \begin{equation} \label{thunder}
        \eta_{C,D}
        =
        \begin{tikzpicture}[centerzero]
            \draw (0.2,-0.6) \botlabel{D} \braidup (-0.2,0);
            \draw[over] (-0.2,-0.6) \botlabel{C} \braidup (0.2,0) \braidup (-0.2,0.6) \toplabel{C};
            \draw[over] (-0.2,0) \braidup (0.2,0.6) \toplabel{D^\sigma};
            \coupon{-0.2,0}{\theta};
        \end{tikzpicture}
        .
    \end{equation}
    The multiplicativity axiom \cref{multaxiom} with $M = \one$ becomes \cref{twist}.  Conversely, given a cylinder $\sigma$-twist $\theta$ on $\cC$, \cref{thunder} defines a natural isomorphism in both variables; \cref{hexagon} and the cylinder $\sigma$-twist equation imply the octagon axiom \cref{octagon} and the multiplicativity axiom \cref{multaxiom}.  This is the regular right module case of \cite[Prop.~D.2]{DNTY19}, after translating their conventions to ours.  The two constructions are mutually inverse.  When $\sigma = \id_\cC$, this is a \emph{twist} on $\cC$, endowing $\cC$ with the structure of a balanced monoidal category.
\end{rem}

\subsection{Tensor-pair cylinder twists}

We now discuss the concept of a tensor-pair cylinder twist, which appears in the literature on quantum symmetric pairs.  Unlike the situation of \cref{cylindertwistcorrespondence}, tensor-pair cylinder twists are \emph{not} special cases of $\sigma$-braided structures; see \cref{bottle}.

Throughout this subsection, we assume that $\cM$ is a strict $\cC$-module satisfying the following conditions:
\begin{itemize}
    \item $\cC$ is a subcategory of $\cM$ with $\Ob(\cC) = \Ob(\cM)$,
    \item when restricted to $\cC \times \cC$, the module category action is given by the tensor product of $\cC$.
\end{itemize}
In this case, one says that $(\cM,\cC)$ is a \emph{strict tensor pair}; see \cite[\S4.1]{BK19}.  Let $i \colon \cC \hookrightarrow \cM$ denote the inclusion.  

\begin{defin}[Tensor-pair cylinder twist]\label{TPcyltwist}
    A \emph{tensor-pair cylinder $\sigma$-twist for $(\cM,\cC)$} is a natural isomorphism $\tau \colon i \Rightarrow i \circ \sigma$ satisfying the first equality in
    \begin{equation} \label{taudub}
        \begin{tikzpicture}[centerzero,scale=1.5]
            \draw (0,-0.6) \botlabel{X \otimes Y} -- (0,0.6) \toplabel{(X \otimes Y)^\sigma};
            \coupon{0,0}{\tau};
        \end{tikzpicture}
        =
        \begin{tikzpicture}[centerzero,scale=1.5]
            \draw (0.2,-0.6) \botlabel{Y} -- (0.2,-0.4) \braidup (-0.2,0);
            \draw[over] (-0.2,-0.6) \botlabel{X} -- (-0.2,-0.4) \braidup (0.2,0);
            \draw (0.2,0) \braidup (-0.2,0.4) -- (-0.2,0.6) \toplabel{X^\sigma};
            \draw[over] (-0.2,0) \braidup (0.2,0.4) -- (0.2,0.6) \toplabel{Y^\sigma};
            \coupon{-0.2,-0.4}{\tau};
            \coupon{-0.2,0}{\tau};
        \end{tikzpicture}
        \left(
            =
            \begin{tikzpicture}[centerzero,scale=1.5]
                \draw (0.2,-0.6) \botlabel{Y} -- (0.2,-0.4) \braidup (-0.2,0);
                \draw[over] (-0.2,-0.6) \botlabel{X} -- (-0.2,-0.4) \braidup (0.2,0) \braidup (-0.2,0.4) -- (-0.2,0.6) \toplabel{X^\sigma};
                \draw[over] (-0.2,0) \braidup (0.2,0.4) -- (0.2,0.6) \toplabel{Y^\sigma};
                \coupon{-0.2,0.4}{\tau};
                \coupon{-0.2,0}{\tau};
            \end{tikzpicture}
        \right)
        \qquad X,Y \in \Ob(\cC).
    \end{equation}
    Here naturality requires only
    \begin{equation} \label{taubasenaturality}
        \begin{tikzpicture}[centerzero]
            \draw (0,-0.8) \botlabel{X} -- (0,0.8) \toplabel{Y^\sigma};
            \coupon{0,-0.3}{f};
            \coupon{0,0.3}{\tau};
        \end{tikzpicture}
        =
        \begin{tikzpicture}[centerzero]
            \draw (0,-0.8) \botlabel{X} -- (0,0.8) \toplabel{Y^\sigma};
            \coupon{0,-0.3}{\tau};
            \coupon{0,0.3}{f^\sigma};
        \end{tikzpicture}
        \ ,\qquad f \in \cC(X,Y),
    \end{equation}
    and does not impose a condition on the additional morphisms of $\cM$.  The equality in parentheses in \cref{taudub} follows by applying \cref{taubasenaturality} to the braiding and using the first equality with $X$ and $Y$ interchanged.
\end{defin}

\begin{rem}
    \begin{enumerate}
        \item \Cref{TPcyltwist} agrees with \cite[Def.~4.6 and (4.6)]{BK19}, where the equality between $\tau_{X \otimes Y}$ and the parenthesized expression in \cref{taudub} is used as the tensor-product compatibility axiom.  When $\sigma^2=\id_\cC$, it also agrees with \cite[Def.~2.1]{SSSsub26}, with $\natural=\sigma$.  Taking $X=Y=\one$ shows that $\tau_\one=1_\one$.
    
        \item For $\sigma=\id_\cC$, a strict tensor pair $(\cM,\cC)$ equipped with a tensor-pair cylinder twist $\tau$ is the strict form of the cylinder-braiding structure described in \cite[\S4]{tDH98}, with tensor-pair cylinder twist identities as in \cite[Prop.~1.3]{tDH98}.  Their ``weak endomorphism of the identity functor'' is, in our notation, a natural automorphism of the inclusion $i\colon\cC\hookrightarrow\cM$.
    \end{enumerate}
\end{rem}

We now relate the concept of a tensor-pair cylinder $\sigma$-twist to that of a $\sigma$-braid.  For a tensor-pair cylinder $\sigma$-twist $\tau$, define
\begin{equation} \label{taureconstruct}
    E^\tau_{X,Y}
    :=
    \begin{tikzpicture}[centerzero]
        \draw (0.3,-0.6) \botlabel{Y} \braidup (-0.3,0);
        \draw[over] (-0.3,-0.6) \botlabel{X} \braidup (0.3,0) \braidup (-0.3,0.6) \toplabel{X};
        \draw[over] (-0.3,0) \braidup (0.3,0.6) \toplabel{Y^\sigma};
        \coupon{-0.3,0}{\tau};
    \end{tikzpicture}
    \colon X \otimes Y \xrightarrow{\sim} X \otimes Y^\sigma,\qquad X,Y \in \Ob(\cC).
\end{equation}
The crossings in \cref{taureconstruct} are braidings in $\cC$, and the coupon denotes $\tau_Y$ tensored on the right with $1_X$; thus this is a well-defined composite in $\cM$.

\begin{prop} \label{tensorbraidcorrespondence}
    The assignment $\eta \mapsto (\eta_{\one,X})_{X \in \Ob(\cC)}$ is a bijection between $\sigma$-braided structures on the right $\cC$-module $\cM$ and tensor-pair cylinder $\sigma$-twists $\tau$ for $(\cM,\cC)$ satisfying
    \begin{equation} \label{taunaturality}
        \begin{tikzpicture}[centerzero]
            \draw (0.3,-1) \botlabel{Z} -- (0.3,-0.6) \braidup (-0.3,0);
            \draw[over] (-0.3,-1) \botlabel{X} -- (-0.3,-0.6) \braidup (0.3,0) \braidup (-0.3,0.6) -- (-0.3,1) \toplabel{Y};
            \draw[over] (-0.3,0) \braidup (0.3,0.6) -- (0.3,1) \toplabel{Z^\sigma};
            \coupon{-0.3,0}{\tau};
            \coupon{-0.3,-0.7}{a};
        \end{tikzpicture}
        =
        \begin{tikzpicture}[centerzero]
            \draw (0.3,-1) \botlabel{Z} -- (0.3,-0.6) \braidup (-0.3,0);
            \draw[over] (-0.3,-1) \botlabel{X} -- (-0.3,-0.6) \braidup (0.3,0) \braidup (-0.3,0.6) -- (-0.3,1) \toplabel{Y};
            \draw[over] (-0.3,0) \braidup (0.3,0.6) -- (0.3,1) \toplabel{Z^\sigma};
            \coupon{-0.3,0}{\tau};
            \coupon{-0.3,0.7}{a};
        \end{tikzpicture}
        \ ,\qquad a \in \cM(X,Y),\quad Z \in \Ob(\cC).
    \end{equation}
    The inverse sends $\tau$ to $E^\tau$.
\end{prop}

\begin{proof}
    For a $\sigma$-braid $\eta$, the $\sigma$-octagon axiom \cref{octagon} at $M=\one$ gives $\eta_{X,Y}=E^\tau_{X,Y}$, where $\tau_X=\eta_{\one,X}$.  The preceding discussion shows that $\tau$ is a tensor-pair cylinder $\sigma$-twist, and naturality of $\eta$ in its first argument gives \cref{taunaturality}.

    Conversely, let $\tau$ be a tensor-pair cylinder $\sigma$-twist.  Substituting \cref{taureconstruct} in \cref{octagon} and applying the identities \cref{hexagon} proves the octagon axiom \cref{octagon} for $E^\tau$.  Applying the first equality in \cref{taudub} to $\tau_{X \otimes Y \otimes Z}$ using the two parenthesizations of the triple tensor product $X \otimes Y \otimes Z$, and cancelling $\tau_X$ on the input strand, gives the equality below after relabeling $X^\sigma$ as $X$:
    \begin{equation} \label{taumultiplicativity}
        \begin{tikzpicture}[anchorbase]
            \draw (0.35,-0.8) \botlabel{Y \otimes Z} \braidup (-0.35,0);
            \draw[over] (-0.35,-0.8) \botlabel{X} \braidup (0.35,0) \braidup (-0.35,0.8) \toplabelshift{-3pt}{X};
            \draw[over] (-0.35,0) \braidup (0.35,0.8) \toplabelshift{3pt}{(Y \otimes Z)^\sigma};
            \coupon{-0.35,0}{\tau};
        \end{tikzpicture}
        =
        \begin{tikzpicture}[anchorbase]
            \draw (0,-0.8) \botlabel{Y} \braidup (-0.5,-0.4);
            \draw[over] (0.5,-0.8) \botlabel{Z} -- (0.5,0) \braidup (-0.5,0.6);
            \draw[over] (-0.5,-0.8) \botlabel{X} \braidup (0,-0.4) \braidup (-0.5,0);
            \draw[over] (-0.5,-0.4) \braidup (0,0) \braidup (0.5,0.6) \braidup (0,1.2) \toplabel{Y^\sigma};
            \draw[over] (-0.5,0) \braidup (0,0.6) \braidup (-0.5,1.2) \toplabel{X};
            \draw[over] (-0.5,0.6) \braidup (0.5,1.2) \toplabel{Z^\sigma};
            \coupon{-0.5,-0.4}{\tau};
            \coupon{-0.5,0.6}{\tau};
        \end{tikzpicture}
        .
    \end{equation}
    This is precisely the multiplicativity axiom \cref{multaxiom} after substituting \cref{taureconstruct}.  Naturality of $\tau$ and the braiding gives naturality of $E^\tau$ in its second argument, and in its first argument for morphisms in $\cC$.  Condition \cref{taunaturality} supplies naturality for all morphisms in $\cM$.  Hence $E^\tau$ is a $\sigma$-braid.  Finally, $E^\tau_{\one,X}=\tau_X$, so the assignments are mutually inverse.
\end{proof}

The octagon and multiplicativity axioms for $E^\tau$, established in the proof of \cref{tensorbraidcorrespondence}, also follow by specializing \cite[Prop.~D.2]{DNTY19}.  More precisely, take their module category to be $\cM$, their distinguished object to be $\one$, their twisting autoequivalence to be $\sigma^{-1}$, and their tensor-pair cylinder twist to have components $\theta_V=\tau_{\sigma^{-1}(V)}$.  Under this identification, their ribbon twist equation is equivalent to our multiplicativity axiom \cref{multaxiom}, using the octagon axiom \cref{octagon}.  As noted immediately after that proposition, naturality of $E^\tau$ in its first argument with respect to all morphisms of $\cM$ must be imposed in addition to the tensor-pair cylinder $\sigma$-twist axioms; in our setting, this is condition \cref{taunaturality}.  The following restriction of morphisms is the corresponding case of the construction in \cite[App.~D, pp.~765--767]{DNTY19}.

In \cite{SSSsub26}, the authors developed a theory of classifying submodules over monoidal categories.  One of the key assumptions on tensor pairs there was \cite[Ass.~3.1(c)]{SSSsub26}, which states that $\cM$ is generated as a $\cC$-module by the components of the tensor-pair cylinder $\sigma$-twist.  The following corollary explains the significance of this assumption.

\begin{cor} \label{taugenerated}
    Let $\tau$ be a tensor-pair cylinder $\sigma$-twist for $(\cM,\cC)$.  Let $\cM^\tau$ have the same objects as $\cM$, with $\cM^\tau(X,Y)$ consisting of the morphisms $a \in \cM(X,Y)$ satisfying \cref{taunaturality} for every $Z \in \Ob(\cC)$.  Then $\cM^\tau$ is a subcategory of $\cM$ containing $\cC$, closed under the right $\cC$-action, and equipped with the $\sigma$-braid $E^\tau$.

    In particular, if $\cM$ is generated under composition and the right $\cC$-action by morphisms of $\cC$ and the components $\tau_X$ and $\tau_X^{-1}$, then $\tau$ determines a unique $\sigma$-braided structure on $\cM$.  In the $\kk$-linear setting, $\cM^\tau$ is a $\kk$-linear subcategory, and linear combinations may also be used in the generation hypothesis.
\end{cor}

\begin{proof}
    Condition \cref{taunaturality} is preserved under composition and inversion of invertible morphisms.  In the $\kk$-linear setting, it is also preserved under linear combinations.  It holds for morphisms of $\cC$ by naturality of the braiding, and the octagon identity \cref{octagon} shows that it is preserved by tensoring on the right with an object of $\cC$.  Thus $\cM^\tau$ is a subcategory closed under the right action.  The equality of the two right-hand expressions in \cref{taudub} is exactly \cref{taunaturality} for $a=\tau_X$.  Consequently, $\cM^\tau$ contains all twist components and their inverses, and hence also the composites \cref{taureconstruct} and their inverses.  \Cref{tensorbraidcorrespondence} now gives the braided structure on $\cM^\tau$, and the generation hypothesis implies $\cM^\tau=\cM$.
\end{proof}

\begin{rem} \label{tauconjugation}
    Condition \cref{taunaturality} also has a formulation in terms of the right action.  Conjugation by $\tau$ defines a category automorphism $\widehat\sigma \colon \cM \to \cM$ extending $\sigma$, with $\widehat\sigma(X)=X^\sigma$ and
    \[
        \widehat\sigma(a)
        :=
        \begin{tikzpicture}[centerzero]
            \draw (0,-1) \botlabel{X^\sigma} -- (0,1) \toplabel{Y^\sigma};
            \coupon{0,-0.55}{\tau^{-1}};
            \coupon{0,0}{a};
            \coupon{0,0.55}{\tau};
        \end{tikzpicture}
        \ ,\qquad a \in \cM(X,Y).
    \]
    Condition \cref{taunaturality} is equivalent to compatibility of $\widehat\sigma$ with the right action, expressed diagrammatically as
    \begin{equation} \label{tauaction}
        \begin{tikzpicture}[centerzero]
            \draw (-0.3,-1) \botlabel{X^\sigma} -- (-0.3,1) \toplabel{Y^\sigma};
            \draw (0.3,-1) \botlabel{Z^\sigma} -- (0.3,1) \toplabel{W^\sigma};
            \genbox{-0.45,-0.7}{0.45,-0.4}{\tau^{-1}};
            \coupon{-0.3,0}{a};
            \coupon{0.3,0}{f};
            \genbox{-0.45,0.4}{0.45,0.7}{\tau};
        \end{tikzpicture}
        =
        \begin{tikzpicture}[centerzero]
            \draw (-0.3,-1) \botlabel{X^\sigma} -- (-0.3,1) \toplabel{Y^\sigma};
            \draw (0.3,-1) \botlabel{Z^\sigma} -- (0.3,1) \toplabel{W^\sigma};
            \coupon{-0.3,-0.55}{\tau^{-1}};
            \coupon{-0.3,0}{a};
            \coupon{0.3,0}{f^\sigma};
            \coupon{-0.3,0.55}{\tau};
        \end{tikzpicture}
        \ ,\qquad a \in \cM(X,Y),\quad f \in \cC(Z,W).
    \end{equation}
    Indeed, expanding the coupons straddling two strands using the last expression in \cref{taudub} and cancelling the outer twist components gives \cref{taunaturality} when $f=1_Z$; the general case then follows from naturality with respect to $f$.
\end{rem}

\begin{eg} \label{taunonnatural}
    A tensor-pair cylinder $\sigma$-twist need not satisfy \cref{taunaturality}, even when $\sigma=\id_\cC$.  Work over $\mathbb{Q}$.  Let $\cC$ have objects $\mathbb{Z}$, tensor product given by addition, and
    \[
        \cC(n,m)=
        \begin{cases}
            \mathbb{Q}, & n=m,\\
            0, & n\ne m.
        \end{cases}
    \]
    Composition and tensor product of morphisms are multiplication of scalars.  Give $\cC$ the braiding $\beta_{n,m}=2^{nm}1_{n+m}$.  Let $\cM$ have the same objects, with $\cM(n,m)=\mathbb{Q}$ for all $n,m$, scalar composition, and right action given by addition on objects and multiplication on morphisms.  The inclusion is the identity on endomorphism spaces, so $(\cM,\cC)$ is a strict tensor pair.  Its distinguished unit object is the integer $0$.

    The family $\tau_n=2^{n^2}1_n$ is natural for morphisms of $\cC$ and satisfies the tensor-pair cylinder twist equation \cref{taudub}, since
    \[
        2^{(n+m)^2}=2^{n^2}\,2^{nm}\,2^{m^2}\,2^{nm}.
    \]
    However, \cref{taureconstruct} gives $E^\tau_{n,m}=2^{m^2+2nm}1_{n+m}$.  Let $a \colon 0 \to 1$ be the morphism represented by the scalar $1$.  In $\cM(1,2)=\mathbb{Q}$, the two sides of \cref{taunaturality} with $Z=1$ evaluate to
    \[
        \begin{tikzpicture}[centerzero]
            \draw (0.3,-1.2) \botlabel{1} -- (0.3,-0.6) \braidup (-0.3,0);
            \draw[over] (-0.3,-1.2) \botlabel{0} -- (-0.3,-0.6) \braidup (0.3,0) \braidup (-0.3,0.6) -- (-0.3,1.2) \toplabel{1};
            \draw[over] (-0.3,0) \braidup (0.3,0.6) -- (0.3,1.2) \toplabel{1};
            \coupon{-0.3,0}{\tau};
            \coupon{-0.3,-0.9}{a};
        \end{tikzpicture}
        =8\ne2=
        \begin{tikzpicture}[centerzero]
            \draw (0.3,-1.2) \botlabel{1} -- (0.3,-0.6) \braidup (-0.3,0);
            \draw[over] (-0.3,-1.2) \botlabel{0} -- (-0.3,-0.6) \braidup (0.3,0) \braidup (-0.3,0.6) -- (-0.3,1.2) \toplabel{1};
            \draw[over] (-0.3,0) \braidup (0.3,0.6) -- (0.3,1.2) \toplabel{1};
            \coupon{-0.3,0}{\tau};
            \coupon{-0.3,0.9}{a};
        \end{tikzpicture}
        .
    \]

    Thus this tensor-pair cylinder twist does not arise from a $\sigma$-braided structure on $\cM$.
\end{eg}

\begin{rem}\label{bottle}
    When $\cM=\cC$, condition \cref{taunaturality} is automatic, so \cref{tensorbraidcorrespondence} specializes to the correspondence described in \cref{cylindertwistcorrespondence} between cylinder $\sigma$-twists on $\cC$ and $\sigma$-braided structures on its regular right module.  For a general tensor pair, \cref{tensorbraidcorrespondence,taunonnatural} show that $\sigma$-braided structures correspond to tensor-pair cylinder $\sigma$-twists satisfying an additional naturality condition.  In particular, arbitrary tensor-pair cylinder twists are \emph{not} a special case of \cref{braidedmodcat}.
\end{rem}

\section{Affinizations as braided module categories\label{S:affbraided}}

We now connect the notion of affinization to that of braided module categories and tensor-pair cylinder twists.  Throughout this section, we continue to assume that $\cC$ is a braided strict monoidal category and $\sigma\colon\cC\to\cC$ is a strict braided monoidal automorphism.  Recall from \cref{S:affmodule} that the right $\cC$-action on $\Aff_\sigma(\cC)$ is the restriction of its tensor product along $\imath_\cC$.

\begin{prop} \label{affbraiding}
    The twisted affinization $\Aff_\sigma(\cC)$ has a canonical $\sigma$-braided right $\cC$-module structure with components
    \begin{equation} \label{affbraid}
        \eta^{\Aff}_{X,Y}
        :=
        \begin{tikzpicture}[centerzero]
            \draw (-0.3,-0.6) \botlabel{X} -- (-0.3,0.6) \toplabel{X};
            \draw (0.3,-0.6) \botlabel{Y} -- (0.3,0.6) \toplabel{Y^\sigma};
            \negdot{0.3,0};
        \end{tikzpicture}
        =
        \begin{tikzpicture}[centerzero]
            \draw (0.3,-0.6) \botlabel{Y} \braidup (-0.3,0);
            \draw[over] (-0.3,-0.6) \botlabel{X} \braidup (0.3,0) \braidup (-0.3,0.6) \toplabel{X};
            \draw[over] (-0.3,0) \braidup (0.3,0.6) \toplabel{Y^\sigma};
            \negdot{-0.3,0};
        \end{tikzpicture}
        \ ,\qquad X,Y \in \Ob(\cC).
    \end{equation}
    The first diagram denotes $1_X\otimes\xi_Y^{-1}\colon X\otimes Y\to X\otimes Y^\sigma$ in the monoidal category $\Aff_\sigma(\cC)$.  In particular, $\eta^{\Aff}_{\one,X}=\xi_X^{-1}$.
\end{prop}

\begin{proof}
    The first relation in \cref{jaguar-} gives the equality in \cref{affbraid}.  For $M\in\Ob(\Aff_\sigma(\cC))$ and $C,D\in\Ob(\cC)$, the $\sigma$-octagon axiom \cref{octagon} follows from
    \[
        \begin{tikzpicture}[centerzero]
            \draw (0.35,-0.8) \botlabel{D} \braidup (-0.35,0);
            \draw[over] (-0.35,-0.8) \botlabel{M\otimes C} \braidup (0.35,0) \braidup (-0.35,0.8) \toplabel{M\otimes C};
            \draw[over] (-0.35,0) \braidup (0.35,0.8) \toplabel{D^\sigma};
        \end{tikzpicture}
        \ceqref{affbraid}
        \begin{tikzpicture}[centerzero]
            \draw (-0.6,-0.8) \botlabel{M} -- (-0.6,0.8) \toplabel{M};
            \draw (0,-0.8) \botlabel{C} -- (0,0.8) \toplabel{C};
            \draw (0.6,-0.8) \botlabel{D} -- (0.6,0.8) \toplabel{D^\sigma};
            \negdot{0.6,0};
        \end{tikzpicture}
        \ceqref{affbraid}
        \begin{tikzpicture}[centerzero]
            \draw (-0.6,-0.8) \botlabel{M} -- (-0.6,0.8) \toplabel{M};
            \draw (0.6,-0.8) \botlabel{D} \braidup (0,0);
            \draw[over] (0,-0.8) \botlabel{C} \braidup (0.6,0) \braidup (0,0.8) \toplabel{C};
            \draw[over] (0,0) \braidup (0.6,0.8) \toplabel{D^\sigma};
            \negdot{0,0};
        \end{tikzpicture}
        \ceqref{affbraid}
        \begin{tikzpicture}[centerzero]
            \draw (0.6,-0.8) \botlabel{D} \braidup (-0.6,0);
            \draw[over] (0,-0.8) \botlabel{C} \braidup (0.6,0) \braidup (0,0.8) \toplabel{C};
            \draw[over] (-0.6,-0.8) \botlabel{M} \braidup (0,0) \braidup (-0.6,0.8) \toplabel{M};
            \draw[over] (-0.6,0) \braidup (0.6,0.8) \toplabel{D^\sigma};
        \end{tikzpicture}
        \ .
    \]

    The $\sigma$-multiplicativity axiom \cref{multaxiom} follows from
    \[
        \begin{tikzpicture}[centerzero]
            \draw (0.35,-1) \botlabel{C\otimes D} \braidup (-0.35,0);
            \draw[over] (-0.35,-1) \botlabel{M} \braidup (0.35,0) \braidup (-0.35,1) \toplabelshift{-3pt}{M};
            \draw[over] (-0.35,0) \braidup (0.35,1) \toplabelshift{3pt}{(C\otimes D)^\sigma};
        \end{tikzpicture}
        \ceqref{affbraid}
        \begin{tikzpicture}[centerzero]
            \draw (-0.35,-1) \botlabel{M} -- (-0.35,1) \toplabelshift{-3pt}{M};
            \draw (0.35,-1) \botlabel{C\otimes D} -- (0.35,1) \toplabelshift{3pt}{(C\otimes D)^\sigma};
            \negdot{0.35,0};
        \end{tikzpicture}
        \ceqref{jaguar-}
        \begin{tikzpicture}[centerzero]
            \draw (-0.6,-1) \botlabel{M} -- (-0.6,1) \toplabel{M};
            \draw (0,-1) \botlabel{C} -- (0,1) \toplabel{C^\sigma};
            \draw (0.6,-1) \botlabel{D} -- (0.6,1) \toplabel{D^\sigma};
            \negdot{0,-0.625};
            \negdot{0.6,0.375};
        \end{tikzpicture}
        \ceqref{affbraid}
        \begin{tikzpicture}[centerzero]
            \draw (0,-1) \botlabel{C} \braidup (-0.6,-0.625);
            \draw[over] (-0.6,-1) \botlabel{M} \braidup (0,-0.625) \braidup (-0.6,-0.25) -- (-0.6,1) \toplabel{M};
            \draw[over] (-0.6,-0.625) \braidup (0,-0.25) -- (0,1) \toplabel{C^\sigma};
            \draw (0.6,-1) \botlabel{D} -- (0.6,1) \toplabel{D^\sigma};
            \negdot{0.6,0.375};
        \end{tikzpicture}
        \ceqref{affbraid}
        \begin{tikzpicture}[centerzero]
            \draw (0,-1) \botlabel{C} \braidup (-0.6,-0.625);
            \draw[over] (0.6,-1) \botlabel{D} -- (0.6,-0.25) \braidup (-0.6,0.375);
            \draw[over] (-0.6,-1) \botlabel{M} \braidup (0,-0.625) \braidup (-0.6,-0.25);
            \draw[over] (-0.6,-0.625) \braidup (0,-0.25) \braidup (0.6,0.375) \braidup (0,1) \toplabel{C^\sigma};
            \draw[over] (-0.6,-0.25) \braidup (0,0.375) \braidup (-0.6,1) \toplabel{M};
            \draw[over] (-0.6,0.375) \braidup (0.6,1) \toplabel{D^\sigma};
        \end{tikzpicture}
        \ .
    \]

    Naturality in $X$ holds for all morphisms of $\Aff_\sigma(\cC)$ by functoriality of the tensor product, and naturality in $Y$ follows from the fourth relation in \cref{jaguar-}.
\end{proof}

\begin{theo}[Extension of module actions] \label{affactionextension}
    Let $\cM$ be a strict right $\cC$-module category.  There is a bijection between $\sigma$-braided module structures on $\cM$ and strict right $\Aff_\sigma(\cC)$-actions on $\cM$ extending the given $\cC$-action.  For a $\sigma$-braid $\eta$, the corresponding extended action $\odot$ is uniquely determined by
    \begin{equation} \label{affdotaction}
        1_M\odot
        \begin{tikzpicture}[centerzero]
            \draw (0,-0.6) \botlabel{X} -- (0,0.6) \toplabel{X^\sigma};
            \negdot{0,0};
        \end{tikzpicture}
        =
        \begin{tikzpicture}[centerzero]
            \draw (0.3,-0.6) \botlabel{X} \braidup (-0.3,0);
            \draw[over] (-0.3,-0.6) \botlabel{M} \braidup (0.3,0) \braidup (-0.3,0.6) \toplabel{M};
            \draw[over] (-0.3,0) \braidup (0.3,0.6) \toplabel{X^\sigma};
        \end{tikzpicture}
        =\eta_{M,X},\qquad M\in\Ob(\cM),\quad X\in\Ob(\cC).
    \end{equation}
\end{theo}

\begin{proof}
    We use the all-object version of the monoidal dot presentation discussed in \cref{monpresrem}\cref{monpresrem1}.  Inverting its positive-dot generators gives a presentation of $\Aff_\sigma(\cC)$ by the invertible morphisms $\xi_X^{-1}\colon X\to X^\sigma$, for all $X\in\Ob(\cC)$, with the first, third, and fourth relations in \cref{jaguar-}.  This presentation does not require the objects of $\cC$ to be freely generated.

    Write $\bA \colon \cC\to\cEnd(\cM)^\rev$ for the given action functor.  Suppose first that $\odot$ is a strict $\Aff_\sigma(\cC)$-action extending the given $\cC$-action, and define $\eta_{M,X}$ by \cref{affdotaction}.  These morphisms are invertible, and functoriality of $\odot$ gives naturality in $M$.  The fourth relation in \cref{jaguar-} gives naturality in $X$.

    The following calculation shows that the first relation in \cref{jaguar-}, after applying the action, is equivalent to the $\sigma$-octagon axiom \cref{octagon} with $C=Y$ and $D=X$:
    \[
        \begin{tikzpicture}[centerzero]
            \draw (0.35,-0.8) \botlabel{X} \braidup (-0.35,0);
            \draw[over] (-0.35,-0.8) \botlabel{M\otimes Y} \braidup (0.35,0) \braidup (-0.35,0.8) \toplabel{M\otimes Y};
            \draw[over] (-0.35,0) \braidup (0.35,0.8) \toplabel{X^\sigma};
        \end{tikzpicture}
        \ceqref{affdotaction}
        1_M\odot
        \begin{tikzpicture}[centerzero]
            \draw (-0.3,-0.8) \botlabel{Y} -- (-0.3,0.8) \toplabel{Y};
            \draw (0.3,-0.8) \botlabel{X} -- (0.3,0.8) \toplabel{X^\sigma};
            \negdot{0.3,0};
        \end{tikzpicture}
        \ceqref{Reid2}
        1_M\odot
        \begin{tikzpicture}[centerzero]
            \draw (0.3,-0.8) \botlabel{X} \braidup (-0.3,-0.2) \braidup (0.3,0.4) -- (0.3,0.8) \toplabel{X^\sigma};
            \draw[over] (-0.3,-0.8) \botlabel{Y} \braidup (0.3,-0.2) \braidup (-0.3,0.4) -- (-0.3,0.8) \toplabel{Y};
            \negdot{0.3,0.6};
        \end{tikzpicture}
        \ceqref{jaguar-}
        1_M\odot
        \begin{tikzpicture}[centerzero]
            \draw (0.3,-0.8) \botlabel{X} \braidup (-0.3,-0.2);
            \draw[over] (-0.3,-0.8) \botlabel{Y} \braidup (0.3,-0.2) \braidup (-0.3,0.4) -- (-0.3,0.8) \toplabel{Y};
            \draw[over] (-0.3,-0.2) \braidup (0.3,0.4) -- (0.3,0.8) \toplabel{X^\sigma};
            \negdot{-0.3,-0.2};
        \end{tikzpicture}
        \ceqref{affdotaction}
        \begin{tikzpicture}[centerzero]
            \draw (0.6,-0.8) \botlabel{X} \braidup (-0.6,0);
            \draw[over] (0,-0.8) \botlabel{Y} \braidup (0.6,0) \braidup (0,0.8) \toplabel{Y};
            \draw[over] (-0.6,-0.8) \botlabel{M} \braidup (0,0) \braidup (-0.6,0.8) \toplabel{M};
            \draw[over] (-0.6,0) \braidup (0.6,0.8) \toplabel{X^\sigma};
        \end{tikzpicture}
        \ .
    \]

    Using the $\sigma$-octagon identity just established, the third relation in \cref{jaguar-}, after applying the action, becomes the $\sigma$-multiplicativity axiom \cref{multaxiom}, by the same diagrammatic expansion as in the proof of \cref{affbraiding}.  This calculation is reversible: the intermediate $C$-strand is labeled $C^\sigma$ after passing through its negative dot.  Thus $\eta$ is a $\sigma$-braid.

    Conversely, suppose that $\eta$ is a $\sigma$-braid.  For each $X$, naturality in $M$ makes the family $(\eta_{M,X})_M$ an invertible natural transformation from $\bA(X)$ to $\bA(X^\sigma)$.  Assign this transformation to the negative-dot generator $\xi_X^{-1}$, retaining $\bA$ on $\cC$.  Naturality in $X$ gives the fourth relation in \cref{jaguar-}, the $\sigma$-octagon axiom gives the first by the equivalence established above, and the $\sigma$-multiplicativity axiom \cref{multaxiom} then gives the third by reversing the same diagrammatic expansion.  By the all-object presentation, this assignment extends uniquely to a strict monoidal functor on $\Aff_\sigma(\cC)$, hence to a strict action $\odot$ satisfying \cref{affdotaction}.

    If we start with a $\sigma$-braid $\eta$ and construct its extended action, then \cref{affdotaction} shows that the $\sigma$-braid obtained from this action is exactly $\eta$.

    Now start with an extended action $\odot$, form its $\sigma$-braid, and use that braiding to construct an action $\odot'$.  By \cref{affdotaction}, each negative-dot generator acts in the same way under $\odot'$ as under $\odot$.  The two actions also agree on $\cC$, since both extend the given action.  The all-object presentation therefore implies that the two actions agree on every morphism.  Hence the two constructions are inverse.
\end{proof}

A strict $\cC$-module functor between $\sigma$-braided module categories is \emph{braided} if it preserves the $\sigma$-braid.  A \emph{module natural transformation} between strict module functors is a natural transformation $\alpha$ satisfying $\alpha_{M\otimes X}=\alpha_M\otimes1_X$.  We write $\operatorname{Fun}^{\mathrm{br,str}}_{\cC}(\cM,\cN)$ for the category of strict braided module functors and module natural transformations.  Under \cref{affactionextension}, these are exactly the strict $\Aff_\sigma(\cC)$-module functors and their module natural transformations: compatibility with the additional negative-dot generators is precisely preservation of the $\sigma$-braid.

The following result shows that $\Aff_\sigma(\cC)$ is the \emph{free $\sigma$-braided right $\cC$-module category on one object}, with distinguished generator $\one$.

\begin{cor}[Freeness] \label{afffree}
    For every strict $\sigma$-braided right $\cC$-module category $\cM$, evaluation at $\one$ gives an isomorphism of categories, natural in $\cM$,
    \begin{equation} \label{affevaluation}
        \operatorname{Fun}^{\mathrm{br,str}}_{\cC}\bigl(\Aff_\sigma(\cC),\cM\bigr)
        \xrightarrow{\sim}\cM,\qquad \bF\longmapsto\bF(\one).
    \end{equation}
    Its inverse sends $M\in\Ob(\cM)$ to the unique strict braided module functor $\bF_M$ with $\bF_M(\one)=M$.  This functor is given by
    \[
        \bF_M(X)=M\otimes X,\qquad
        \bF_M\bigl(\imath_\cC(f)\bigr)=1_M\otimes f,
    \]
    and
    \begin{equation} \label{afffreegenerator}
        \bF_M\colon
        \begin{tikzpicture}[centerzero]
            \draw (0,-0.6) \botlabel{X} -- (0,0.6) \toplabel{X^\sigma};
            \negdot{0,0};
        \end{tikzpicture}
        \longmapsto
        \begin{tikzpicture}[centerzero]
            \draw (0.3,-0.6) \botlabel{X} \braidup (-0.3,0);
            \draw[over] (-0.3,-0.6) \botlabel{M} \braidup (0.3,0) \braidup (-0.3,0.6) \toplabel{M};
            \draw[over] (-0.3,0) \braidup (0.3,0.6) \toplabel{X^\sigma};
        \end{tikzpicture}
        \ .
    \end{equation}
    A morphism $a\colon M\to N$ corresponds to the module natural transformation with components $a\otimes1_X$.
\end{cor}

\begin{proof}
    Under \cref{affactionextension}, the canonical $\sigma$-braid $\eta^{\Aff}$ in \cref{affbraid} corresponds to the regular right $\Aff_\sigma(\cC)$-action on $\Aff_\sigma(\cC)$ given by its tensor product.  Let $\odot$ be the strict right $\Aff_\sigma(\cC)$-action on $\cM$ corresponding to its given $\sigma$-braid.

    For each $M\in\Ob(\cM)$, the functor $\bF_M$ defined by $\bF_M(f):=1_M\odot f$ is a strict $\Aff_\sigma(\cC)$-module functor, hence a strict braided $\cC$-module functor, with the stated values.  Conversely, any strict braided $\cC$-module functor $\bF\colon\Aff_\sigma(\cC)\to\cM$ is a strict $\Aff_\sigma(\cC)$-module functor for these actions.  It therefore satisfies $\bF(f)=1_{\bF(\one)}\odot f$ for every $f\in\Mor(\Aff_\sigma(\cC))$, so $\bF=\bF_{\bF(\one)}$.

    Finally, every $\cC$-module natural transformation $\alpha\colon\bF_M\Rightarrow\bF_N$ satisfies $\alpha_X=\alpha_\one\otimes1_X$.  Conversely, each morphism $a\colon M\to N$ defines a $\cC$-module natural transformation with components $a\otimes1_X$, by functoriality of the extended action $\odot$.
\end{proof}

\begin{rem} \label{affstrongfunctors}
    There is also a formulation using module functors that are not required to be strict.  A module functor $(\bH,\omega)\colon\cM\to\cN$ between $\sigma$-braided module categories is braided if
    \[
        \begin{tikzcd}[column sep=2.5cm]
            \bH(M\otimes X) \ar[r,"\bH(\eta^{\cM}_{M,X})"] \ar[d,"\omega_{M,X}"']
            & \bH(M\otimes X^\sigma) \ar[d,"\omega_{M,X^\sigma}"] \\
            \bH(M)\otimes X \ar[r,"\eta^{\cN}_{\bH(M),X}"']
            & \bH(M)\otimes X^\sigma
        \end{tikzcd}
    \]
    commutes.  A module natural transformation $\alpha\colon(\bF,\omega^{\bF})\Rightarrow(\bG,\omega^{\bG})$ satisfies
    \[
        \omega^{\bG}_{M,X}\circ\alpha_{M\otimes X}
        =(\alpha_M\otimes1_X)\circ\omega^{\bF}_{M,X}.
    \]
    With these functors and transformations, evaluation at $\one$ is an equivalence of categories.  Indeed, the components $\omega_{\one,X}\colon\bH(X)\to\bH(\one)\otimes X$ give a braided module natural isomorphism $\bH\cong\bF_{\bH(\one)}$; this follows from module coherence and compatibility with the braiding on the negative-dot generators.
\end{rem}

\begin{cor}[Tensor-pair cylinder $\sigma$-twists] \label{affcylindertwists}
    Let $(\cM,\cC)$ be a strict tensor pair, with inclusion $j\colon\cC\hookrightarrow\cM$.  There is a bijection between tensor-pair cylinder $\sigma$-twists $\tau$ for $(\cM,\cC)$ and strict $\cC$-module functors $\bF\colon\Aff_\sigma(\cC)\to\cM$ satisfying $\bF\circ\imath_\cC=j$.  The correspondence is determined by
    \begin{equation} \label{afftwistgenerator}
        \bF\colon
        \begin{tikzpicture}[centerzero]
            \draw (0,-0.5) \botlabel{X} -- (0,0.5) \toplabel{X^\sigma};
            \negdot{0,0};
        \end{tikzpicture}
        \longmapsto
        \begin{tikzpicture}[centerzero]
            \draw (0,-0.5) \botlabel{X} -- (0,0.5) \toplabel{X^\sigma};
            \coupon{0,0}{\tau};
        \end{tikzpicture}
        \ .
    \end{equation}
    This functor takes values in the subcategory $\cM^\tau$ of \cref{taugenerated}, and its corestriction to $\cM^\tau$ is braided for the $\sigma$-braid $E^\tau$.

    In particular, cylinder $\sigma$-twists on $\cC$ correspond to strict $\cC$-module retractions $r\colon\Aff_\sigma(\cC)\to\cC$ of $\imath_\cC$, with $\tau_X=r(\xi_X^{-1})$.
\end{cor}

\begin{proof}
    In \cref{jungleneg}, the defining relations \cref{griffin} are exactly \cref{taudub,taubasenaturality}.  The equality in parentheses in \cref{taudub} follows from the other defining conditions, so the presentation in \cref{jungleneg} gives the claimed bijection without requiring \cref{taunaturality} for all morphisms of $\cM$.

    By \cref{taugenerated}, the subcategory $\cM^\tau$ contains the image of $\cC$, all components of $\tau$ and their inverses, and is closed under the right $\cC$-action.  Thus it contains the image of $\bF$.  Applying $\bF$ to the last diagram in \cref{affbraid} gives precisely \cref{taureconstruct}, proving that the corestriction is braided.  Taking $\cM=\cC$ gives the final assertion.
\end{proof}

\begin{rem}[Relation to the literature]
    For $\sigma=\id_\cC$, Brochier proves an analogous freeness result for framed oriented tangles in the solid torus, viewed as a braided module category over the ordinary tangle category \cite[Thms.~5.2--5.3]{Bro13}.  The specialization of \cref{affactionextension} to $\sigma=\id_\cC$ is an algebraic counterpart of the description of braided module categories as modules over an annulus category in the factorization-homology setting of \cite[Thm.~3.11]{BBJ18}.  The distinction between annulus and cylinder framings is explained in \cite[Rem.~3.6]{BBJ18}; our statements use the cylinder convention of \cref{braidedmodcat}.  The proofs here use the dot presentation and require neither rigidity nor a balancing.
\end{rem}

\section{Applications\label{S:applications}}

In this final section, we apply the results of \cref{S:affmodule} to deduce module-category presentations of some categories that play a prominent role in the representation theory of quantum groups.  Such presentations are fundamental ingredients in diagrammatic treatments of quantum symmetric pairs; see \cite{SSSsub26,SSS26,SSS-AIII,SSS-BCD}.  Throughout this section, $\kk$ is an arbitrary commutative ring and $z,t \in \kk^\times$.

\subsection{The affine Kauffman skein category}

The \emph{Kauffman skein category} $\Kcat(z,t)$ is the strict $\kk$-linear monoidal category generated by an object $\Kobj$ and morphisms
\[
    \Kcap \colon \Kobj \otimes \Kobj \to \one,\quad
    \Kcup \colon \one \to \Kobj \otimes \Kobj,\quad
    \poscross,\, \negcross \colon \Kobj \otimes \Kobj \to \Kobj \otimes \Kobj,
\]
subject to the relations
\begin{gather} \label{braid}
    \begin{tikzpicture}[centerzero]
        \draw (0.2,-0.4) to[out=135,in=down] (-0.15,0) to[out=up,in=225] (0.2,0.4);
        \draw[over] (-0.2,-0.4) to[out=45,in=down] (0.15,0) to[out=up,in=-45] (-0.2,0.4);
    \end{tikzpicture}
    \ =\
    \begin{tikzpicture}[centerzero]
        \draw (-0.2,-0.4) -- (-0.2,0.4);
        \draw (0.2,-0.4) -- (0.2,0.4);
    \end{tikzpicture}
    \ =\
    \begin{tikzpicture}[centerzero]
        \draw (-0.2,-0.4) to[out=45,in=down] (0.15,0) to[out=up,in=-45] (-0.2,0.4);
        \draw[over] (0.2,-0.4) to[out=135,in=down] (-0.15,0) to[out=up,in=225] (0.2,0.4);
    \end{tikzpicture}
    \ ,\qquad
    \begin{tikzpicture}[centerzero]
        \draw (0.4,-0.4) -- (-0.4,0.4);
        \draw[over] (0,-0.4) to[out=135,in=down] (-0.32,0) to[out=up,in=225] (0,0.4);
        \draw[over] (-0.4,-0.4) -- (0.4,0.4);
    \end{tikzpicture}
    \ =\
    \begin{tikzpicture}[centerzero]
        \draw (0.4,-0.4) -- (-0.4,0.4);
        \draw[over] (0,-0.4) to[out=45,in=down] (0.32,0) to[out=up,in=-45] (0,0.4);
        \draw[over] (-0.4,-0.4) -- (0.4,0.4);
    \end{tikzpicture}
    \ ,
    \\ \label{drops}
    \begin{tikzpicture}[centerzero]
        \draw (-0.3,0.4) -- (-0.3,0) arc(180:360:0.15) arc(180:0:0.15) -- (0.3,-0.4);
    \end{tikzpicture}
    =
    \begin{tikzpicture}[centerzero]
        \draw (0,-0.4) -- (0,0.4);
    \end{tikzpicture}
    = \
    \begin{tikzpicture}[centerzero]
        \draw (-0.3,-0.4) -- (-0.3,0) arc(180:0:0.15) arc(180:360:0.15) -- (0.3,0.4);
    \end{tikzpicture}
    \ ,\qquad
    \begin{tikzpicture}[anchorbase]
        \draw (-0.15,-0.4) to[out=60,in=-90] (0.15,0) arc(0:180:0.15);
        \draw[over] (-0.15,0) to[out=-90,in=120] (0.15,-0.4);
    \end{tikzpicture}
    = t \
    \Kcap
    \ ,\qquad
    \begin{tikzpicture}[centerzero]
        \draw (-0.2,-0.3) -- (-0.2,-0.1) arc(180:0:0.2) -- (0.2,-0.3);
        \draw[over] (-0.3,0.3) \braiddown (0,-0.3);
    \end{tikzpicture}
    =
    \begin{tikzpicture}[centerzero]
        \draw (-0.2,-0.3) -- (-0.2,-0.1) arc(180:0:0.2) -- (0.2,-0.3);
        \draw[over] (0.3,0.3) \braiddown (0,-0.3);
    \end{tikzpicture}
    \ ,
    \\ \label{skein}
    \begin{tikzpicture}[centerzero]
        \draw (0.35,-0.35) -- (-0.35,0.35);
        \draw[over] (-0.35,-0.35) -- (0.35,0.35);
    \end{tikzpicture}
    \ -\
    \begin{tikzpicture}[centerzero]
        \draw (-0.35,-0.35) -- (0.35,0.35);
        \draw[over] (0.35,-0.35) -- (-0.35,0.35);
    \end{tikzpicture}
    = z
    \left(\
        \begin{tikzpicture}[centerzero]
            \draw (-0.2,-0.35) -- (-0.2,0.35);
            \draw (0.2,-0.35) -- (0.2,0.35);
        \end{tikzpicture}
        \ -\
        \begin{tikzpicture}[centerzero]
            \draw (-0.2,-0.35) -- (-0.2,-0.3) arc(180:0:0.2) -- (0.2,-0.35);
            \draw (-0.2,0.35) -- (-0.2,0.3) arc(180:360:0.2) -- (0.2,0.35);
        \end{tikzpicture}\
    \right)
    ,\qquad
    \Kbub\ = \frac{t-t^{-1}}{z}+1.
\end{gather}
The first relation in \cref{skein} is called the (Dubrovnik form of the) \emph{Kauffman skein relation}.

As shown in \cite[\S8.3]{MS21}, the affinization of $\Kcat(z,t)$, called the \emph{affine Kauffman skein category} $\AKcat(z,t)$, is the $\kk$-linear strict monoidal category obtained from $\Kcat(z,t)$ by adjoining an invertible morphism
\[
    \Kopendot \colon \Kobj \to \Kobj,
\]
which we call a \emph{dot}, subject to the relations
\begin{equation} \label{Kdotslide}
    \begin{tikzpicture}[centerzero]
        \draw (-0.3,-0.3) -- (0.3,0.3);
        \draw[over] (0.3,-0.3) -- (-0.3,0.3);
        \opendot{-0.19,-0.19};
    \end{tikzpicture}
    =
    \begin{tikzpicture}[centerzero]
        \draw (0.3,-0.3) -- (-0.3,0.3);
        \draw[over] (-0.3,-0.3) -- (0.3,0.3);
        \opendot{0.15,0.15};
    \end{tikzpicture}
    \ ,\qquad
    \begin{tikzpicture}[centerzero]
        \draw (-0.2,-0.2) -- (-0.2,0) arc(180:0:0.2) -- (0.2,-0.2);
        \opendot{-0.2,0};
    \end{tikzpicture}
    =
    \begin{tikzpicture}[centerzero]
        \draw (-0.2,-0.2) -- (-0.2,0) arc(180:0:0.2) -- (0.2,-0.2);
        \closeddot{0.2,0};
    \end{tikzpicture}
    \ ,
\end{equation}
where $\Kcloseddot$ denotes the inverse of the dot.  This category first appeared in \cite[Def.~1.3]{GRS22}.

\begin{prop}\label{AKpresent}
    As a right $\Kcat(z,t)$-module, the affine Kauffman skein category $\AKcat(z,t)$ is generated by an invertible morphism
    \begin{equation} \label{AKdot}
        \Kopendot \colon \Kobj \to \Kobj,
    \end{equation}
    subject to the relations
    \begin{equation} \label{AKmodrels}
        \begin{tikzpicture}[centerzero]
            \draw (-0.25,0) \braidup (0.25,0.5) -- (0.25,0.8);
            \draw[over] (-0.25,-0.8) -- (-0.25,-0.5) \braidup (0.25,0) \braidup (-0.25,0.5) -- (-0.25,0.8);
            \draw[over] (0.25,-0.8) -- (0.25,-0.5) \braidup (-0.25,0);
            \opendot{-0.25,-0.55};
            \opendot{-0.25,0};
        \end{tikzpicture}
        =
        \begin{tikzpicture}[centerzero]
            \draw (-0.25,0) \braidup (0.25,0.5) -- (0.25,0.8);
            \draw[over] (-0.25,-0.8) -- (-0.25,-0.5) \braidup (0.25,0) \braidup (-0.25,0.5) -- (-0.25,0.8);
            \draw[over] (0.25,-0.8) -- (0.25,-0.5) \braidup (-0.25,0);
            \opendot{-0.25,0.5};
            \opendot{-0.25,0};
        \end{tikzpicture}
        \ ,\qquad
        \begin{tikzpicture}[anchorbase]
            \draw (0.25,0.5) to[out=down,in=45] (-0.25,0);
            \draw[over] (0.25,0) to[out=135,in=down] (-0.25,0.5) arc(180:0:0.25);
            \opendot{-0.25,0.5};
        \end{tikzpicture}
        = t^{-1}\,
        \begin{tikzpicture}[anchorbase]
            \draw (-0.25,0) -- (-0.25,0.5) to[out=up,in=up,looseness=1.7] (0.25,0.5) -- (0.25,0);
            \closeddot{-0.25,0.3};
        \end{tikzpicture}
    \end{equation}
    Alternatively, it is generated by an invertible morphism
    \begin{equation} \label{AKdotinv}
        \Kcloseddot \colon \Kobj \to \Kobj,
    \end{equation}
    subject to the relations
    \begin{equation} \label{AKmodrelsinv}
        \begin{tikzpicture}[centerzero]
            \draw (0.25,-0.8) -- (0.25,-0.5) \braidup (-0.25,0);
            \draw[over] (-0.25,-0.8) -- (-0.25,-0.5) \braidup (0.25,0) \braidup (-0.25,0.5) -- (-0.25,0.8);
            \draw[over] (-0.25,0) \braidup (0.25,0.5) -- (0.25,0.8);
            \closeddot{-0.25,-0.5};
            \closeddot{-0.25,0};
        \end{tikzpicture}
        =
        \begin{tikzpicture}[centerzero]
            \draw (0.25,-0.8) -- (0.25,-0.5) \braidup (-0.25,0);
            \draw[over] (-0.25,-0.8) -- (-0.25,-0.5) \braidup (0.25,0) \braidup (-0.25,0.5) -- (-0.25,0.8);
            \draw[over] (-0.25,0) \braidup (0.25,0.5) -- (0.25,0.8);
            \closeddot{-0.25,0.5};
            \closeddot{-0.25,0};
        \end{tikzpicture}
        \ ,\qquad
        \begin{tikzpicture}[anchorbase]
            \draw (0.25,0) to[out=135,in=down] (-0.25,0.5) arc(180:0:0.25);
            \draw[over] (0.25,0.5) to[out=down,in=45] (-0.25,0);
            \closeddot{-0.25,0.5};
        \end{tikzpicture}
        = t\,
        \begin{tikzpicture}[anchorbase]
            \draw (-0.25,0) -- (-0.25,0.5) to[out=up,in=up,looseness=1.7] (0.25,0.5) -- (0.25,0);
            \opendot{-0.25,0.3};
        \end{tikzpicture}
    \end{equation}
\end{prop}

\begin{proof}
    By \cref{swamp+}, writing $\Kopendot$ instead of
    \(
        \begin{tikzpicture}[centerzero]
            \draw (0,-0.25) -- (0,0.25);
            \posdot{0,0};
        \end{tikzpicture}
    \),
    $\AKcat(z,t)$ is generated by \cref{AKdot}, subject to the first relation in \cref{AKmodrels} and the second relation in \cref{life} for $f \in \{\Kcup, \Kcap\}$.  Taking $f = \Kcap$ and using the first equality in \cref{terraform} to define the dot on $\Kobj \otimes \Kobj$, the second relation in \cref{life} is
    \[
        \begin{tikzpicture}[centerzero]
            \draw (-0.25,0) \braidup (0.25,0.5) arc(0:180:0.25);
            \draw[over] (-0.25,-0.8) -- (-0.25,-0.5) \braidup (0.25,0) \braidup (-0.25,0.5);
            \draw[over] (0.25,-0.8) -- (0.25,-0.5) \braidup (-0.25,0);
            \opendot{-0.25,-0.55};
            \opendot{-0.25,0};
        \end{tikzpicture}
        =
        \begin{tikzpicture}[anchorbase]
            \draw (-0.3,0) -- (-0.3,0.5) to[out=up,in=up,looseness=1.7] (0.3,0.5) -- (0.3,0);
        \end{tikzpicture}
        \ \overset{\cref{drops}}{\Longleftrightarrow}\  t\
        \begin{tikzpicture}[anchorbase]
            \draw (-0.25,-0.8) -- (-0.25,-0.5) \braidup (0.25,0) arc(0:180:0.25);
            \draw[over] (0.25,-0.8) -- (0.25,-0.5) \braidup (-0.25,0);
            \opendot{-0.25,-0.55};
            \opendot{-0.25,0};
        \end{tikzpicture}
        =
        \begin{tikzpicture}[anchorbase]
            \draw (-0.3,0) -- (-0.3,0.5) to[out=up,in=up,looseness=1.7] (0.3,0.5) -- (0.3,0);
        \end{tikzpicture}
        \ \Longleftrightarrow\
        \begin{tikzpicture}[anchorbase]
            \draw (0.25,0.5) to[out=down,in=45] (-0.25,0);
            \draw[over] (0.25,0) to[out=135,in=down] (-0.25,0.5) arc(180:0:0.25);
            \opendot{-0.25,0.5};
        \end{tikzpicture}
        = t^{-1}\,
        \begin{tikzpicture}[anchorbase]
            \draw (-0.25,0) -- (-0.25,0.5) to[out=up,in=up,looseness=1.7] (0.25,0.5) -- (0.25,0);
            \closeddot{-0.25,0.3};
        \end{tikzpicture}
        \ .
    \]
    On the other hand, taking $f = \Kcup$, the second relation in \cref{life} is
    \[
        \begin{tikzpicture}[centerzero]
            \draw (-0.25,0) \braidup (0.25,0.5) -- (0.25,0.8);
            \draw[over] (0.25,-0.5) arc(360:180:0.25) \braidup (0.25,0) \braidup (-0.25,0.5) -- (-0.25,0.8);
            \draw[over] (0.25,-0.5) \braidup (-0.25,0);
            \opendot{-0.25,0.5};
            \opendot{-0.25,0};
        \end{tikzpicture}
        =
        \begin{tikzpicture}[anchorbase]
            \draw (-0.3,0) -- (-0.3,-0.5) to[out=down,in=down,looseness=1.7] (0.3,-0.5) -- (0.3,0);
        \end{tikzpicture}
        \ \overset{\cref{drops}}{\Longleftrightarrow}\  t\
        \begin{tikzpicture}[anchorbase]
            \draw (-0.25,0) \braidup (0.25,0.5) -- (0.25,0.8);
            \draw[over] (-0.25,0) arc(180:360:0.25) \braidup (-0.25,0.5) -- (-0.25,0.8);
            \opendot{-0.25,0.5};
            \opendot{-0.25,0};
        \end{tikzpicture}
        =
        \begin{tikzpicture}[anchorbase]
            \draw (-0.3,0) -- (-0.3,-0.5) to[out=down,in=down,looseness=1.7] (0.3,-0.5) -- (0.3,0);
        \end{tikzpicture}
        \ \Longleftrightarrow\
        \begin{tikzpicture}[anchorbase]
            \draw (0.25,0) to[out=-135,in=up] (-0.25,-0.5) arc(180:360:0.25);
            \draw[over] (0.25,-0.5) to[out=up,in=-45] (-0.25,0);
            \opendot{-0.25,-0.5};
        \end{tikzpicture}
        = t^{-1}\,
        \begin{tikzpicture}[anchorbase]
            \draw (-0.25,0) -- (-0.25,-0.5) to[out=down,in=down,looseness=1.7] (0.25,-0.5) -- (0.25,0);
            \closeddot{-0.25,-0.3};
        \end{tikzpicture}
        \ ,
    \]
    which is equivalent to the last relation in \cref{AKmodrels} using caps and cups to rotate diagrams.
    
    The proof of the alternate presentation is analogous, using \cref{swamp-}.
\end{proof}

\subsection{The affine oriented skein category\label{SS:AOScat}}

The \emph{framed HOMFLYPT skein category}, or \emph{oriented skein category} for short, $\OScat(z,t)$, is the $\kk$-linear strict monoidal category generated by objects $\upobj$, $\downobj$ and morphisms
\begin{gather*}
    \posupcross,\ \negupcross \colon \upobj \otimes \upobj \to \upobj \otimes \upobj,
    \\
    \rightcup \colon \one \to \downobj \otimes \upobj,\quad
    \rightcap \colon \upobj \otimes \downobj \to \one,\quad
    \leftcup \colon \one \to \upobj \otimes \downobj,\quad
    \leftcap \colon \downobj \otimes \upobj \to \one,
\end{gather*}
subject to the relations
\begin{gather} \label{obraid}

    \ .
\end{gather}
In \cref{obraid}, we have used the following morphisms:
\begin{equation} \label{lego}
    \posrightcross :=
    %
    \ .
\end{equation}

This category was first introduced in \cite[\S5.2]{Tur89}, where it was called the \emph{Hecke category}. In \cite[Def.~2.1]{QS19} it is called the \emph{quantized oriented Brauer category}.  It was studied in depth in \cite{Bru17}.

We define the following morphisms:
\begin{equation} \label{windmill}
    \negrightcross:=
    %
    \ .
\end{equation}
The category $\OScat(z,t)$ is strict pivotal.  Furthermore, the following relations hold for any orientation of the strands:
\begin{equation} \label{venom1}
    %
    \ ,\quad
    \unbub=\frac{t-t^{-1}}{z}.
\end{equation}
We also have
\begin{equation} \label{ocurlinv}
    %
    \ .
\end{equation}

As shown in \cite[\S7.3]{MS21}, the affinization of $\OScat(z,t)$, called the \emph{affine oriented skein category} $\AOScat(z,t) = \Aff(\OScat(z,t))$, is the $\kk$-linear strict monoidal category obtained from $\OScat(z,t)$ by adjoining an invertible morphism
\[
    \upopendot = \upposdot \colon \upobj \to \upobj,
\]
which we call a \emph{dot}, subject to the relation
\begin{equation} \label{turkey}
    %
    \ .
\end{equation}
We let $\upcloseddot = \upnegdot$ denote the inverse of the dot, so that
\begin{equation} \label{lime}
    %
    \ .
\end{equation}
We define the downward dots by
\begin{equation} \label{cranberry}
    %
    \ ,
\end{equation}
using \cite[(4.8)]{MS21}.

\begin{prop} \label{AOSpresent}
    As a right $\OScat(z,t)$-module, the affine oriented skein category $\AOScat(z,t)$ is generated by an invertible morphism
    \[
        \upopendot \colon \upobj \to \upobj,
    \]
    subject to the relation
    \begin{equation} \label{AOSreflect}
        %
        \ .
    \end{equation}
    Alternatively, it is generated by an invertible morphism
    \[
        \upcloseddot \colon \upobj \to \upobj,
    \]
    subject to the relation
    \[
        %
        \ .
    \]
\end{prop}

\begin{proof}
    Let $\Dcat$ denote the category presented as in the statement of the proposition.  We draw its string diagrams in cyan to distinguish them from diagrams in $\AOScat(z,t)$.  We define
    \begin{equation}\label{tornado}
        \upposdot[present] := \upopendot[present]
        \ ,\qquad
        \upnegdot[present] := \left( \upopendot[present] \right)^{-1} = \upcloseddot[present]
        \ ,\qquad
        \downposdot[present] :=
        t^{-1}\,
        %
    \end{equation}
    Composing with $\upcloseddot \upstrand$ on the top and bottom of \cref{AOSreflect} and cancelling the inverse dots gives
    \begin{equation}\label{lemon}
        %
        \ .
    \end{equation}
    The same calculation with all crossings reversed, open and closed dots interchanged, and $t$ replaced by $t^{-1}$ uses the relation obtained by taking the inverses of both sides of \cref{AOSreflect} and gives
    \begin{equation}\label{cato2}
        %
        \ .
    \end{equation}
    Thus, $\downnegdot[present] = \left( \downposdot[present] \right)^{-1}$.
    
    By \cref{swamp+}, $\AOScat(z,t)$ is generated as a right $\OScat(z,t)$-module by invertible morphisms
    \[
        \upposdot \colon \upobj \to \upobj
        \qquad \text{and} \qquad
        \downposdot \colon \downobj \to \downobj
    \]
    subject to the relations
    \begin{equation} \label{alarm}
        %
        \ ,\qquad
    \end{equation}
    and the second relation in \cref{life} for $f \in \{\rightcap, \rightcup, \leftcap, \leftcup\}$.  We claim that there is an isomorphism of $\OScat(z,t)$-module categories
    \[
        \bF \colon \AOScat(z,t) \to \Dcat,\qquad
        \upposdot \mapsto \upposdot[present]\, ,\qquad
        \downposdot \mapsto \downposdot[present]\, .
    \]
    To show that $\bF$ is well defined, we must show that the cyan versions of the relations \cref{alarm} hold and that, for each $f \in \{\rightcap[present], \rightcup[present], \leftcap[present], \leftcup[present]\}$, the cyan version of the second relation in \cref{life} holds.
    
    When $f = \rightcap[present]$, 
    \[
        \begin{tikzpicture}[anchorbase,present]
            \draw (0,-0.7) -- (0,0.7);
            \posdot{0,-0.35};
            \coupon{0,0.25}{f};
        \end{tikzpicture}
        =
        \begin{tikzpicture}[centerzero,present]
            \draw (-0.25,0) \braidup (0.25,0.5) arc(0:180:0.25);
            \draw[over] (-0.25,-0.8) -- (-0.25,-0.5) \braidup (0.25,0) \braidup (-0.25,0.5);
            \draw[over,<-] (0.25,-0.8) -- (0.25,-0.5) \braidup (-0.25,0);
            \posdot{-0.25,-0.55};
            \posdot{-0.25,0};
        \end{tikzpicture}
        \ceqref{ocurlbub} t\
        \begin{tikzpicture}[anchorbase,present]
            \draw (-0.25,-0.8) -- (-0.25,-0.5) \braidup (0.25,0) arc(0:180:0.25);
            \draw[over,<-] (0.25,-0.8) -- (0.25,-0.5) \braidup (-0.25,0);
            \posdot{-0.25,-0.55};
            \posdot{-0.25,0};
        \end{tikzpicture}
        \ceqref{tornado}
        \begin{tikzpicture}[anchorbase,present]
            \draw[->] (-0.25,-0.8) -- (-0.25,0) arc(180:0:0.25) -- (0.25,-0.8);
            \opendot{-0.25,-0.5};
            \closeddot{-0.25,-0.1};
        \end{tikzpicture}
        =
        \begin{tikzpicture}[anchorbase,present]
            \draw[->] (-0.25,-0.4) -- (-0.25,0) arc(180:0:0.25) -- (0.25,-0.4);
        \end{tikzpicture}
        =
        \begin{tikzpicture}[anchorbase,present]
            \draw (0,-0.7) -- (0,0.7);
            \posdot{0,0.35};
            \coupon{0,-0.25}{f};
        \end{tikzpicture}
        \ ,
    \]
    proving the second relation in \cref{life}.  The proofs for $f \in \{\rightcup[present], \leftcap[present], \leftcup[present]\}$ are analogous.
    \details{
        When $f = \rightcup$,
        \[
            \begin{tikzpicture}[anchorbase,present]
                \draw (0,-0.7) -- (0,0.7);
                \posdot{0,-0.35};
                \coupon{0,0.25}{f};
            \end{tikzpicture}
            =
            \begin{tikzpicture}[centerzero,present]
                \draw[->] (-0.25,0) \braidup (0.25,0.5) -- (0.25,0.8);
                \draw[over] (0.25,-0.5) arc(360:180:0.25) \braidup (0.25,0) \braidup (-0.25,0.5) -- (-0.25,0.8);
                \draw[over] (0.25,-0.5) \braidup (-0.25,0);
                \posdot{-0.25,0.5};
                \posdot{-0.25,0};
            \end{tikzpicture}
            \ceqref{ocurlbub}
            t\
            \begin{tikzpicture}[anchorbase,present]
                \draw[->] (-0.25,0) \braidup (0.25,0.5) -- (0.25,0.8);
                \draw[over] (-0.25,0) arc(180:360:0.25) \braidup (0.25,0) \braidup (-0.25,0.5) -- (-0.25,0.8);
                \posdot{-0.25,0.5};
                \posdot{-0.25,0};
            \end{tikzpicture}
            \ceqref{tornado}
            \begin{tikzpicture}[anchorbase,present]
                \draw[->] (-0.25,0.8) -- (-0.25,0) arc(180:360:0.25) -- (0.25,0.8);
                \negdot{-0.25,0.1};
                \posdot{-0.25,0.5};
            \end{tikzpicture}
            \ceqref{cato1}
            \begin{tikzpicture}[anchorbase,present]
                \draw[->] (-0.25,0.4) -- (-0.25,0) arc(180:360:0.25) -- (0.25,0.4);
            \end{tikzpicture}
            =
            \begin{tikzpicture}[anchorbase,present]
                \draw (0,-0.7) -- (0,0.7);
                \posdot{0,0.35};
                \coupon{0,-0.25}{f};
            \end{tikzpicture}
            \ .
        \]
        When $f = \leftcap\,$,
        \[
            \begin{tikzpicture}[anchorbase,present]
                \draw (0,-0.7) -- (0,0.7);
                \posdot{0,-0.35};
                \coupon{0,0.25}{f};
            \end{tikzpicture}
            =
            \begin{tikzpicture}[centerzero,present]
                \draw (-0.25,0) \braidup (0.25,0.5) arc(0:180:0.25);
                \draw[over,<-] (-0.25,-0.8) -- (-0.25,-0.5) \braidup (0.25,0) \braidup (-0.25,0.5);
                \draw[over] (0.25,-0.8) -- (0.25,-0.5) \braidup (-0.25,0);
                \posdot{-0.25,-0.55};
                \posdot{-0.25,0};
            \end{tikzpicture}
            \ceqref{ocurlbub} t\
            \begin{tikzpicture}[anchorbase,present]
                \draw[<-] (-0.25,-0.8) -- (-0.25,-0.5) \braidup (0.25,0) arc(0:180:0.25);
                \draw[over] (0.25,-0.8) -- (0.25,-0.5) \braidup (-0.25,0);
                \posdot{-0.25,-0.55};
                \posdot{-0.25,0};
            \end{tikzpicture}
            \ceqref{tornado}
            \begin{tikzpicture}[anchorbase,present]
                \draw[<-] (-0.25,-0.8) -- (-0.25,0) arc(180:0:0.25) -- (0.25,-0.8);
                \posdot{-0.25,-0.5};
                \negdot{-0.25,-0.1};
            \end{tikzpicture}
            \ceqref{cato2}
            \begin{tikzpicture}[anchorbase,present]
                \draw[<-] (-0.25,-0.4) -- (-0.25,0) arc(180:0:0.25) -- (0.25,-0.4);
            \end{tikzpicture}
            =
            \begin{tikzpicture}[anchorbase,present]
                \draw (0,-0.7) -- (0,0.7);
                \posdot{0,0.35};
                \coupon{0,-0.25}{f};
            \end{tikzpicture}
            \ .
        \]
        When $f = \leftcup\,$,
        \[
            \begin{tikzpicture}[anchorbase,present]
                \draw (0,-0.7) -- (0,0.7);
                \posdot{0,-0.35};
                \coupon{0,0.25}{f};
            \end{tikzpicture}
            =
            \begin{tikzpicture}[centerzero,present]
                \draw (-0.25,0) \braidup (0.25,0.5) -- (0.25,0.8);
                \draw[over,->] (0.25,-0.5) arc(360:180:0.25) \braidup (0.25,0) \braidup (-0.25,0.5) -- (-0.25,0.8);
                \draw[over] (0.25,-0.5) \braidup (-0.25,0);
                \posdot{-0.25,0.5};
                \posdot{-0.25,0};
            \end{tikzpicture}
            \ceqref{ocurlbub}
            t\
            \begin{tikzpicture}[anchorbase,present]
                \draw (-0.25,0) \braidup (0.25,0.5) -- (0.25,0.8);
                \draw[over,->] (-0.25,0) arc(180:360:0.25) \braidup (0.25,0) \braidup (-0.25,0.5) -- (-0.25,0.8);
                \posdot{-0.25,0.5};
                \posdot{-0.25,0};
            \end{tikzpicture}
            \ceqref{tornado}
            \begin{tikzpicture}[anchorbase,present]
                \draw[<-] (-0.25,0.8) -- (-0.25,0) arc(180:360:0.25) -- (0.25,0.8);
                \negdot{-0.25,0.1};
                \posdot{-0.25,0.5};
            \end{tikzpicture}
            =
            \begin{tikzpicture}[anchorbase,present]
                \draw[<-] (-0.25,0.4) -- (-0.25,0) arc(180:360:0.25) -- (0.25,0.4);
            \end{tikzpicture}
            =
            \begin{tikzpicture}[anchorbase,present]
                \draw (0,-0.7) -- (0,0.7);
                \posdot{0,0.35};
                \coupon{0,-0.25}{f};
            \end{tikzpicture}
            \ .
        \]
    }
    
    The first relation in \cref{alarm} is \cref{AOSreflect}.  For the second relation in \cref{alarm}, we have
    \[
        \begin{tikzpicture}[centerzero,present]
            \draw (-0.25,0) \braidup (0.25,0.5) -- (0.25,0.8);
            \draw[over,->] (-0.25,-0.8) -- (-0.25,-0.5) \braidup (0.25,0) \braidup (-0.25,0.5) -- (-0.25,0.8);
            \draw[over,<-] (0.25,-0.8) -- (0.25,-0.5) \braidup (-0.25,0);
            \posdot{-0.25,-0.55};
            \posdot{-0.25,0};
        \end{tikzpicture}
        \ceqref{tornado}
        \begin{tikzpicture}[centerzero,present]
            \draw (-0.25,0) \braidup (0.25,0.5) -- (0.25,0.8);
            \draw[over,->] (-0.25,-0.8) -- (-0.25,-0.5) \braidup (0.25,0) \braidup (-0.25,0.5) -- (-0.25,0.8);
            \draw[over,<-] (0.25,-0.8) -- (0.25,-0.5) \braidup (-0.25,0);
            \opendot{-0.25,-0.55};
            \closeddot{-0.25,0};
        \end{tikzpicture}
        \ceqsref{tornado}{venom2} t^{-1}\ 
        \begin{tikzpicture}[centerzero,present]
            \draw (1.05,1.15) -- (1.05,-0.25) to[out=down,in=0] (0.55,-0.9) to[out=180,in=down] (0.25,-0.7) -- (0.25,-0.5);
            \draw (0.25,-0.5) \braidup (-0.25,0);
            \draw[over,->] (-0.25,-1.15) -- (-0.25,-0.5) \braidup (0.25,0) \braidup (-0.25,0.5) -- (-0.25,1.15);
            \draw[over,->] (-0.25,0) \braidup (0.25,0.5) -- (0.25,0.7) to[out=up,in=up,looseness=1.7] (0.75,0.7) -- (0.75,-1.15);
            \opendot{-0.25,-0.7};
            \closeddot{-0.25,0};
        \end{tikzpicture}
        \ceqref{lemon} t^{-1}\ 
        \begin{tikzpicture}[centerzero,present]
            \draw (1.05,1.15) -- (1.05,-0.25) to[out=down,in=0] (0.55,-0.9) to[out=180,in=down] (0.25,-0.7) -- (0.25,-0.5);
            \draw (0.25,-0.5) \braidup (-0.25,0);
            \draw[over,->] (-0.25,-1.15) -- (-0.25,-0.5) \braidup (0.25,0) \braidup (-0.25,0.5) -- (-0.25,1.15);
            \draw[over,->] (-0.25,0) \braidup (0.25,0.5) -- (0.25,0.7) to[out=up,in=up,looseness=1.7] (0.75,0.7) -- (0.75,-1.15);
            \opendot{-0.25,0.7};
            \closeddot{-0.25,0};
        \end{tikzpicture}
        \ceqsref{venom2}{tornado}
        \begin{tikzpicture}[centerzero,present]
            \draw (-0.25,0) \braidup (0.25,0.5) -- (0.25,0.8);
            \draw[over,->] (-0.25,-0.8) -- (-0.25,-0.5) \braidup (0.25,0) \braidup (-0.25,0.5) -- (-0.25,0.8);
            \draw[over,<-] (0.25,-0.8) -- (0.25,-0.5) \braidup (-0.25,0);
            \opendot{-0.25,0.55};
            \closeddot{-0.25,0};
        \end{tikzpicture}
        \ceqref{tornado}
        \begin{tikzpicture}[centerzero,present]
            \draw (-0.25,0) \braidup (0.25,0.5) -- (0.25,0.8);
            \draw[over,->] (-0.25,-0.8) -- (-0.25,-0.5) \braidup (0.25,0) \braidup (-0.25,0.5) -- (-0.25,0.8);
            \draw[over,<-] (0.25,-0.8) -- (0.25,-0.5) \braidup (-0.25,0);
            \posdot{-0.25,0.5};
            \posdot{-0.25,0};
        \end{tikzpicture}
        \ .
    \]
    Here the middle equality uses the relation obtained by taking inverses of both sides of \cref{lemon}; its crossings are reversed, its middle dot is closed, and its other dot is open.

    For the third relation in \cref{alarm}, we use \cref{venom1,venom2} to move the cups and caps past the crossings, obtaining
    \[
        \begin{tikzpicture}[centerzero,present]
            \draw[->] (-0.25,0) \braidup (0.25,0.5) -- (0.25,0.8);
            \draw[over,<-] (-0.25,-0.8) -- (-0.25,-0.5) \braidup (0.25,0) \braidup (-0.25,0.5) -- (-0.25,0.8);
            \draw[over] (0.25,-0.8) -- (0.25,-0.5) \braidup (-0.25,0);
            \posdot{-0.25,-0.55};
            \posdot{-0.25,0};
        \end{tikzpicture}
        \ceqref{tornado}
        \begin{tikzpicture}[centerzero,present]
            \draw[->] (-0.25,0) \braidup (0.25,0.5) -- (0.25,0.8);
            \draw[over,<-] (-0.25,-0.8) -- (-0.25,-0.5) \braidup (0.25,0) \braidup (-0.25,0.5) -- (-0.25,0.8);
            \draw[over] (0.25,-0.8) -- (0.25,-0.5) \braidup (-0.25,0);
            \closeddot{-0.25,-0.55};
            \opendot{-0.25,0};
        \end{tikzpicture}
        \ceqref{tornado} t^{-1}\ 
        \begin{tikzpicture}[centerzero,present]
            \draw[->] (0,0) \braidup (1.2,1) -- (1.2,1.4);
            \draw (0.8,-0.8) to[out=down,in=down,looseness=1.2] (0,-0.8) -- (0,-0.6);
            \draw[over] (0.8,1.4) -- (0.8,1) \braiddown (1.2,0) \braiddown (0.8,-0.8);
            \draw[over,->] (0,-0.6) -- (0,-0.4) \braidup (0.4,0) \braidup (0,0.4) -- (0,0.75) arc(180:0:0.2) \braiddown (0.8,0) \braiddown (0.4,-0.75) -- (0.4,-1.6);
            \draw[over] (1.2,-1.6) -- (1.2,-1) \braidup (0,0);
            \closeddot{0,-0.6};
            \opendot{0,0};
        \end{tikzpicture}
        \ceqref{lemon} t^{-1}\ 
        \begin{tikzpicture}[centerzero,present]
            \draw[->] (0,0) \braidup (1.2,1) -- (1.2,1.4);
            \draw (0.8,-0.8) to[out=down,in=down,looseness=1.2] (0,-0.8) -- (0,-0.6);
            \draw[over] (0.8,1.4) -- (0.8,1) \braiddown (1.2,0) \braiddown (0.8,-0.8);
            \draw[over,->] (0,-0.6) -- (0,-0.4) \braidup (0.4,0) \braidup (0,0.4) -- (0,0.75) arc(180:0:0.2) \braiddown (0.8,0) \braiddown (0.4,-0.75) -- (0.4,-1.6);
            \draw[over] (1.2,-1.6) -- (1.2,-1) \braidup (0,0);
            \closeddot{0,0.6};
            \opendot{0,0};
        \end{tikzpicture}
        \ceqref{tornado}
        \begin{tikzpicture}[centerzero,present]
            \draw[->] (-0.25,0) \braidup (0.25,0.5) -- (0.25,0.8);
            \draw[over,<-] (-0.25,-0.8) -- (-0.25,-0.5) \braidup (0.25,0) \braidup (-0.25,0.5) -- (-0.25,0.8);
            \draw[over] (0.25,-0.8) -- (0.25,-0.5) \braidup (-0.25,0);
            \closeddot{-0.25,0.55};
            \opendot{-0.25,0};
        \end{tikzpicture}
        \ceqref{tornado}
        \begin{tikzpicture}[centerzero,present]
            \draw[->] (-0.25,0) \braidup (0.25,0.5) -- (0.25,0.8);
            \draw[over,<-] (-0.25,-0.8) -- (-0.25,-0.5) \braidup (0.25,0) \braidup (-0.25,0.5) -- (-0.25,0.8);
            \draw[over] (0.25,-0.8) -- (0.25,-0.5) \braidup (-0.25,0);
            \posdot{-0.25,0.5};
            \posdot{-0.25,0};
        \end{tikzpicture}
        \ .
    \]
 
    To prove the fourth relation in \cref{alarm}, first compose the second relation in \cref{alarm}, which we have already proved, with $\upcloseddot[present] \downstrand[present]$ on the top and bottom.  Cancelling inverse dots gives
    \begin{equation} \label{lemon2}
        \begin{tikzpicture}[centerzero,present]
            \draw (-0.25,0) \braidup (0.25,0.5) -- (0.25,0.8);
            \draw[over,->] (-0.25,-0.8) -- (-0.25,-0.5) \braidup (0.25,0) \braidup (-0.25,0.5) -- (-0.25,0.8);
            \draw[over,<-] (0.25,-0.8) -- (0.25,-0.5) \braidup (-0.25,0);
            \closeddot{-0.25,-0.55};
            \closeddot{-0.25,0};
        \end{tikzpicture}
        =
        \begin{tikzpicture}[centerzero,present]
            \draw (-0.25,0) \braidup (0.25,0.5) -- (0.25,0.8);
            \draw[over,->] (-0.25,-0.8) -- (-0.25,-0.5) \braidup (0.25,0) \braidup (-0.25,0.5) -- (-0.25,0.8);
            \draw[over,<-] (0.25,-0.8) -- (0.25,-0.5) \braidup (-0.25,0);
            \closeddot{-0.25,0.55};
            \closeddot{-0.25,0};
        \end{tikzpicture}
        \ .
    \end{equation}
    Thus,
    \[
        \begin{tikzpicture}[centerzero,present]
            \draw (-0.25,0) \braidup (0.25,0.5) -- (0.25,0.8);
            \draw[over,<-] (-0.25,-0.8) -- (-0.25,-0.5) \braidup (0.25,0) \braidup (-0.25,0.5) -- (-0.25,0.8);
            \draw[over,<-] (0.25,-0.8) -- (0.25,-0.5) \braidup (-0.25,0);
            \posdot{-0.25,-0.55};
            \posdot{-0.25,0};
        \end{tikzpicture}
        \ceqref{tornado} t^{-1}\ 
        \begin{tikzpicture}[centerzero,present]
            \draw (0,0) \braidup (1.2,1) -- (1.2,1.4);
            \draw (0.8,-0.8) to[out=down,in=down,looseness=1.2] (0,-0.8) -- (0,-0.6);
            \draw[over] (0.8,1.4) -- (0.8,1) \braiddown (1.2,0) \braiddown (0.8,-0.8);
            \draw[over,->] (0,-0.6) \braidup (0.4,0) \braidup (0,0.6) -- (0,0.75) arc(180:0:0.2) \braiddown (0.8,0) \braiddown (0.4,-0.75) -- (0.4,-1.6);
            \draw[over,<-] (1.2,-1.6) -- (1.2,-1) \braidup (0,0);
            \closeddot{0,-0.6};
            \closeddot{0,0};
        \end{tikzpicture}
        \ceqref{lemon2} t^{-1}\ 
        \begin{tikzpicture}[centerzero,present]
            \draw (0,0) \braidup (1.2,1) -- (1.2,1.4);
            \draw (0.8,-0.8) to[out=down,in=down,looseness=1.2] (0,-0.8) -- (0,-0.6);
            \draw[over] (0.8,1.4) -- (0.8,1) \braiddown (1.2,0) \braiddown (0.8,-0.8);
            \draw[over,->] (0,-0.6) \braidup (0.4,0) \braidup (0,0.6) -- (0,0.75) arc(180:0:0.2) \braiddown (0.8,0) \braiddown (0.4,-0.75) -- (0.4,-1.6);
            \draw[over,<-] (1.2,-1.6) -- (1.2,-1) \braidup (0,0);
            \closeddot{0,0.6};
            \closeddot{0,0};
        \end{tikzpicture}
        \ceqref{tornado}
        \begin{tikzpicture}[centerzero,present]
            \draw (-0.25,0) \braidup (0.25,0.5) -- (0.25,0.8);
            \draw[over,<-] (-0.25,-0.8) -- (-0.25,-0.5) \braidup (0.25,0) \braidup (-0.25,0.5) -- (-0.25,0.8);
            \draw[over,<-] (0.25,-0.8) -- (0.25,-0.5) \braidup (-0.25,0);
            \posdot{-0.25,0.5};
            \posdot{-0.25,0};
        \end{tikzpicture}
        \ .
    \]
    Thus the cyan versions of all the relations in \cref{alarm} hold.
    
    Conversely, sending the generating dot $\upopendot[present]$ of $\Dcat$ to $\upopendot$ defines a right $\OScat(z,t)$-module functor $\bG \colon \Dcat \to \AOScat(z,t)$, since \cref{AOSreflect} holds in $\AOScat(z,t)$. The composite $\bF\bG$ fixes the generating dot of $\Dcat$ and all morphisms of $\OScat(z,t)$, so $\bF\bG=\id_{\Dcat}$.

    To check the other composite, note that the formula for the downward positive dot in \cref{tornado} also holds in $\AOScat(z,t)$, with all diagrams drawn in black:
    \[
        t^{-1}
        \begin{tikzpicture}[centerzero]
            \draw (-0.4,0) to[out=down,in=180] (-0.25,-0.15) to[out=0,in=down] (0,0.4);
            \draw[over,<-] (0,-0.4) to[out=up,in=0] (-0.25,0.15) to[out=180,in=up] (-0.4,0);
            \closeddot{-0.4,0};
        \end{tikzpicture}
        \ceqsref{cranberry}{lime} t^{-1}
        \begin{tikzpicture}[centerzero]
            \draw[<-] (0,-0.4) to[out=up,in=0] (-0.25,0.15) to[out=180,in=up] (-0.4,0);
            \draw[over] (-0.4,0) to[out=down,in=180] (-0.25,-0.15) to[out=0,in=down] (0,0.4);
            \closeddot{0,0.2};
        \end{tikzpicture}
        \ceqref{ocurlbub}
        \begin{tikzpicture}[centerzero]
            \draw[<-] (0,-0.4) -- (0,0.4);
            \closeddot{0,0};
        \end{tikzpicture}
        \ceqref{cranberry}
        \begin{tikzpicture}[centerzero]
            \draw[<-] (0,-0.4) -- (0,0.4);
            \posdot{0,0};
        \end{tikzpicture}
        \ .
    \]
    Since $\bG\bF$ fixes the upward dot, its inverse, and all morphisms of $\OScat(z,t)$, this formula shows that it also fixes the downward positive dot. Thus $\bG\bF$ fixes both module generators in the presentation supplied by \cref{swamp+}, and hence $\bG\bF=\id_{\AOScat(z,t)}$. Therefore $\bF$ and $\bG$ are mutually inverse.
        
    The alternate presentation is obtained by taking inverses of both sides of \cref{AOSreflect}.
\end{proof}

\subsection{The twisted affine oriented skein category\label{SS:tAOScat}}

Let $\sigma$ be the strict braided monoidal involution of $\OScat(z,t)$ that reverses the orientation of every strand.  This is the involution denoted $\Theta$ in \cite[(2.12)]{SSS26}.  We first describe $\AOScat_\sigma(z,t) := \Aff_\sigma(\OScat(z,t))$ as a monoidal category in the spirit of \cref{lion+}. 

\begin{prop} \label{tAOSmonoidal}
    The category $\AOScat_\sigma(z,t)$ is isomorphic to the $\kk$-linear strict monoidal category obtained from $\OScat(z,t)$ by adjoining an invertible morphism
    \[
        \togupdown \colon \downobj \to \upobj,
    \]
    whose inverse we denote $\togdownup$, subject to the relation
    \begin{equation}\label{togcross1}

        \ .
    \end{equation}
    Alternatively, it is generated by an invertible morphism
    \[
        \togdownup \colon \upobj \to \downobj,
    \]
    whose inverse we denote $\togupdown$, subject to the relation
    \[
        %
        \ .
    \]
\end{prop}

\begin{proof}
    Let $\Dcat$ denote the category presented as in the statement of the proposition.  We draw its string diagrams in cyan to distinguish them from diagrams in $\AOScat_\sigma(z,t)$.  We define $\togdownup[present] := \left( \togupdown[present] \right)^{-1} \colon \upobj[cyan] \to \downobj[cyan]$.  First note that
    \begin{equation}\label{plug}
        %
        \ .
    \end{equation}
    The first and last diagrams are the left and right duals, respectively, of the generator $\togupdown$.  Hence these two duals coincide.  Taking inverses gives the same equality for the two duals of $\togdownup$.
    
    By \cref{lion+}, $\AOScat_\sigma(z,t)$ is generated by invertible morphisms
    \begin{equation} \label{tAOSdots}
        %
        \colon \upobj \to \downobj,
    \end{equation}
    whose inverses we denote by
    \begin{equation}
        %
        \right)^{-1}
        \colon \downobj \to \upobj,
    \end{equation}
    subject to the relations \cref{lion+rels}, where $S_0 = \{\upobj,\downobj\}$ and $S_1' = \{\rightcap, \rightcup, \leftcap, \leftcup\}$.  Define
    \begin{equation}\label{sherbrooke}
        %
        \ .
    \end{equation}
    
    We claim that
    \begin{equation}\label{beatz}
        %
        \ ,
    \end{gather*}
    proving the first and third relations in \cref{beatz}.  The proofs of the second and fourth are similar, using the analogue of \cref{plug} for $\togdownup$.

    We claim that there is an isomorphism of $\kk$-linear strict monoidal categories
    \[
        \AOScat_\sigma(z,t) \to \Dcat
    \]
    that sends every generator of $\OScat(z,t)$ to the corresponding generator in $\Dcat$, and maps the positive dots as follows:
    \[
        %
        \ .
    \]
    To prove this, we must show that the cyan versions of the relations \cref{lion+rels} hold, where $S_0 = \{\upobj[cyan],\downobj[cyan]\}$ and $S_1' = \{\rightcap[present], \rightcup[present], \leftcap[present], \leftcup[present]\}$.
    
    Sliding the left positive dot to the right using \cref{beatz}, then cancelling inverse dots, gives
    \begin{equation}\label{train}
        %
        \ ,\qquad
        X \in \{\upobj[cyan],\downobj[cyan]\}.
    \end{equation}
    Thus, the last relation in \cref{lion+rels} holds for $f \in \{\rightcap[present], \rightcup[present], \leftcap[present], \leftcup[present]\}$.
    
    For $X,Y \in \{\upobj[cyan],\downobj[cyan]\}$, let $S(X,Y)$ denote the first relation in \cref{lion+rels}.   Thus, \cref{togcross1} is $S(\upobj[cyan],\upobj[cyan])$.  We must recover the other three cases.  Assuming $S(X,Y)$, we have
    \begin{equation} \label{tAOSbend}
        %
        \ ,
    \]
    which is $S(X,Y^\sigma)$.

    Taking left duals of both sides of \cref{togcross1}, by bending the boundary strands through $180\degree$ using cups and caps, and applying \cref{beatz} to the dotted strand gives
    \[
        %
        \ .
    \]
    Composing on the bottom with
    \(
        %
    \)
    and on the top with
    \(
        %
    \)
    then gives $S(\downobj[cyan],\downobj[cyan])$.  The preceding bending calculation then gives $S(\downobj[cyan],\upobj[cyan])$.  Thus, the first relation in \cref{lion+rels} holds for all choices of $X,Y \in \{\upobj[cyan], \downobj[cyan]\}$, so the proposed assignment defines a $\kk$-linear strict monoidal functor.

    Conversely, sending
    \[
        %
    \]
    and fixing all morphisms of $\OScat(z,t)$ defines a $\kk$-linear strict monoidal functor in the opposite direction, since \cref{togcross1} holds in $\AOScat_\sigma(z,t)$ by \cref{lion+}.  Reversing the argument surrounding \cref{train} shows that taking $f = \leftcap$ in \cref{lion+rels}, together with cancellation of inverse dots, gives
    \[
        %
        \ .
    \]
    This is the same formula used in \cref{sherbrooke} to define
    \(
        %
    \).
    Hence both composites fix the generating dots and all morphisms of $\OScat(z,t)$, so the functors are mutually inverse.
    
    The alternate presentation is obtained by taking inverses of both sides of \cref{togcross1}.
\end{proof}

We now give a presentation of $\AOScat_\sigma(z,t)$ as a module category.

\begin{prop} \label{tAOSmodule}
    As a right $\OScat(z,t)$-module, $\AOScat_\sigma(z,t)$ is generated by the invertible morphism
    \[
        \togupdown \colon \downobj \to \upobj,
    \]
    whose inverse we denote $\togdownup$, subject to the relation
    \begin{equation} \label{tAOSreflect+}
        %
        \ .
    \end{equation}
    Alternatively, it is generated by the invertible morphism
    \[
        \togdownup \colon \upobj \to \downobj,
    \]
    whose inverse we denote $\togupdown$, subject to the relation
    \[
        %
        \ .
    \]
\end{prop}

\begin{proof}
    Let $\Dcat$ denote the category presented as in the statement of the proposition.  We draw its string diagrams in cyan to distinguish them from diagrams in $\AOScat_\sigma(z,t)$.
    
    Define
    \begin{equation} \label{tAOScurls}
        %
        \ .
    \end{equation}
    An argument analogous to the one surrounding \cref{cato1,cato2} shows that
    \[
        %
        \right)^{-1}.
    \]
    
    By \cref{swamp+}, $\AOScat_\sigma(z,t)$ is generated by the two invertible positive dots \cref{tAOSdots}, subject to the relations \cref{life} for $X,Y \in \{\upobj,\downobj\}$ and $f \in \{\rightcap, \rightcup, \leftcap, \leftcup\}$.  We claim that there is an isomorphism of $\OScat(z,t)$-module categories
    \[
        \AOScat_\sigma(z,t) \to \Dcat,\qquad
        %
        \ .
    \]
    To prove this, we must show that the cyan versions of the relations \cref{life} hold for $X,Y \in \{\upobj[cyan],\downobj[cyan]\}$ and $f \in \{\rightcap[present], \rightcup[present], \leftcap[present], \leftcup[present]\}$.  The argument is analogous to that used in the proof of \cref{AOSpresent} and so we omit the details here.
    
    The alternate presentation is obtained by taking inverses of both sides of \cref{tAOSreflect+}.
\end{proof}

\subsection{The disoriented skein category\label{SS:DScat}}

Now fix $q \in \kk^\times$ such that $z = q-q^{-1}$.  The \emph{disoriented skein category} of \cite[Def.~2.1]{SSS26}, denoted $\DScat(q,t)$, is the right $\OScat(q-q^{-1},t)$-module generated by mutually inverse \emph{toggles}
\[
    \togupdown \colon \downobj \to \upobj, \qquad
    \togdownup \colon \upobj \to \downobj,
\]
subject to the relation \cref{tAOSreflect+} and the relations
\begin{equation}\label{DSold}
    \begin{tikzpicture}[anchorbase]
        \draw[<-] (-0.2,0.2) to[out=down,in=135] (0.2,-0.3);
        \draw[over,->] (-0.2,-0.3) to[out=45,in=down] (0.2,0.2) arc(0:180:0.2);
        \opendot{-0.2,0.2};
    \end{tikzpicture}
    \ = q\
    \begin{tikzpicture}[anchorbase]
        \draw[->] (0.2,-0.3) -- (0.2,0.1) arc(0:180:0.2) -- (-0.2,0);
        \draw[<-] (-0.2,0) -- (-0.2,-0.3);
        \opendot{-0.2,0};
    \end{tikzpicture}
    \ ,\qquad
   \begin{tikzpicture}[anchorbase]
        \draw[<-] (-0.2,0.3) to[out=-45,in=up] (0.2,-0.2) arc(360:180:0.2);
        \draw[over,->] (-0.2,-0.2) to[out=up,in=-135] (0.2,0.3);
        \opendot{-0.2,-0.2};
    \end{tikzpicture}
    \ = q\
    \begin{tikzpicture}[anchorbase]
        \draw[<-] (0.2,0.3) -- (0.2,-0.1) arc(360:180:0.2) -- (-0.2,0);
        \draw[->] (-0.2,0) -- (-0.2,0.3);
        \opendot{-0.2,0};
    \end{tikzpicture}
    \ .
\end{equation}

\begin{theo} \label{DStheorem}
    As $\OScat(q-q^{-1},t)$-module categories, $\DScat(q,t)$ is isomorphic to the quotient of $\AOScat_\sigma(q-q^{-1},t)$ by the relation
    \begin{equation} \label{gandalf}
        \begin{tikzpicture}[anchorbase]
            \draw[<->] (0,-0.35) -- (0,0.35);
            \negdot{0,0};
        \end{tikzpicture}
        = q^{-1} t\
        \begin{tikzpicture}[anchorbase]
            \draw[<->] (0,-0.35) -- (0,0.35);
            \posdot{0,0};
        \end{tikzpicture}
        \ .
    \end{equation}
    Equivalently, in the presentation of $\AOScat_\sigma(q-q^{-1},t)$ given by \cref{tAOSmodule}, it is the quotient of $\AOScat_\sigma(q-q^{-1},t)$ by either one of the two relations in \cref{DSold}.
\end{theo}

\begin{proof}
    This follows from the fact that
    \[
        \begin{tikzpicture}[anchorbase]
            \draw[<-] (-0.2,0.3) to[out=-45,in=up] (0.2,-0.2) arc(360:180:0.2);
            \draw[over,->] (-0.2,-0.2) to[out=up,in=-135] (0.2,0.3);
            \opendot{-0.2,-0.2};
        \end{tikzpicture}
        \ = q\
        \begin{tikzpicture}[anchorbase]
            \draw[<-] (0.2,0.3) -- (0.2,-0.1) arc(360:180:0.2) -- (-0.2,0);
            \draw[->] (-0.2,0) -- (-0.2,0.3);
            \opendot{-0.2,0};
        \end{tikzpicture}
        \Longleftrightarrow
        \begin{tikzpicture}[anchorbase]
            \draw[->] (-0.2,-0.2) to[out=up,in=-135] (0.2,0.3);
            \draw[over,<-] (-0.2,0.3) to[out=-45,in=up] (0.2,-0.2) arc(360:180:0.2);
            \opendot{-0.2,-0.2};
        \end{tikzpicture}
        \ = q^{-1}\
        \begin{tikzpicture}[anchorbase]
            \draw[<-] (0.2,0.3) -- (0.2,-0.1) arc(360:180:0.2) -- (-0.2,0);
            \draw[->] (-0.2,0) -- (-0.2,0.3);
            \opendot{-0.2,0};
        \end{tikzpicture}
        \Longleftrightarrow
        \begin{tikzpicture}[centerzero]
            \draw[<-] (0,-0.4) to[out=up,in=0] (-0.25,0.15) to[out=180,in=up] (-0.4,0);
            \draw[over,->] (-0.4,0) to[out=down,in=180] (-0.25,-0.15) to[out=0,in=down] (0,0.4);
            \opendot{-0.4,0};
        \end{tikzpicture}
        \ = q^{-1}
        \begin{tikzpicture}[centerzero]
            \draw[<->] (0,-0.4) -- (0,0.4);
            \opendot{0,0};
        \end{tikzpicture}
        \overset{\cref{tAOScurls}}{\Longleftrightarrow} t^{-1}\
        \begin{tikzpicture}[anchorbase]
            \draw[<->] (0,-0.35) -- (0,0.35);
            \negdot{0,0};
        \end{tikzpicture}
        = q^{-1}\
        \begin{tikzpicture}[anchorbase]
            \draw[<->] (0,-0.35) -- (0,0.35);
            \posdot{0,0};
        \end{tikzpicture}
        \ ,
    \]
    and
    \[
        \begin{tikzpicture}[anchorbase]
            \draw[<-] (-0.2,0.2) to[out=down,in=135] (0.2,-0.3);
            \draw[over,->] (-0.2,-0.3) to[out=45,in=down] (0.2,0.2) arc(0:180:0.2);
            \opendot{-0.2,0.2};
        \end{tikzpicture}
        \ = q\
        \begin{tikzpicture}[anchorbase]
            \draw[->] (0.2,-0.3) -- (0.2,0.1) arc(0:180:0.2) -- (-0.2,0);
            \draw[<-] (-0.2,0) -- (-0.2,-0.3);
            \opendot{-0.2,0};
        \end{tikzpicture}
        \Longleftrightarrow
        \begin{tikzpicture}[centerzero]
            \draw[<-] (-0.4,0) to[out=down,in=180] (-0.25,-0.15) to[out=0,in=down] (0,0.4);
            \draw[over,->] (0,-0.4) to[out=up,in=0] (-0.25,0.15) to[out=180,in=up] (-0.4,0);
            \opendot{-0.4,0};
        \end{tikzpicture}
        \ = q\
        \begin{tikzpicture}[centerzero]
            \draw[->] (0,-0.4) -- (0,0);
            \draw[<-] (0,0) -- (0,0.4);
            \opendot{0,0};
        \end{tikzpicture}
        \overset{\cref{tAOScurls}}{\Longleftrightarrow}
        t\ 
        \begin{tikzpicture}[anchorbase]
            \draw[-<-=0.55] (0,0) -- (0,-0.35);
            \draw[-<-=0.55] (0,0) -- (0,0.35);
            \posdot{0,0};
        \end{tikzpicture}
        =
        q\
        \begin{tikzpicture}[anchorbase]
            \draw[-<-=0.55] (0,0) -- (0,-0.35);
            \draw[-<-=0.55] (0,0) -- (0,0.35);
            \negdot{0,0};
        \end{tikzpicture}
        \ ,
    \]
    which is obtained from \cref{gandalf} by taking inverses of both sides.
\end{proof}

\begin{rem}
    The normalization \cref{gandalf} agrees with the tensor-pair cylinder $\Theta$-twist constructed in \cite[\S2.5]{SSSsub26}.  \Cref{DStheorem} shows that the presentation of $\DScat(q,t)$ given in \cite[Def.~2.1]{SSS26} has redundant relations.  One need only impose one of the two relations in \cref{DSold}; the other then holds automatically.
\end{rem}

\begin{rem}
    The quotient in \cref{DStheorem} is proper over an integral domain whenever $q+q^{-1} \ne 0$.  Indeed, applying \cref{DStheorem} with $q$ replaced by $-q^{-1}$ gives another quotient of the same twisted affinization, in which
    \[
        \begin{tikzpicture}[centerzero]
            \draw[->-=0.7] (0,-0.6) -- (0,-0.2);
            \draw[-<-=0.5] (0,-0.2) -- (0,0.2);
            \draw[->] (0,0.2) -- (0,0.6);
            \posdot{0,-0.2};
            \posdot{0,0.2};
        \end{tikzpicture}
        \ = -q^{-1}t^{-1}\
        \begin{tikzpicture}[centerzero]
            \draw[->] (0,-0.6) -- (0,0.6);
        \end{tikzpicture}
        \ .
    \]
    By the basis theorem \cite[Thm.~7.1]{SSS26}, the identity $1_{\upobj}$ spans a free rank-one endomorphism space in $\DScat(-q^{-1},t)$.  Hence
    \[
        \begin{tikzpicture}[centerzero]
            \draw[->-=0.7] (0,-0.6) -- (0,-0.2);
            \draw[-<-=0.5] (0,-0.2) -- (0,0.2);
            \draw[->] (0,0.2) -- (0,0.6);
            \posdot{0,-0.2};
            \posdot{0,0.2};
        \end{tikzpicture}
        \ \ne qt^{-1}\
        \begin{tikzpicture}[centerzero]
            \draw[->] (0,-0.6) -- (0,0.6);
        \end{tikzpicture}
    \]
    in the twisted affinization.
\end{rem}


\bibliographystyle{alphaurl}
\bibliography{ABMC}

\end{document}